\documentclass[a4paper,12pt,reqno]{amsart}

\usepackage{hyperref}
\usepackage{comment}
\usepackage{bm,epsfig,xcolor,graphicx}

  \newtheorem{thm}{Theorem} 
  \newtheorem{prop}{Proposition}[section]
  \newtheorem{lemma}{Lemma}[section]

\theoremstyle{definition}
  \newtheorem{definition}{Definition}
  \newtheorem{rem}{Remark}
  \newtheorem{ack}{Acknowledgements}

\newcommand{\ad}{\operatorname{ad}}
\newcommand{\Image}{\operatorname{Im}}
\newcommand{\BSL}{\operatorname{BSL}}
\newcommand{\SL}{\operatorname{SL}}
\newcommand{\Ad}{\operatorname{Ad}}
\newcommand{\trace}{\operatorname{tr}_q}
\newcommand{\bx}{\boldsymbol{x}}
\newcommand{\by}{\boldsymbol{y}}

\allowdisplaybreaks

\begin{document}
\baselineskip17pt
\title[Quantum fundamental group and its $\mathrm{SL}_{2}$ representation]{Quantum fundamental group of knot and its $\mathrm{SL}_{2}$ representation}
\author[Jun Murakami]{Jun Murakami}
\address{Department of Mathematics, 
Faculty of Science and Engineering,
Waseda University,
3-4-1 Ohkubo, Shinjuku-ku, Tokyo, 169-8555, Japan}
\email{murakami@waseda.jp}

\author[Roland van der Veen]{Roland van der Veen}
\address{University of Groningen, Bernoulli Institute, P.O. Box 407, 9700 AK Groningen, The Netherlands}
\email{roland.mathematics@gmail.com}

\begin{abstract}
The theory of bottom tangles is used to construct a quantum fundamental group.  
On the other hand, the skein module  is considered as a quantum analogue of the $\mathrm{SL}(2)$ representation of the fundamental group.   
Here we construct the skein module of a knot complement by using the bottom tangles.  
We first construct the universal space of quantum representations, which is a quantum analogue of  the fundamental group, 
and then factor it by the skein relation to get the skein module.  
We also investigate the action of the quantum torus to the boundary of complement, and derive the recurrence relation of the colored Jones polynomial, which is known as $A_q$ polynomial.  
\end{abstract}

\thanks{This work was supported by JSPS KAKENHI Grant Numbers JP20H01803, JP20K20881}
\maketitle

%

\section*{Introduction}
The aim of this paper is to give a concrete construction of the quantum fundamental group proposed by Habiro in \cite{H2} for knot complements.  A similar construction was done for braided Hopf algebras instead of bottom tangles in \cite{MV}, but the construction here is more universal.  
We also get the skein module of the complement by imposing the skein relation to the quantum fundamental group. 
The resulting skein module is represented by the stated skein module introduced in \cite{Le2}.    
\par
Knots and links are isotopic to plat closures of braids with an even number of strings as in Figure \ref{fig:plat}.
Plat closures of two braids $b_1$ and $b_2$ are isotopic if and only if there is a sequence of plat Markov moves introduced in \cite{CG} which transforms $b_1$ to $b_2$.  
\par
\begin{figure}[htb]
\[
\begin{matrix}
\includegraphics[scale=1]{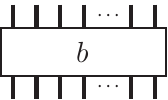}
\end{matrix}
\ \ 
\longrightarrow
\ \ 
\hat b = 
\begin{matrix}
\includegraphics[scale=1]{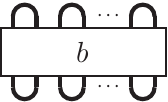}
\end{matrix}
\]
\caption{Plat closure of $b \in B_{2k}$.}
\label{fig:plat}
\end{figure}
%
Let $K$ be a knot or a link and $b \in B_{2k}$ be a braid whose plat closure $\widehat b$ is isotopic to $K$. 
Then we get a presentation of the fundamental group of 
the complement of $K$ in terms of $b$ as follows.  
Let $D_{k}$ be the punctured disk with $k$ punctures and a base point $p$ on the boundary.  
Then the fundamental group $\pi_1(D_k, p)$ is isomorphic to the free group $F_k$ with $k$ generators $x_1$, $\cdots$, $x_k$ where $x_i$ is the loop enclosing the $i$-th puncture of $D_k$ as in Figure \ref{fig:generators}.  
Similarly, $\pi_1(D_{2k}, q)$  is the free group generated by $2k$ generators $y_1$, $\cdots$, $y_{2k}$.  
\begin{figure}[htb]
\[
\begin{matrix}
x_i & 
{\begin{matrix}
\includegraphics[scale=1]{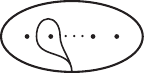}_{\text{\normalsize $D_k$}}
\\
p\ \ \ \ 
\end{matrix}}
& \qquad & 
y_i & 
{\begin{matrix}
{\includegraphics[scale=1]{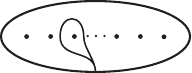}}_{\text{\normalsize $D_{2k}$}}
\\ q\ \ \ \ 
\end{matrix}}
\end{matrix}
\]
\caption{Generators $x_i$ and $y_i$ of $\pi_1(D_k, p)$ and $\pi_1(D_{2k}, p)$.}
\label{fig:generators}
\end{figure}
The standard generators  $\sigma_1$, $\cdots$, $\sigma_{2k-1}$  of $B_{2k}$ act on
 $y \in \pi_1(D_{2k}, p)$ 
by pushing down the disk $D_{2k}$ along the  generator.
In other words,    
 each generator $\sigma_i$  acts on $D_{2k}$ by swapping the $i$-th and $(i+1)$-th punctures by a counter-clockwise rotation as in Figure \ref{fig:generatoraction}. 
 \begin{figure}[htb]
 \[
 \begin{matrix}
 \includegraphics[scale=1]{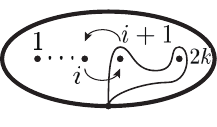}
 \end{matrix}
 \ 
 \longrightarrow
 \ 
 \begin{matrix}
 \includegraphics[scale=1]{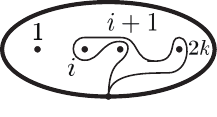}
 \end{matrix}
 \]
 \caption{The action of the braid generator $\sigma_i$ on $\pi_1(D_{2k})$.}
 \label{fig:generatoraction}
 \end{figure}
This defines an action of $b \in B_{2k}$ on $y \in\pi_1(D_{2k}, p)$, and 
the result $b\cdot y$ is obtained by pushing down $y$ along $b$ as in Figure \ref{fig:push}.  
\begin{figure}[htb]
\[
\begin{matrix}
\begin{matrix}
\begin{matrix}
\includegraphics[scale=0.8]{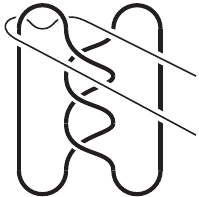}
\end{matrix}
&
\begin{matrix}
\includegraphics[scale=0.8]{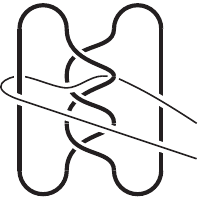}
\end{matrix}
&
\begin{matrix}
\includegraphics[scale=0.8]{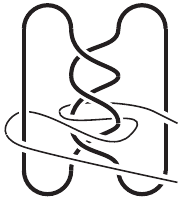}
\end{matrix}
&
\begin{matrix}
\includegraphics[scale=0.8]{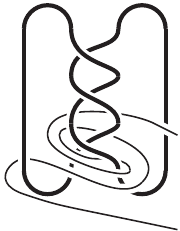}
\end{matrix}
&
\begin{matrix}
\includegraphics[scale=0.8]{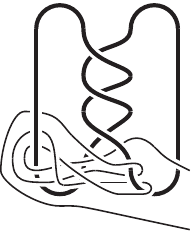}
\end{matrix}
\\
\begin{matrix}
\includegraphics[scale=0.9]{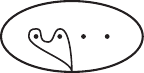}
\end{matrix}
&
\begin{matrix}
\includegraphics[scale=0.9]{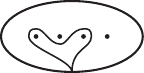}
\end{matrix}
&
\begin{matrix}
\includegraphics[scale=0.9]{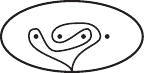}
\end{matrix}
&
\begin{matrix}
\includegraphics[scale=0.9]{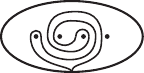}
\end{matrix}
&
\begin{matrix}
\includegraphics[scale=0.9]{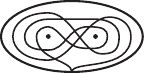}
\end{matrix}
\end{matrix}
\\
\qquad
\delta(x_1)=y_1y_2
\overset{\sigma_2}{\longrightarrow}
y_1y_3
\overset{\sigma_2}{\longrightarrow}
y_1y_3^{-1}y_2y_3
\overset{\sigma_2}{\longrightarrow}
y_1y_3^{-1}y_2^{-1}y_3 y_2 y_3
\overset{\varpi}{\longrightarrow}
x_1x_2^{-1}x_1x_2 x_1^{-1} x_2 = r_1
\end{matrix}
\]
\caption{Pushing a trivial element of $\pi_1(S^3\setminus K, p)$ along the plat presentation of the trefoil knot to get a relation of $\pi_1(S^3\setminus K, p)$.}
\label{fig:push}
\end{figure}
\par
Let $\delta$ be the homomorphism from $\pi_1(D_k, p)$ to $\pi_1(D_{2k}, q)$ and $\varpi$ be the  homomorphism from $\pi_1(D_{2k}, q)$ to $\pi_1(D_k, p)$ given by
 \[
 \delta(x_i)=y_{2i-1}y_{2i},
 \qquad
\begin{cases}\varpi(y_{2i-1})=x_i,
\\
\varpi(y_{2i})=x_i^{-1}.
\end{cases}
\quad(1 \leq i \leq k)  
\]
Then $\varpi\circ \delta(x) = 1$ for  $x \in \pi_1(D_k, p)$.  
Let $\iota_1$ and $\iota_2$ be the inclusions of $D_k$ into $D_{2k}$ and $D_{2k}$ into $S^3\setminus K$ respectively as in Figure \ref{fig:inclusion},  
$\iota = \iota_2\circ \iota_1$ and
$\iota_{1*}$, 
$\iota_*$ be the induced homomorphisms of $\iota_{1}$, $\iota$ from $\pi_1(D_k, p)$ to $\pi_1(D_{2k}, p)$, $\pi_1(D_k, p)$ to $\pi_1(S^3\setminus K, p)$ respectively.  
\begin{figure}[htb]
\[
\begin{matrix}
\includegraphics[scale=1]{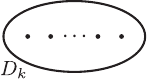}
\end{matrix}
\overset{\iota_1}{\longrightarrow}
\begin{matrix}
\includegraphics[scale=1]{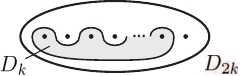}
\end{matrix}
\]
\[
\begin{matrix}
\includegraphics[scale=0.8]{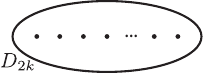}
\end{matrix}
\overset{\iota_2}{\longrightarrow}
\begin{matrix}
\includegraphics[scale=1]{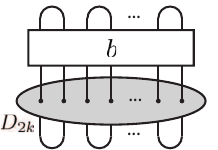}
\end{matrix}
\]
\caption{Inclusions $\iota_1 : D_k \to D_{2k}$ and $\iota_2 : D_{2k} \to S^3 \setminus K$ where $K = \widehat b$.}
\label{fig:inclusion}
\end{figure}
%
Let 
$r_i =\varpi(b\cdot \delta(x_i))$,
then we have $\iota_*(r_i) = 1$ and it is known that 
\[
\pi_1(S^3\setminus K) \cong 
F_k/(r_1, \cdots, r_k),
\]
which is equivalent to the Wirtinger representation of $\pi_1(S^3\setminus K)$.
Moreover, we have
\begin{equation}
\iota_*\Big(\varpi\big((b\cdot\delta(x)) \iota_{1*}(y)\big)\Big) = \iota_*(y).   \quad 
(x, y \in \pi_1(D_k, p))
\label{eq:brelation}
\end{equation}
This relation is equivalent to the relations $r_1=1$, $\cdots$, $r_k=1$. 
However, in quantized case, this relation does not come from $r_1=1$, $\cdots$, $r_k=1$, and we use the relation \eqref{eq:brelation} instead of 
$r_1=1$, $\cdots$, $r_k=1$ to define the quantum fundamental group.  
\begin{figure}[htb]
\[
\begin{matrix}
\qquad\qquad
\begin{matrix}
\delta(x)
\\[30pt]
\iota_{1*}(y)
\end{matrix}
\begin{matrix}
\includegraphics[scale=0.8]{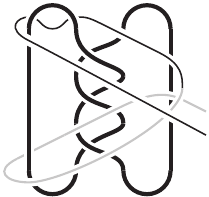}
\end{matrix}
&&
\qquad
\begin{matrix}
{}\\[30pt]
b \cdot\delta(x)
\\
\iota_{1*}(y)
\end{matrix}
\begin{matrix}
\includegraphics[scale=0.8]{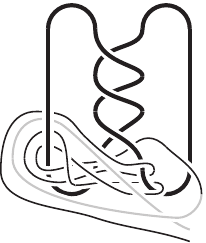}
\end{matrix}
\\
\delta(x) \iota_{1*}(y)
\qquad
\begin{matrix}
\includegraphics[scale=0.9]{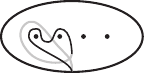}
\end{matrix}
&&
b \cdot\delta(x) \iota_{1*}(y)
\qquad
\begin{matrix}
\includegraphics[scale=0.9]{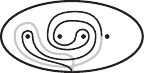}
\end{matrix}
\end{matrix}
\]
\caption{Two elements $\delta(x) \iota_{1*}(y)$ and $(b\cdot\delta(x)) \iota_{1*}(y)$ in $\pi_1(D_{2k}, p)$ which are equal in $\pi_1(S^3\setminus K, p)$.}
\label{fig:newrelation}
\end{figure}
\par
The purpose of this paper is to construct quantum analogue of $\pi_1(S^3\setminus K)$ and get the skein module from it.  
The quantum analogue of $\pi_1(S^3\setminus K)$ is obtained by replacing  $\pi_1(D_k, p)$ with linear combinations of bottom tangles.  
In Figure \ref{fig:push}, the path in $S^3 \setminus \hat b$ is projected to $D_{2k}$.  
The bottom tangle is obtained by replacing the crossings of the path as in Figure \ref{fig:crossing} by considering the height of the points of the path.  In other words, we replace homotopy classes of paths by isotopy classes of paths.
\begin{figure}[htb]
\[
\begin{matrix}
\epsfig{file=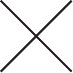, scale=1}
\end{matrix}
\ \longrightarrow\ 
\begin{matrix}
\epsfig{file=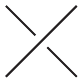, scale=1}
\end{matrix}
\ \text{or}\ 
\begin{matrix}
\epsfig{file=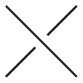, scale=1}
\end{matrix}
\]
\caption{A quantization of paths.}
\label{fig:crossing}
\end{figure}
For example, the last diagram in Figure \ref{fig:push} is modified as in Figure \ref{fig:bottomhopf}.  
Next we apply the Kauffman bracket skein theory to bottom tangles to obtain a quantum analogue of the $\SL(2)$ representation space.  
\begin{figure}[hbt]
\[
\begin{matrix}
\begin{matrix}
\includegraphics[scale=1.2]{hopfdisk6}
\end{matrix}
&
\qquad
\longrightarrow
&
\begin{matrix}
\includegraphics[scale=1.2]{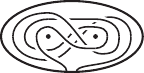}
\end{matrix}
\\
\text{path in $D_2$}
&&
\text{corresponding bottom tangle on $D_2$}
\end{matrix}
\]
\caption{Bottom tangle version of the path in $D_k$.  
A twist is added to adjust the framing.}
\label{fig:bottomhopf}
\end{figure}
\par
We first expose the theory of bottom tangles developed by Habiro in \cite{H} as a quantization of $\pi_1(D_{k}, p)$ which is isomorphic to the free group with $k$ generators.  
Let $\mathbb{C}\pi_1(D_k, p)$ be the group algebra which is the linear space spanned by formal linear combinations of elements of $\pi_1(D_k, p)$.  
Then $\mathbb{C}\pi_1(D_k, p)$ has  a Hopf algebra structure.  
Let $\mathcal{T}_k$ be the space of linear combinations of bottom tangles with $k$ punctures, then 
$\mathcal{T}_k$ has a braided Hopf algebra, which is a generalization of the Hopf algebra structure of $\mathbb{C}\pi_1(D_k, p)$.  
\par
Let $K$ be a knot or link given by the plat closure of  $b\in B_{2k}$.  
We assign a bottom tangle $\widehat{T}_b$ from $b$ representing the plat closure of $b$ with twined paths and $k$ extra paths as in Figure \ref{fig:tbxy}.  
Here a  Hopf diagram is used to explain a bottom tangle, the blank circle is the bottom tangle representing the antipode,  and  $\varepsilon$ is the bottom tangle representing the counit.  
See Section 2 for the details.  
\begin{figure}[htb]
\[
\widehat{T}_b : 
\begin{matrix}
\includegraphics[scale=0.8]{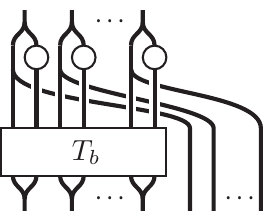}
\end{matrix}
\ ,
\qquad
\varpi\Big(\big(b\cdot\delta(x)\big) \iota_{1*}(y)\Big) : 
\begin{matrix}
\includegraphics[scale=0.8]{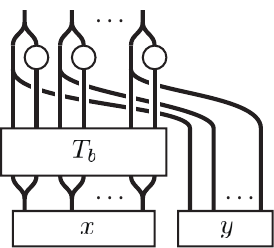}
\end{matrix}
\]
\caption{The bottom tangle $\widehat{T}_b$ to explain $\varpi\Big(\big(b\cdot\delta(x)\big) \iota_{1*}(y)\Big)$, where
$T_b$ is the bottom tangle representing the action of $b\in B_{2k}$.}
\label{fig:tbxy}
\end{figure}
Let $I_b$ be a submodule of $\mathcal{T}_k$ which is spanned by the set $\{\widehat{T}_b(x \otimes y) - \varepsilon(x) \otimes y: x, y \in \mathcal{T}_k\}$.      
Then, the quotient space $\mathcal{A}_b = \mathcal{T}_k/I_b$ is a quantum analogue of $\mathbb{C}\pi_1(S^3\setminus K)$, which is given by a quotient of  $\mathbb{C}F_k$ divided by the ideal generated by the relations.  
If $b$ and $b'$ represent isotopic knots, then  it is shown that $\mathcal{A}_b$ and $\mathcal{A}_{b'}$ is isomorphic in Theorem \ref{thm:main}, and we call  $\mathcal{A}_b$ {\it the quantum fundamental group} of $K$.    
To prove Theorem \ref{thm:main}, we use the Markov moves for plat presentations introduced in \cite{CG}.  
\par
To get a quantum analogue of the space of $\SL(2)$ representations, the Kauffman bracket skein relation is imposed upon the bottom tangles. 
By applying the skein relation to the bottom tangles, we get the skein algebra  $\mathcal{S}_k$ of the  punctured disk $D_k$ with a base line.  
Let $\widetilde{I}_b$ be the image of $I_b$ in $\mathcal{S}_k$,  
we get the quotient space  
$\widetilde{\mathcal{A}}_b =\mathcal{S}_k/\widetilde{I}_b$.  
By specializing the parameter for the quantization to the classical case, $\widetilde{I}_b$ becomes the ideal corresponding to the relations of the matrix elements of the generators of $\pi_1(S^3\setminus K)$.  
So we call  $\widetilde{\mathcal{A}}_b$  {\it the space of quantum $\SL(2)$ representations} of $K$.  
Moreover, the adjoint coinvariant part of  $\widetilde{\mathcal{A}}_b$ is isomorphic to the skein module of the complement of $\widehat{b}$ in \cite{B}, \cite{PS}, which descends to the ring defining the classical character variety.  
\par
It is known that the recurrence relation of the colored Jones polynomial is obtained by looking at the kernel of
 the action of the quantum torus corresponding to the boundary of complement.  
 This kernel is called the peripheral ideal in \cite{GS}.  
We explain how the quantum torus acts on $\widetilde{\mathcal{A}}_b$.
By using this action, we obtain the recurrence relations of Jones polynomials for some simple examples, which is called the $A_q$-polynomials.   
\par
The paper is organized as follows. 
In Section 1, we review the theory of bottom tangles and its rich algebraic structures.  
We also investigate the braid group action on the bottom tangles.  
In Section 2, we construct the quantum fundamental group in terms of bottom tangles by using a plat presentation of a knot, and show that this space is an invariant of the knot.  
In Section 3, the Kauffman bracket skein relation is imposed to bottom tangles and we get the skein algebra of a punctured disk.  
In Section 4, we investigate the space of quantized $\SL(2)$ representations of $K$, which is the quotient of the quantum fundamental group by the skein relations.   
In Section 5, some simple examples are investigated. 
Some final remarks are given in Section 6.   
%
\begin{ack}
The authors would like to thank Julien Korinman for useful discussions about the stated skein algeras.  
\end{ack}
\section{Bottom tangles}
In this section, we extend the space of representations introduced in \cite{MV} to a universal one by using  Habiro's bottom tangles.  

\subsection{Bottom tangles}
The notion of the bottom tangles and their algebraic operations are introduced by Habiro  \cite{H}, and, in this paper, we introduce these algebraic operations in an opposite (dual) direction so that these operations fit with the stated skein algebra theory.  
\par
Let $\mathcal{R}$ be a commutative ring with a unit.  
Let $D_k$ be a closed disk with $k$ punctures $q_1$, $q_2$, $\cdots$, $q_k$ inside $D_k$ and let $p_0$ be a point in $\partial D_k$,  which is the base point of the ribbons we are going to introduce.  
\begin{definition}
Let $T = (c_1, \cdots, c_n)$ 
be a sequence of embeddings of 
$[0, 1]$ into $D_k \times [0,1]$ equipped with framings satisfying the following.  
\begin{enumerate}
\item The embedding $c_i$ is contained in $D_k\times(0,1)$ and  $c_1$, $\cdots$, $c_n$ do not have any intersection point.  
\item
Let $h_i$ be the composition of $c_i$ with the projection of $D_k\times[0,1]$ on the second component,  
then $h_1(1)< h_1(0)< h_2(1) < h_2(0) < \cdots < h_n(1)< h_n(0)$.  
\item
The end points of the $c_i$ are in $p_0\times[0,1]$ and the framings there coincide with the direction from $0$ to $1$ of $p_0\times [0, 1]$.  
An embedding $c_i$ with such a framing is called a {\it ribbon}.
\end{enumerate}
An isotopy of $D_k \times [0,1]$ fixing $ D_k\times \partial[0,1]$ and sending $p_0\times [0,1]$ to itself is called a {\it bottom tangle isotopy},
and the bottom isotopy class of $T$  is called a {\it bottom tangle}.    
Let $\mathcal{T}_{k,n}$ be the free $\mathcal{R}$-module  spanned by bottom tangles in $D_k$ with $n$ ribbons.  
\end{definition}
We will often encode an element of $\mathcal{T}_{k,n}$ by a tangle diagram on $(0,1)\times [0,1)-\bigcup_{j=1}^k (\frac{j}{k+1},\frac{1}{2})$.
We call the the line $(0,1)\times \{0\}$ the bottom line and we require all the end points of the tangle appear on the bottom line in decreasing height as one travels
left to right along the bottom line. The framing is given by the blackboard framing.  
Sometimes, the diagram is drawn as a rectangle with  handles as  the rightmost diagram of Figure \ref{fig:freearcs}.  
\begin{figure}[htb]
\[
\begin{matrix}
\begin{matrix}
1 \\[30pt]
0
\end{matrix}
\begin{matrix}
\epsfig{file=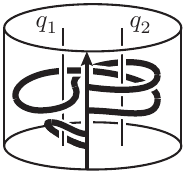, scale=0.8}
\end{matrix}
\\[-5pt]
\ \ p_0
\\[3pt]
D_2 \times[0,1]
\end{matrix}
\quad
\longrightarrow
\quad
\begin{matrix}
\\
\epsfig{file=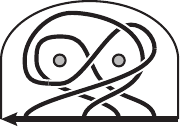, scale=0.8}
\\
1 \qquad\qquad\quad 0
\\[5pt]
\text{diagram on $D_2$}
\end{matrix}
\ \ \simeq \ \ 
\begin{matrix}
\epsfig{file=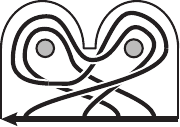, scale=0.8}
\\
1 \qquad\qquad\quad 0
\\[-10pt]{}\\
\end{matrix}
\]
\caption{Ribbons in
 $D_2\times[0,1]$.}
\label{fig:freearcs}
\end{figure}
\par
The tensor product structure $\otimes : \mathcal{T}_{k_1, n_1} \times \mathcal{T}_{k_2, n_2} \to \mathcal{T}_{k_1+k_2, n_1+n_2}$ is given by glueing two disks $D_{k_1}$ and $D_{k_2}$ as in Figure \ref{fig:glueing}.  
\begin{figure}[htb]
\[
\begin{matrix}
\text{$k_1$ punctures}
\\
\epsfig{file=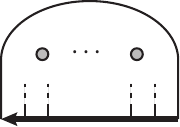, scale=0.8}
\\
\text{$n_1$ ribbons}
\end{matrix}
\ \ \otimes
\quad
\begin{matrix}
\text{$k_2$ punctures}
\\
\epsfig{file=freearcs3, scale=0.8}
\\
\text{$n_2$ ribbons}
\end{matrix}
\quad=\quad
\begin{matrix}
\text{$k_1+k_2$ punctures}
\\
\epsfig{file=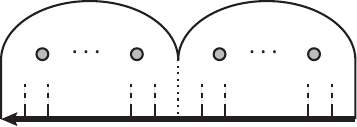, scale=0.8}
\\
\text{$n_1+n_2$ ribbons}
\end{matrix}
\]
\caption{The tensor product $\otimes : \mathcal{T}_{k_1, n_1} \times \mathcal{T}_{k_2, n_2} \to \mathcal{T}_{k_1+k_2, n_1+n_2}$.}
\label{fig:glueing}
\end{figure}
The multiplication 
$\boldsymbol{\mu} : \mathcal{T}_{k,n_1} \times \mathcal{T}_{k, n_2} \to \mathcal{T}_{k, n_1+n_2}$ 
is given by stacking $D_k\times [0, 1]$ as in Figure \ref{fig:multiplication}.  
%
\begin{figure}[htb]
\[
\begin{matrix}
\boldsymbol{\mu}\Big( &
\begin{matrix}
\epsfig{file=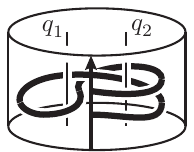, scale=0.8}
\end{matrix}
& , &
\begin{matrix}
\epsfig{file=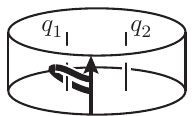, scale=0.8}
\end{matrix}
&
\Big)
&= &
\begin{matrix}
\epsfig{file=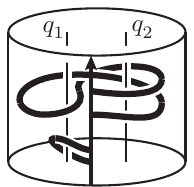, scale=0.8}
\end{matrix}
\\
& \downarrow & & \downarrow & & & \downarrow
\\[5pt]
\boldsymbol{\mu}\Big(
&\begin{matrix}
\epsfig{file=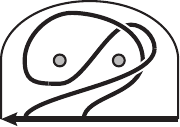, scale=0.8}
\end{matrix}
& , &
\begin{matrix}
\epsfig{file=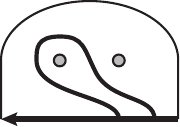, scale=0.8}
\end{matrix}
&\Big)
& = &
\begin{matrix}
\epsfig{file=freearcs2, scale=0.8}
\end{matrix}
\end{matrix}
\]
\caption{The multiplication 
$\boldsymbol{\mu} : \mathcal{T}_{k,n_1} \times \mathcal{T}_{k, n_2} \to \mathcal{T}_{k, n_1+n_2}$. }
\label{fig:multiplication}
\end{figure}
\par
For $T_1 \in \mathcal{T}_{k,l}$ and $T_2\in \mathcal{T}_{l,n}$, the composition $T_1\circ T_2 \in \mathcal{T}_{k, n}$ is defined by glueing the handles of $T_2$ to the ribbons of $T_1$ as in Figure \ref{fig:composition}.  
This composition gives an algebra structure in $\mathcal{T}_{k,k}$ and the action of $\mathcal{T}_{k,k}$ on $\mathcal{T}_{k, n}$ gives a $\mathcal{T}_{k,k}$ module structure on $\mathcal{T}_{k, n}$.  
\begin{figure}[htb]
\[
\begin{matrix}
\epsfig{file=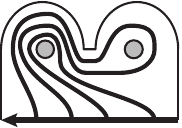, scale=0.8}
\\
T_1 \in \mathcal{T}_{2,2}
\end{matrix}
\ \circ \ 
\begin{matrix}
\epsfig{file=freearcs11, scale=0.8}
\\
T_2 \in \mathcal{T}_{2,2}
\end{matrix}
\ = \ 
\begin{matrix}
\epsfig{file=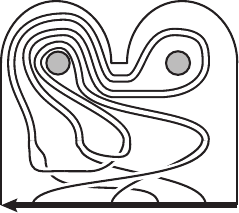, scale=0.8}
\\
T_1 \circ T_2
\end{matrix}
\]
\caption{The composition of a bottom tangle $T_1\in \mathcal{T}_{k,l}$ and an element $T_2 \in \mathcal{T}_{l,n}$ of the algebra of free ribbons in the case $k=n=l=2$.}
\label{fig:composition}
\end{figure}
\par
The bottom tangles form the following monoidal category $\mathcal{B}$.  
The set of objects of $\mathcal{B}$ is $\{0, 1, 2, \cdots\}$ and $\mathcal{T}_{k,n}$ is the set of morphisms from $n$ to $k$.  
The zero object is $0$ and the zero morphism is  the unique bottom tangle $\boldsymbol{1}$ in $\mathcal{T}_{0,0}$ whose diagram is a square without a ribbon.  
\subsection{Braided Hopf algebra structure of bottom tangles}
A braided Hopf algebra  structure is given to $\mathcal{B}$ as in Figure \ref{fig:Hopf}.  
\begin{figure}[htb]
\[
\begin{matrix}
\begin{matrix}
\epsfig{file=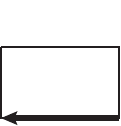, scale=0.8}
\\
\boldsymbol{1}
\\
\text{identity for}
\\
\text{tensor}
\end{matrix}
\qquad
\begin{matrix}
\epsfig{file=hopfid, scale=0.8}
\\
id
\\
\text{identity for}
\\
\text{composition}
\end{matrix}
\qquad
\begin{matrix}
\epsfig{file=hopfmu, scale=0.8}
\\
\mu
\\
\text{multiplication}
\\{}
\end{matrix}
\qquad
\begin{matrix}
\epsfig{file=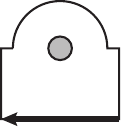, scale=0.8}
\\
\eta
\\
\text{unit}
\\{}
\end{matrix}
\qquad
\begin{matrix}
\epsfig{file=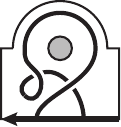, scale=0.8}
\\
S
\\
\text{antipode}
\\{}
\end{matrix}
\\{}\\
\begin{matrix}
\epsfig{file=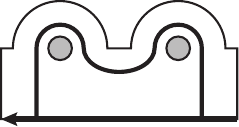, scale=0.8}
\\
\Delta
\\
\text{coproduct}
\end{matrix}
\qquad
\begin{matrix}
\\{}\\[-10pt]
\epsfig{file=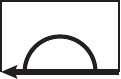, scale=0.8}
\\
\varepsilon
\\
\text{counit}
\end{matrix}
\qquad
\begin{matrix}
\epsfig{file=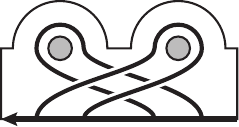, scale=0.8}
\\
\Psi
\\
\text{braiding}
\end{matrix}
\end{matrix}
\]
\caption{Braided Hopf algebra  structure of bottom tangles.}
\label{fig:Hopf}
\end{figure}
In general, we define
\begin{align*}
\mu_i &= id^{\otimes(i-1)} \otimes \mu \otimes id^{\otimes(n-i-1)},
&\Delta_i = id^{\otimes(i-1)} \otimes \Delta \otimes id^{\otimes(n-i)},
\\
\eta_i &= id^{\otimes(i-1)} \otimes \eta \otimes id^{\otimes(n-i)},
&\varepsilon_i = id^{\otimes(i-1)} \otimes \varepsilon \otimes id^{\otimes(n-i)},\quad
\\
S_i &= id^{\otimes(i-1)} \otimes S_i \otimes id^{\otimes(n-i)},
&\Psi_i = id^{\otimes(i-1)} \otimes \Psi \otimes id^{\otimes(n-i-1)}.
\end{align*}
By using the braiding, we define the braided multiplication 
$\boldsymbol{\mu} \in \mathcal{T}_{k, 2k}$ 
by
\begin{equation}
\boldsymbol{\mu}
=
(\underset{k}{\underbrace{\mu \otimes \cdots \otimes \mu}}) \circ \Psi_{2k-2} \circ (\Psi_{2k-4}\circ \Psi_{2k-3}) \circ \cdots \circ (\Psi_4 \circ \Psi_5 \circ \cdots \circ \Psi_{k+1}) \circ (\Psi_2 \circ \Psi_3 \circ \cdots \circ \Psi_{k}).  
\label{eq:braidedmultiplication}
\end{equation}
This $\boldsymbol{\mu}$ gives a multiplication from $\mathcal{T}_{k, n_1} \otimes \mathcal{T}_{k, n_2}$ to $\mathcal{T}_{k, n_1+n_2}$  by $\boldsymbol{\mu}\circ(T_1 \otimes T_2)$ where $T_1\in \mathcal{T}_{k,n_1}$ and $T_2 \in \mathcal{T}_{k, n_2}$.
The mapping $\boldsymbol{\mu}$ agrees with
the previous multiplication given by stacking two bottom tangles discussed in the previous subsection.  
With respect to the multiplication $\boldsymbol{\mu}$, 
\[
\mathcal{T}_k = \bigoplus_{j=0, 1, 2, \cdots} \mathcal{T}_{k,j}
\]
 has a graded algebra structure.  
From the stacking interpretation we see that $\boldsymbol{\mu}$ is associative and we write \[
\boldsymbol{\mu} = 
\boldsymbol{\mu}\circ (\boldsymbol{\mu}\otimes id) \circ \cdots \circ(\boldsymbol{\mu}\otimes id^{\otimes(m-2)})
: \mathcal{T}_{k, n_1} \otimes \cdots\otimes \mathcal{T}_{k, n_l} \to \mathcal{T}_{k, n_1+ \cdots + n_l}.
\]
 \par
 Let 
\[
\alpha_i = \eta^{\otimes(i-1)}\otimes id \otimes \eta^{\otimes(k-i)}
\in \mathcal{T}_{k,1},  
\]
Then the braiding $\Psi$ is expressed by $\Psi = \boldsymbol{\mu}(\alpha_2 \otimes \alpha_1)$ for $\alpha_1, \alpha_2 \in \mathcal{T}_{2, 1}$.  
\begin{figure}[htb]
\[
\alpha_i \ = \ 
\begin{matrix}
\epsfig{file=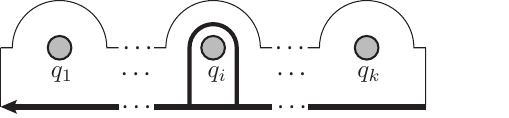, scale=0.8}
\end{matrix}
\]
\caption{$\alpha_i = \eta^{\otimes(i-1)}\otimes id \otimes \eta^{\otimes(k-i)}
\in \mathcal{T}_{k,1}$}
\label{fig:alpha}
\end{figure}
 \par
We also define another multiplication $m : \mathcal{T}_{k, 1}\otimes \mathcal{T}_{k, 1}\to \mathcal{T}_{k, 1}$ by 
\[
m(T_1 \otimes T_2) = 
\boldsymbol{\mu} \circ (T_1 \otimes T_2)\circ\Delta \quad
\text{for $T_1$, $T_2 \in  \mathcal{T}_{k, 1}$,}
\]
and use the shorthand $T_1  T_2$ to mean $m(T_1 \otimes T_2)$. This product is associative because $\boldsymbol{\mu}$ is. 
By this new product, we have $S(\alpha_i) \alpha_i = \alpha_i S(\alpha_i) = \epsilon\, ( = \epsilon\otimes \eta^m)$, so we denote $S(\alpha_i)$ by $\alpha_i^{-1}$.  
\begin{figure}[htb]
\[
\begin{matrix}
\Big(
&\begin{matrix}
\epsfig{file=freearcs9, scale=0.8}
\end{matrix}
& , &
\begin{matrix}
\epsfig{file=freearcs8, scale=0.8}
\end{matrix}
&\Big)
& = &
\begin{matrix}
\epsfig{file=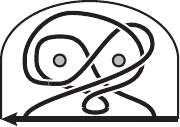, scale=0.8}
\end{matrix}
\end{matrix}
\]
\caption{The multiplication $m$ of $\mathcal{T}_{k, 1}$.}
\label{fig:multiplication2}
\end{figure}
\subsection{Coadjoint action}
Here we introduce the coadjoint action to bottom tangles.  
\begin{definition}
The {\it coadjoint action} $\ad$ is defined by
\[
\ad = \mu_2\circ \Psi_1\circ(S \otimes \Delta)\circ \Delta
\in \mathcal{T}_{2, 1}.
\]
\end{definition}
\begin{figure}[htb]
\[
\ad \ = \ 
\begin{matrix}
\epsfig{file=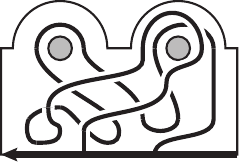, scale=0.8}
\end{matrix}
\ =\ 
\begin{matrix}
\epsfig{file=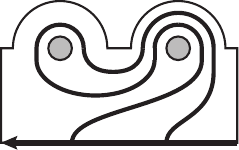, scale=0.8}
\end{matrix}
\]
\caption{The coadjoint action $\ad$.}
\label{fig:coadjoint}
\end{figure}
The leftmost picture of Figure \ref{fig:coadjoint}  means that 
\[
\ad = \alpha_2^{-1} \alpha_1^{} \alpha_2^{} \in \mathcal{T}_{2,1}.
\]
\begin{prop}
The braided Hopf algebra structure of the actions of the bottom tangles to the algebra of free arcs satisfies the following braided commutativity.
\begin{equation}
\mu_2\circ(\ad \otimes id)
=
\mu_2\circ\Psi_1\circ(id \otimes \ad)\circ\Psi
\in \mathcal{T}_{2, 2}.  
\label{eq:commutativity}
\end{equation}
\end{prop}
\begin{proof}
This is proved by deforming the diagrams corresponding to the left hand side and the right hand side of \eqref{eq:commutativity}.
Both diagrams are deformed into the same diagram as in Figure \ref{fig:commutativity}.
\end{proof}
\begin{figure}[htb]
\begin{align*}
\mu_2\circ(\ad \otimes id) \ &= \ \ 
\begin{matrix}
\epsfig{file=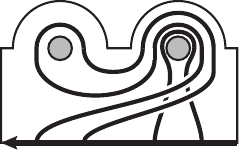, scale=0.8}
\end{matrix}\ ,
\\
\mu_2\circ\Psi_1\circ(id \otimes \ad)\circ\Psi
\ &=\ 
\begin{matrix}
\epsfig{file=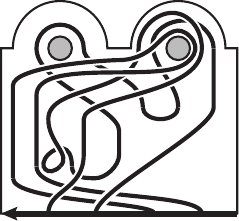, scale=0.8}
\end{matrix}
\ =\ 
\begin{matrix}
\epsfig{file=hopfcocommutative1, scale=0.8}
\end{matrix}\ .
\end{align*}
\caption{Proof of the braided commutativity.}
\label{fig:commutativity}
\end{figure}
\begin{definition}
Let $\Ad$ be the element in $\mathcal{T}_{k+1, k}$ given by 
\[
\Ad = (id^{\otimes k}\otimes \boldsymbol{\mu})\circ  (\Psi_k\Psi_{k-1}\cdots\Psi_2\ad_1) \circ
(\Psi_{k}\Psi_{k-1}\cdots\Psi_3\ad_2) \circ\cdots\circ
(\Psi_{k}\ad_{k-1})\circ \ad_k
\]
where $ad_i = id^{\otimes (i-1)}\otimes \ad \otimes id^{\otimes (k-i)}$.  
\end{definition}
The bottom tangle $\Ad$ is given by a crossingless bottom tangle as in Figure  \ref{fig:Ad}.  
\begin{figure}[htb]
\[
\begin{matrix}
(id^{\otimes k}\otimes \boldsymbol{\mu}) \circ
\begin{matrix}
\epsfig{file=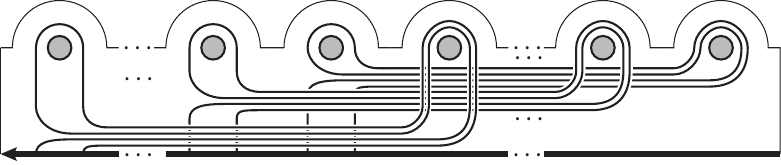, scale=0.8}
\\
(\Psi_k\circ\Psi_{k-1}\cdots\Psi_2\circ\ad_1)\cdots(\Psi_k\circ\ad_{k-1})\ad_k
\end{matrix}
\\
=\  \begin{matrix}
\epsfig{file=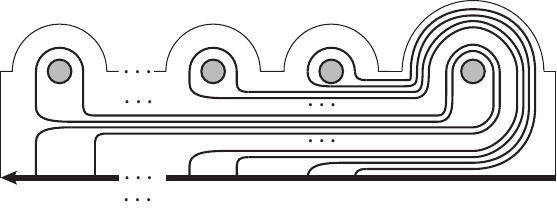, scale=0.8}
\\
\Ad
\end{matrix}
\end{matrix}
\]
\caption{The adjoint action $\Ad$.}
\label{fig:Ad}
\end{figure}
\subsection{Flat bottom tangles}
\begin{definition}
A tangle $T \in \mathcal{T}_{k,n}$ is called a {\it flat bottom tangle} if $T$ has a projection without crossings.  
Let $\mathcal{T}^F_{k, n}$ be the subspace  of $\mathcal{T}_{k, n}$ spanned by all the flat bottom tangles in $\mathcal{T}_{k, n}$.  
\end{definition}
\begin{prop}
The composition of two flat bottom tangles is  a flat bottom tangle.  
So the flat bottom tangles form a subcategory $\mathcal{B}^F$ of $\mathcal{B}$.  
\end{prop}
\begin{prop}
Any element $F$ of $\mathcal{T}^F_{k,n}$ commutes with the multiplication $\boldsymbol{\mu} : \mathcal{T}_{n, l_1} \otimes \mathcal{T}_{n, l_2} \to \mathcal{T}_{n, l_1+l_2}$,  i.e.
\[
F\circ \boldsymbol{\mu}(T_1 \otimes T_2)
=
\boldsymbol{\mu}\left((F\circ T_1)\otimes(F\circ T_2)\right) \quad
\text{for \  $T_1 \in  \mathcal{T}_{n, l_1}$ and \ $T_2 \in  \mathcal{T}_{n, l_2}$}.
\] 
\label{prop:flat}
\end{prop}
\begin{proof}
The ribbons  of $F\circ \mu(T_1 \otimes T_2)$ are split into the arcs from $T_1$ and those from $T_2$ as in Figure \ref{fig:commutativityFM}.  
Since $F$ has no self-intersection, the arcs from $T_1$ and those from $T_2$ can be separated into different levels.  
So it is equal to $\mu\left((F\circ T_1)\otimes(F\circ T_2)\right)$.  
\end{proof}
\begin{figure}[htb]
\begin{multline*}
\begin{matrix}
\epsfig{file=freearcs12, scale=0.8}
\\F
\end{matrix}
\ \circ\ 
\boldsymbol{\mu}\left(
\begin{matrix}
\epsfig{file=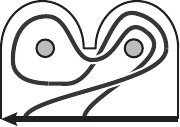, scale=0.8}
\\
T_1
\end{matrix}
\otimes\begin{matrix}
\epsfig{file=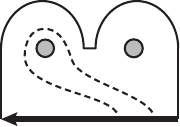, scale=0.8}
\\T_2
\end{matrix}
\right)
\\ =\
\begin{matrix}
\epsfig{file=freearcs12, scale=0.8}
\end{matrix}
\ \circ\ 
\begin{matrix}
\epsfig{file=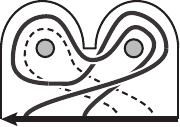, scale=0.8}
\end{matrix}
\\ =\ 
\begin{matrix}
\epsfig{file=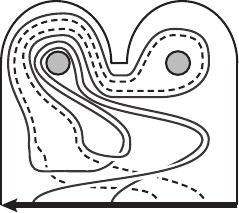, scale=0.7}
\end{matrix}
\ 
\underset{*}{=} \ 
\boldsymbol{\mu}\left(
\begin{matrix}
\epsfig{file=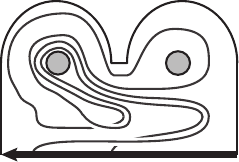, scale=0.7}
\end{matrix}\otimes
\begin{matrix}
\epsfig{file=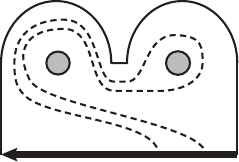, scale=0.7}
\end{matrix}
\right)
\\
 =\ 
\boldsymbol{\mu}\left(
\left(
\begin{matrix}
\epsfig{file=freearcs12, scale=0.7}
\end{matrix}
\ \circ \ 
\begin{matrix}
\epsfig{file=freearcs181, scale=0.7}
\end{matrix}\right)
\ \otimes\ 
\left(
\begin{matrix}
\epsfig{file=freearcs12, scale=0.7}
\end{matrix}
\ \circ \ 
\begin{matrix}
\epsfig{file=freearcs171, scale=0.7}
\end{matrix}
\right)
\right)
\end{multline*}
\caption{Commutativity of elements of $\mathcal{T}^F_{k,n}$ and $\boldsymbol{\mu}$. 
The flatness of $F$ is used at the equality marked by $*$.}
\label{fig:commutativityFM}
\end{figure}
\begin{prop}
The bottom tangle $\Ad$ commutes with any bottom tangle $T \in \mathcal{T}_{k,n}$, i.e.
\begin{equation}
\Ad \circ T = (T \otimes id) \circ \Ad.
\label{eq:Ad}
\end{equation}  
\label{prop:Ad}
\end{prop}
\begin{proof}
$\Ad$ commutes with $T$ as in Figure \ref{fig:Adcommute}.  
\end{proof}
\begin{figure}[htb]
\[
\begin{matrix}
\epsfig{file=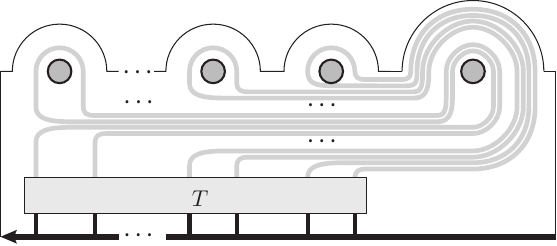, scale=0.7}
\end{matrix}
\ = \ 
\begin{matrix}
\epsfig{file=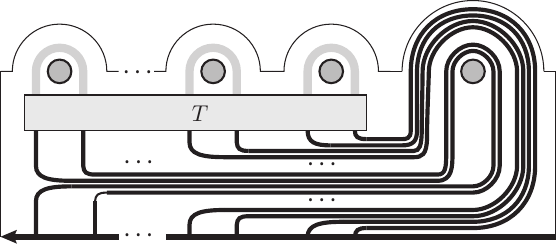, scale=0.7}
\end{matrix}
\]
\caption{The adjoint coaction $\Ad$ commutes with $T \in \mathcal{T}_{k,n}$. 
The gray thick lines represent bunches of strings.  }
\label{fig:Adcommute}
\end{figure}
\subsection{Braid group action}
\label{subsection:braid}
The braid group $B_k$ of $k$ strings acts on the $k$ punctured disk where the generator $\sigma_i$ acts by the counter-clockwise half twist to swap the $i$-th and $(i+1)$-th punctures.  
Extending this action to $D_k\times [0,1]$ the action of 
the generator $\sigma_i$ and its inverse $\sigma_i^{-1}$ is given by composing with the bottom tangles $id^{\otimes(i-1)} \otimes T_\sigma \otimes id^{\otimes(n-i-1)}$ and $id^{\otimes(i-1)} \otimes T_{\sigma^{-1}} \otimes id^{\otimes(n-i-1)}$. Here $T_\sigma$ and $T_{\sigma^{-1}}$ are 
bottom tangles in $\mathcal{T}_{2,2}$ given in Figure \ref{fig:sigmaaction}.  
These elements are expressed as follows.
\begin{align*}
T_\sigma &=\mu_2\circ\Psi_1\circ(id \otimes \ad),
\\
T_{\sigma^{-1}}& = 
\mu_1 \circ \Psi_2^{-1}\circ\Psi_1^{-1}\circ\Psi_2^{-1}\circ S_2^{-1}\circ(\ad \otimes id).
\end{align*}
Here $S^{-1}$ is the inverse of $S$ with respect to the composition, which is given by the bottom tangle in Figure \ref{fig:Sinverse}.  
\begin{figure}[htb]
\[
\begin{matrix}
\epsfig{file=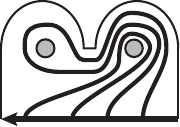, scale=0.8}
\\
T_\sigma
\end{matrix}
\qquad\qquad
\begin{matrix}
\epsfig{file=freearcs12, scale=0.8}
\\
T_{\sigma^{-1}}
\end{matrix}
\]
\caption{The bottom tangles in $\mathcal{T}_{2,2}$ corresponding to $\sigma$ and $\sigma^{-1}$.}
\label{fig:sigmaaction}
\end{figure}
\begin{figure}[htb]
\begin{center}
$S^{-1} \ \ : \ \ 
\begin{matrix}
\epsfig{file=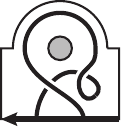, scale=0.8}
\end{matrix}$
\end{center}
\caption{The bottom tangle $S^{-1}$.}
\label{fig:Sinverse}
\end{figure}
Any element in the braid group may be represented by a composition of $T_\sigma$ and $T_{\sigma^{-1}}$. These are both flat bottom tangles so we have the following.  
\begin{prop}
Any element $b$ of the braid group $B_k$ is represented by a flat bottom tangle $T_b$.
\label{prop:Bn}
\end{prop}
%
%
%
Proposition  \ref{prop:flat} and Proposition \ref{prop:Bn} imply the following.  
\begin{prop}
Any element of $B_k$ acts on $\mathcal{T}_{k}$ as an algebra automorphism.  
\label{prop:algaut}
\end{prop}
As mentioned in the introduction, $\mathcal{T}_{k}$ becomes the group algebra of the free group $\pi_1(D_k,p)$ when we ignore the crossings of the bottom tangles.
As such the above construction generalizes the Artin representation of the braid group as automorphisms of the free group.

Our convention for the order of multiplication in $B_n$ is that $a b$ is the braid $a$ stacked on top of braid $b$, also compare the braid in Figure \ref{fig:rotation}.
In terms of Hopf diagrams the fact that $T_b$ acts as an algebra automorphism becomes
\[T_b \, \boldsymbol{\mu}^{}(\boldsymbol x \otimes \boldsymbol y)
=
\boldsymbol{\mu}^{}\big(T_b \, \boldsymbol x\otimes T_b \, \boldsymbol y\big).
\]

\section{A quantum analogue of a knot group}
\subsection{Plat presentation}
For any knot $K$, there is $b \in B_{2k}$ such that $K$ is isotopic to the plat closure $\widehat b$ of $b$ as in Figure \ref{fig:plat}.  
\begin{prop}[Cf. \cite{CG}, Theorem 2]
Plat closures of two braids $b_1 \in B_{2k_1}$ and $b_2 \in B_{2k_2}$ are isotopic if and only if there is a sequence of following moves which transforms $b_1$ to $b_2$.  
\begin{enumerate}
\item[\bf P1.]\quad
$\sigma_1\, b \longleftrightarrow b \longleftrightarrow b\, \sigma_1\qquad
(b \in B_{2k})$
\item[\bf P2.]\quad
$\sigma_{2i}\,\sigma_{2i+1}\,\sigma_{2i-1}\,\sigma_{2i}\, b
\longleftrightarrow
b
\longleftrightarrow
b \, \sigma_{2i}\,\sigma_{2i+1}\,\sigma_{2i-1}\,\sigma_{2i}
\quad
(b \in B_{2k}, \ i = 1, 2, \cdots, k-1)$
\item[\bf  P3.]\quad
$\sigma_2\, \sigma_1^2 \, \sigma_2 \, b 
\longleftrightarrow
b
\longleftrightarrow
b \, \sigma_2\, \sigma_1^2 \, \sigma_2
\qquad
(b \in B_{2k})$
\item[\bf P4.]\quad
$b \longleftrightarrow
\sigma_{2k}\, b \quad
(b \in B_{2k}, \ \  \sigma_{2k} \, b \in B_{2k+2})$.
\end{enumerate}
\label{prop:plat}
\end{prop} 
\begin{proof}
The ``if'' part is obvious and we prove the ``only if'' part.  
Applying Theorem 2 of \cite{CG} to the genus 0 case, we know  that two elements in $B_{2k}$ determine equivalent links if and only if they are connected by a sequence of moves {\bf M1}, {\bf M2}, {\bf M3} and {\bf M6} given in the theorem as follows.  
\begin{enumerate}
\item[\bf M1.]\quad
$\sigma_1 \, b \longleftrightarrow b \longleftrightarrow b\, \sigma_1$
\item[\bf M2.]\quad
$\sigma_{2i}\,\sigma_{2i+1}\, \sigma_{2i-1}\, \sigma_{2i} \, b 
\longleftrightarrow b \longleftrightarrow 
b\,\sigma_{2i}\,\sigma_{2i+1}\, \sigma_{2i-1}\, \sigma_{2i}$
\item[\bf M3.]\quad
$\sigma_{2}\, \sigma_{1}^2 \, \sigma_{2}\, b 
\longleftrightarrow b \longleftrightarrow 
b \, \sigma_{2}\, \sigma_{1}^2 \, \sigma_{2}$ 
\item[\bf M6.]\quad
$b \longleftrightarrow T_j(b)\, \sigma_{2j}$
\quad
where $T_j : B_{2k} \to B_{2k+2}$ is given by
\[
T_j(\sigma_i) = 
\begin{cases}
\sigma_i & \text{if $i < 2j$} 
\\
\sigma_{2j}\,\sigma_{2j+1}\, \sigma_{2j+2}\, \sigma_{2j+1}^{-1}\, \sigma_{2j}^{-1} & \text{if $i = 2j$}
\\
\sigma_{i+2} & \text{if $i > 2j$}.
\end{cases}
\]
\end{enumerate}
Let {\bf M6'} be the following move.
\begin{enumerate}
\item[\bf M6'.]\quad
$b \longleftrightarrow b\, \sigma_{2k}$. 
\end{enumerate}
By Lemma \ref{lem:M6} below, we can use the move {\bf M6'}  instead of {\bf M6}.  
Let $\varphi$ be the automorphism of $B_{2k}$ sending $\sigma_i$ to $\sigma_i^{-1}$, and let $\bar\varphi(b) = \varphi(b)^{-1}$.  	
Then $\bar\varphi$ is  an anti-automorphism of $B_{2k}$, and the plat closure  $\widehat{\bar\varphi(b)}$ is isotopic to $\widehat b$ since $\widehat{\bar\varphi(b)}$ is obtained from $\widehat b$ by $\pi$ rotation around the horizontal line  as in Figure \ref{fig:rotation}.  
\begin{figure}[htb]
\[
\widehat{b} \quad
\text{---}
\begin{matrix}
\includegraphics[scale=0.8]{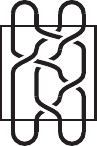}
\end{matrix}
\text{---}
\begin{matrix}
\includegraphics[scale=0.8]{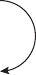}
\end{matrix}
\quad
\longrightarrow
\quad
\begin{matrix}
\includegraphics[scale=0.8]{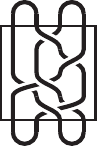}
\end{matrix}
\quad
\widehat{\bar\varphi(b)}
\]
\caption{Compare $\widehat{b}$ and $\widehat{\bar\varphi(b)}$ where $b = \sigma_2^{-1} \sigma_3\sigma_1\sigma_2^{-1}\sigma_3\sigma_2^{-1}$.}
\label{fig:rotation}
\end{figure}
By applying $\bar\varphi$,  the moves {\bf M1}, {\bf M2}, {\bf M3} do not change and are equal to  {\bf P1}, {\bf P2}, {\bf P3} respectively, and {\bf M6'} is converted to   {\bf P4}. 
Hence two elements in $B_{2k}$ determine equivalent links if and only if they are connected by a sequence of moves {\bf P1}, {\bf P2}, {\bf P3} and {\bf P4}.  
\end{proof}
\begin{lemma} 
For genus 0 case, the move {\bf M6} is given by a sequence of {\bf M6'} and {\bf M2}.
\label{lem:M6}
\end{lemma}
\begin{proof}
Let $\tau_i = \sigma_{2i}\sigma_{2i-1}\sigma_{2i+1}\sigma_{2i}$, 
then 
$T_j(b) = (\tau_{j+1} \, \tau_{j+2}\,\cdots\, \tau_k)\, b\, (\tau_{j+1} \, \tau_{j+2}\,\cdots\, \tau_k)^{-1}$ and
\begin{multline*}
T_j(b)\,\sigma_{2j}\underset{\text{\bf M2}}{\longleftrightarrow}
T_j(b)\,\sigma_{2j}\, \tau_{j+1} \, \tau_{j+2}\cdots\tau_k
=
T_j(b)\,\tau_{j+1} \,\tau_{j+2} \cdots\tau_{k}\, \sigma_{2k}=
\\
\tau_{j+1} \, \tau_{j+2}\,\cdots\, \tau_k\,b\,\sigma_{2k}
\underset{\text{\bf M6'}}{\longleftrightarrow}
\tau_{j+1} \, \tau_{j+2}\,\cdots\, \tau_k\,b
\underset{\text{\bf M2}}{\longleftrightarrow}
b.
\end{multline*}
\label{lem:markov}
\end{proof}
\subsection{Braided Hopf diagrams}
For the bottom tangles corresponding to operations of a Hopf algebra, we use the  diagrams in Figure \ref{fig:BHD} instead of the bottom tangles.  
We call such diagrams {\it braided Hopf diagrams} and they should be read bottom to top.

The action of a braided Hopf diagram is given by the corresponding bottom tangle.  
\begin{figure}[htb]
\[
\begin{matrix}
\begin{matrix}
\begin{matrix}
\includegraphics[scale=0.6]{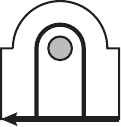}
\end{matrix}
& \leftrightarrow &
\begin{matrix}
\includegraphics[scale=0.8]{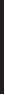}
\end{matrix}\ ,
\\
id & 
\end{matrix}
\qquad
\begin{matrix}
\begin{matrix}
\epsfig{file=hopf1, scale=0.6}
\end{matrix}
& \leftrightarrow &
\begin{matrix}
\epsfig{file=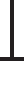, scale=0.8}
\end{matrix}\ ,
\\
\eta
\end{matrix}
\qquad
\begin{matrix}
\begin{matrix}
\\{}\\[-8pt]
\epsfig{file=hopf0, scale=0.6}
\end{matrix}
&\leftrightarrow&
\begin{matrix}
\epsfig{file=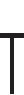, scale=0.8}
\end{matrix},
\\
\varepsilon
\end{matrix}
\qquad
\begin{matrix}
\begin{matrix}
\includegraphics[scale=0.6]{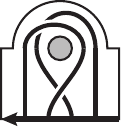}
\end{matrix}
& \leftrightarrow & 
\begin{matrix}
\includegraphics[scale=0.8]{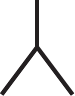}
\end{matrix},
\\
\mu
\end{matrix}
\\{}\\
\begin{matrix}
\begin{matrix}
\epsfig{file=hopfS, scale=0.6}
\end{matrix}
& \leftrightarrow &
\begin{matrix}
\epsfig{file=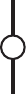, scale=0.8}
\end{matrix}\ ,
\\
S
\end{matrix}
\qquad
\begin{matrix}
\begin{matrix}
\epsfig{file=hopfSinverse, scale=0.6}
\end{matrix}
& \leftrightarrow &
\begin{matrix}
\epsfig{file=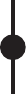, scale=0.8}
\end{matrix}\ ,
\\
S^{-1}
\end{matrix}
\qquad
\begin{matrix}
\begin{matrix}
\epsfig{file=hopfdelta, scale=0.6}
\end{matrix}
&\leftrightarrow&
\begin{matrix}
\epsfig{file=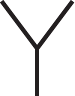, scale=0.8}
\end{matrix},
\\
\Delta
\end{matrix}
\\{}\\
\begin{matrix}
\begin{matrix}
\epsfig{file=hopfbraiding, scale=0.6}
\end{matrix}
&\leftrightarrow&
\begin{matrix}
\epsfig{file=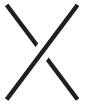, scale=0.8}
\end{matrix},
\\
\Psi
\end{matrix}
\qquad
\begin{matrix}
\begin{matrix}
\epsfig{file=hopfcoadjoint, scale=0.6}
\end{matrix}
&\leftrightarrow&
\begin{matrix}
\epsfig{file=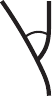, scale=0.8}
\end{matrix}
&=&
\begin{matrix}
\scalebox{1}[-1]{\includegraphics[scale=0.6]{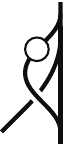}}
\end{matrix}\ .
\\
\mathrm{ad}
\end{matrix}
\\{}\\
\begin{matrix}
\begin{matrix}
\epsfig{file=freearcs19, scale=0.6}
\end{matrix}
&\leftrightarrow&
\begin{matrix}
\epsfig{file=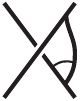, scale=0.8}
\end{matrix}\ ,
\\
T_\sigma
\end{matrix}
\qquad
\begin{matrix}
\begin{matrix}
\epsfig{file=freearcs12, scale=0.6}
\end{matrix}
&\leftrightarrow&
\begin{matrix}
\epsfig{file=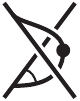, scale=0.8}
\end{matrix}\ .
\\
T_{\sigma^{-1}}
\end{matrix}
\end{matrix}
\]
\caption{Braided Hopf diagrams corresponding to bottom tangles.}
\label{fig:BHD}
\end{figure}
%
%
%
%

We also recall a few properties of $\mathrm{ad}$ in the notation of Hopf diagrams in Figure \ref{fig:propad}. 
For the proofs see \cite{MV} (read after reflecting all pictures in the $x$-axis) or convert them to bottom tangles. Note that braided commutativity may not hold in all braided Hopf algebras but in the context of bottom tangles it is always satisfied. 

\begin{figure}[htb]
\[
\begin{matrix}

\begin{matrix}
\begin{matrix}
\scalebox{1}[-1]{\includegraphics[scale=0.7]{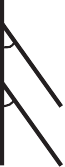}}
\end{matrix}
&= &
\begin{matrix}
\scalebox{1}[-1]{\includegraphics[scale=0.7]{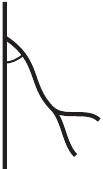}}
\end{matrix}\ ,
\\
\end{matrix}
\qquad

\begin{matrix}
\begin{matrix}
\scalebox{1}[-1]{\includegraphics[scale=0.7]{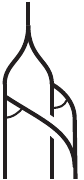}}
\end{matrix}
&= &
\begin{matrix}
\scalebox{1}[-1]{\includegraphics[scale=0.7]{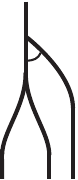}}
\end{matrix}\ ,
\\

\end{matrix}
\qquad

\begin{matrix}
\begin{matrix}
\scalebox{1}[-1]{\includegraphics[scale=0.7]{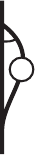}}
\end{matrix}
&= &
\begin{matrix}
\scalebox{1}[-1]{\includegraphics[scale=0.7]{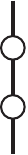}}
\end{matrix}\ ,
\end{matrix}
\\
\begin{matrix}
\begin{matrix}
\scalebox{1}[-1]{\includegraphics[scale=0.7]{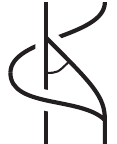}}
\end{matrix}
\ = \ 
\begin{matrix}
\scalebox{1}[-1]{\includegraphics[scale=0.7]{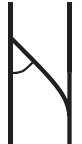}}
\end{matrix}\ ,
&\qquad&
\begin{matrix}
\scalebox{1}[-1]{\includegraphics[scale=0.7]{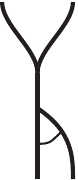}}
\end{matrix}
\ = \ 
\begin{matrix}
\scalebox{1}[-1]{\includegraphics[scale=0.7]{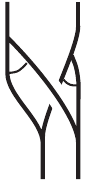}}
\end{matrix}\ ,
\\
\begin{matrix}
\scalebox{1}[-1]{\includegraphics[scale=0.7]{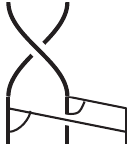}}
\end{matrix}
\ = \ 
\begin{matrix}
\scalebox{1}[-1]{\includegraphics[scale=0.7]{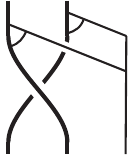}}
\end{matrix}\ ,
&\qquad&
\begin{matrix}
\scalebox{1}[-1]{\includegraphics[scale=0.7]{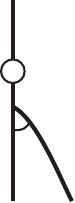}}
\end{matrix}
\ = \ 
\begin{matrix}
\scalebox{1}[-1]{\includegraphics[scale=0.7]{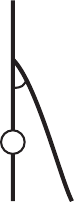}}
\end{matrix}\ .
\\
\text{braided\qquad\quad}
\\
\text{commutativity}
\end{matrix}
\end{matrix}
\]
\caption{Some properties of $\mathrm{ad}$.
The leftmost relation of the bottom row is called the braided commutativity.}
\label{fig:propad}
\end{figure}

\subsection{A quantum deformation of a knot group}
\label{sub:universal}
Let $K$ be a knot or a link and $b\in B_{2k}$ be a braid whose plat closure is isotopic to $K$.  
Let $T_b$ be the bottom tangle corresponding to $b$ and 
$\widehat{T}_b$, $\boldsymbol{\varepsilon}_k$ be the bottom tangles in $\mathcal{T}_{k, 2k}$ as follows.  
\begin{align*}
\widehat{T}_b 
&= 
\big(\mu \circ(id \otimes S)\big)^{\otimes k} \circ 
(\mu_{2k-1}\circ\Psi_{2k})\circ\cdots\circ
\\
&\qquad\qquad
(\mu_3\circ\Psi_4\circ\cdots\circ\Psi_{2k-1}\circ\Psi_{2k})
\circ
(\mu_1\circ\Psi_2\circ\cdots\circ\Psi_{2k-1}\circ\Psi_{2k})
\circ
T_b \circ \Delta^{\otimes k},
\\
\boldsymbol{\varepsilon}_k
&=
\varepsilon^{\otimes k}\otimes id^{\otimes k}.  
\end{align*}
\begin{figure}[htb]
\[
\widehat{T}_b : \quad
\begin{matrix}
\includegraphics[scale=0.8]{hatt}
\end{matrix} \ ,
\qquad
\boldsymbol{\varepsilon}_k : \quad
\begin{matrix}
\includegraphics[scale=0.8]{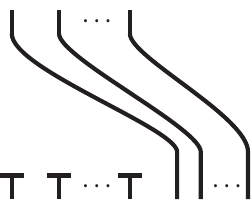}
\end{matrix} \ .
\]
\caption{The bottom tangles $\widehat{T}_b$ and $\boldsymbol{\varepsilon}_k$.}
\label{fig:hatT}
\end{figure}
Let $J_b$ be the submodule of $\mathcal{T}_k$ spanned by the image of $\widehat{T}_b-\boldsymbol{\varepsilon}_k$, and 
let $\mathcal{A}_b = \mathcal{T}_k/J_b$.  
Then the quotient space $\mathcal{A}_b$ has an $\Ad$-coinvariant comodule structure.  
This comodule $\mathcal{A}_b$ is graded as $\mathcal{A}_{b} = \oplus_{n=0}^\infty \mathcal{A}_{b, n}$ where $\mathcal{A}_{b, n} = \mathcal{T}_{k,n}/(\widehat{T}_b-\boldsymbol{\varepsilon}_k)(\mathcal{T}_{2k,n})$.  
\begin{thm}
\label{thm:space}
If the plat closures of two braids $b_1$ and $b_2$ are isotopic, then $\mathcal{A}_{b_1}$ and $\mathcal{A}_{b_2}$ are isomorphic as $\Ad$-comodules.  
\label{thm:main}
\end{thm}
To prove this theorem, we introduce a notion of an equivalent pair. 

\subsection{Equivalent pair}
\begin{definition}
Two bottom tangles $T_1$ and $T_2$ are called {\it an equivalent pair} with respect to $b$ if $T_1-T_2$ generates $J_b$.
Such pair is denoted by $T_1 \sim_b T_2$.  
\end{definition}
For an equivalent pair $T \sim_b \boldsymbol{\varepsilon}_n$, we can deform $T$ by the following moves.   
\begin{prop}[Cf. \cite{MV}, Proposition 4.8]
Let  $T_1$, $T_2$ be a pair of braided Hopf diagrams which are identical except their  bottom parts as in the following pictures.    
Assuming that $T_j(1^{\otimes k} \otimes \by) = \by$ for $j=1$, $2$, we have
 $T_1 \sim_b {\varepsilon}^{\otimes k}\otimes id^{\otimes k}$ if and only if $T_2 \sim_b {\varepsilon}^{\otimes k}\otimes id^{\otimes k}$.  
\label{prop:BHDmult}
\begin{align*}
\text{\bf HmL}&\qquad
T_1 \ \begin{matrix}
\includegraphics[scale=0.7]{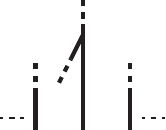}
\\ i 
\end{matrix}
\ \ \longleftrightarrow \quad  T_2 \ 
\begin{matrix}
\includegraphics[scale=0.7]{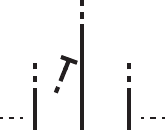}
\\ i
\end{matrix}
\\
\text{\bf HmR}&
\qquad
T_1 \ 
\begin{matrix}
\includegraphics[scale=0.7]{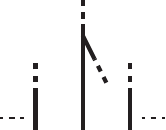}
\end{matrix}
\ \ \longleftrightarrow \quad  T_2 \ 
\begin{matrix}
\includegraphics[scale=0.7]{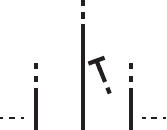}
\end{matrix}
\\
\text{\bf Ha}&\qquad
T_1 \ 
\begin{matrix}
\includegraphics[scale=0.7]{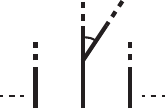}
\end{matrix}
\ \  \longleftrightarrow \quad   T_2 \ 
\begin{matrix}
\includegraphics[scale=0.7]{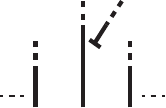}
\end{matrix}
\\
\text{\bf HmL'}&\qquad
T_1 \ \begin{matrix}
\includegraphics[scale=0.7]{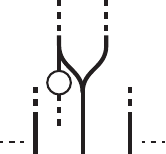}
\end{matrix}
\ \ \longleftrightarrow \quad  T_2 \ 
\begin{matrix}
\includegraphics[scale=0.7]{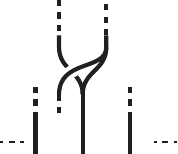}
\end{matrix}
\\
\text{\bf HmR'}&
\qquad
T_1 \ 
\begin{matrix}
\includegraphics[scale=0.7]{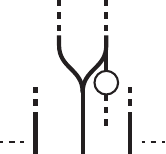}
\end{matrix}
\ \ \longleftrightarrow \quad T_2 \ 
\begin{matrix}
\includegraphics[scale=0.7]{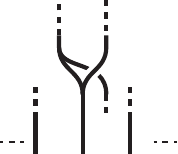}
\end{matrix}
\\
\text{\bf Ha'}&\qquad
T_1 \ 
\begin{matrix}
\includegraphics[scale=0.7]{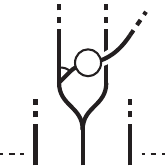}
\end{matrix}
\ \ \ \longleftrightarrow \quad  T_2 \ 
\begin{matrix}
\includegraphics[scale=0.7]{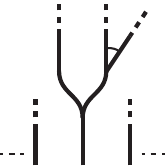}
\end{matrix}
\\
\text{\bf HcL}&\qquad
T_1 \quad
\begin{matrix}
\includegraphics[scale=0.6]{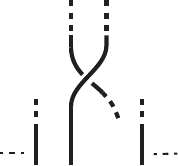}
\end{matrix}
\ \ \longleftrightarrow \quad 
T_2 \ \  
\begin{matrix}
\includegraphics[scale=0.6]{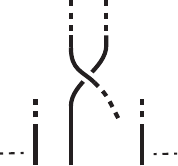}
\end{matrix}
\\
\text{\bf HcR}&\qquad
T_1 \quad 
\begin{matrix}
\includegraphics[scale=0.6]{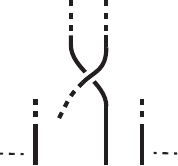}
\end{matrix}
\ \ \longleftrightarrow \quad\ T_2 \ \  
\begin{matrix}
\includegraphics[scale=0.6]{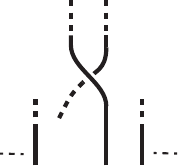}
\end{matrix}
\end{align*}
\label{prop:deform}
\end{prop} 
\begin{proof}
The proof is similar to that for Proposition 4.8 in \cite{MV}.  
As an example, we give a proof for {\bf HmL}.  
Let $M_k = \ker (\boldsymbol{\varepsilon}_{k} \otimes id^{\otimes k})$, 
$J_1 = T_1(M_k)$ and $J_2 = T_2(M_k)$.  
Note that $J_i = \Image\big(T_i - \boldsymbol{\varepsilon}_k\otimes id^{\otimes k}\big)$ for $i=1, 2$.  
Since $\ker \varepsilon_i \subset M_k$, 
 $T_1$ induces the map $\overline T_1$ from $\mathcal{T}_{2k}/\ker \varepsilon_i$ to $\mathcal{T}_k/J_1$ .  
The multiplication at the $i$-th strand in $\mathcal{T}_{2k}/\ker \varepsilon_i$ is equal to the multiplication preceded by the  counit since $\ker \varepsilon_i$ is an ideal of $\mathcal{T}_{2k}$.  
 Hence $\overline{T}_1(\overline{B})$ is equal to the image of $T_2(B)$ in $\mathcal{T}_k/J_1$ 
for any $B\in \mathcal{T}_{2k}$ and its image $\overline B$  in $\mathcal{T}_{2k}/\ker \varepsilon_i$.  
This means that $J_2$ is contained in $J_1$.  
Similar argument shows that $J_1 \subset J_2$ and we get $J_1 = J_2$.  
Hence, if one of $J_1$, $J_2$ is equal to $J_b$, then the other one is also equal to $J_b$.  
%
\end{proof}
\subsection{Proof of the main theorem}
\begin{proof}
We show that $\mathcal{A}_{b_1}$ is isomorphic to $\mathcal{A}_{b_2}$ if $b_2$ is obtained from $b_1$ by one of the moves in Proposition \ref{prop:plat}.  
We will check this for each of the moves.  
\par
{\bf P1.} \ 
First, we check for $b\, \sigma_1 \longleftrightarrow b$.
By Figure \ref{fig:P11}, $T_{\sigma_1}\circ \Delta_1 = \Delta_1$, and so $J_{b\sigma_1} = J_b$.  
\begin{figure}[htb]
\[
\begin{matrix}
\includegraphics[scale=0.8]{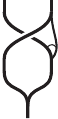}
\end{matrix}
=
\begin{matrix}
\includegraphics[scale=0.8]{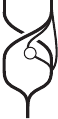}
\end{matrix}
=
\begin{matrix}
\includegraphics[scale=0.8]{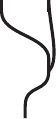}
\end{matrix}
=
\begin{matrix}
\includegraphics[scale=0.8]{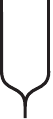}
\end{matrix}
\]
\caption{$\rho(\sigma_1)\circ \Delta_1 = \Delta_1$.}
\label{fig:P11}
\end{figure}
\par
Next, we check for $\sigma_1\, b \longleftrightarrow b$. 
By Figure \ref{fig:P12}, $J_{\sigma_1b} = (S \otimes id^{\otimes(k-1)})\,J_b$ since $S^{-1}$ is an isomorphism.  
Hence 
 $J_b$ is isomorphic to $J_{\sigma_1b}$ since $S$ is an iromorphism.  
\begin{figure}[htb]
\[
\begin{matrix}
\begin{matrix}
\includegraphics[scale=0.8]{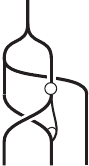}
\end{matrix}
=
\begin{matrix}
\includegraphics[scale=0.8]{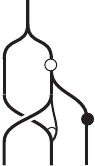}
\end{matrix}
\underset{\text{(bc)}}{=}
\begin{matrix}
\includegraphics[scale=0.8]{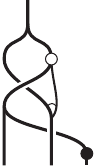}
\end{matrix}
=
\begin{matrix}
\includegraphics[scale=0.8]{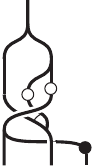}
\end{matrix}
=
\begin{matrix}
\includegraphics[scale=0.8]{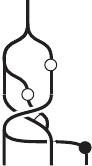}
\end{matrix}
=
\begin{matrix}
\includegraphics[scale=0.8]{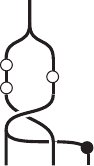}
\end{matrix}
=
\begin{matrix}
\includegraphics[scale=0.8]{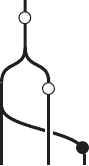}
\end{matrix}
\ ,
\\
\hfill\text{(bc) : braided commutativity}
\end{matrix}
\]
\[
\begin{matrix}
\includegraphics[scale=0.8]{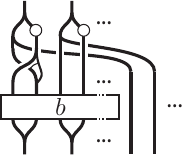}
\end{matrix}
\ - 
\begin{matrix}
\includegraphics[scale=0.8]{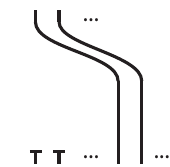}
\end{matrix}
\ = \ \ \ 
\begin{matrix}
\includegraphics[scale=0.8]{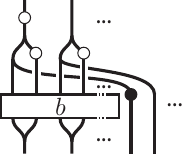}
\end{matrix}
\ - 
\begin{matrix}
\includegraphics[scale=0.8]{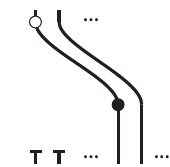}
\end{matrix}\ .
\]
\caption{Invariance for $\sigma_1\, b \longleftrightarrow b$.}
\label{fig:P12}
\end{figure}
\par
{\bf P2.}
First, we check for $b \longleftrightarrow 
b\,\sigma_{2i}\,\sigma_{2i+1}\, \sigma_{2i-1}\, \sigma_{2i}$.  
Since 
\[
\sigma_{2i}\,\sigma_{2i+1}\, \sigma_{2i-1}\, \sigma_{2i}\, \Delta_{2k-1} \cdots\Delta_{2i+1} \, \Delta_{2i-1}\cdots\Delta_1 
=
 \Delta_{2k-1} \cdots\Delta_{2i+1} \, \Delta_{2i-1}\cdots\Delta_1\, \sigma_{i}
\]
 as in Figure \ref{fig:distsigma}, 
 $\varepsilon_{2i-1}\, \varepsilon_{2i}\, \sigma_{2i-1} 
 = 
 \varepsilon_{2i-1}\, \varepsilon_{2i}$
 and $\sigma_i$ is an isomorphism from $\mathcal{T}_n$ to $\mathcal{T}_n$, $\mathcal{A}_{b}$ is isomorphic to $\mathcal{A}_{b\sigma_{2i}\sigma_{2i+1}\sigma_{2i-1} \sigma_{2i}}$.  
 \begin{figure}[htb]
 \[
 \begin{matrix}
\includegraphics[scale=0.8]{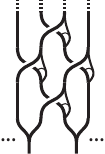}
\end{matrix}
\ = \ 
 \begin{matrix}
\includegraphics[scale=0.8]{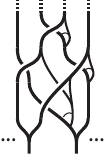}
\end{matrix}
\ = \ 
\begin{matrix}
\includegraphics[scale=0.8]{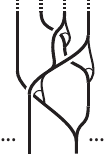}
\end{matrix}
\ = \ 
\begin{matrix}
\includegraphics[scale=0.8]{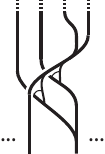}
\end{matrix}
\ = \ 
\begin{matrix}
\includegraphics[scale=0.8]{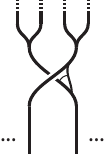}
\end{matrix}
 \]
 \caption{$\sigma_{2i}\sigma_{2i+1}\, \sigma_{2i-1}\sigma_{2i}\Delta_{2k-1} \Delta_{2k-3}\cdots\Delta_1 
=
 \Delta_{2k-1} \Delta_{2k-3}\cdots\Delta_1\sigma_{i}$.}
 \label{fig:distsigma}
 \end{figure}
 \par
Next, we check for $\sigma_{2i}\,\sigma_{2i+1}\, \sigma_{2i-1}\, \sigma_{2i}\, b \longleftrightarrow 
b$. 
Figure \ref{fig:P21} shows that 
$\mathcal{A}_{b\sigma_{2i} \sigma_{2i+1} \sigma_{2i-1} \sigma_{2i}} = \Psi_i \, \mathcal{A}_b$ and so $\mathcal{A}_{b\sigma_{2i} \sigma_{2i+1} \sigma_{2i-1} \sigma_{2i}}$ is isomorphic to $\mathcal{A}_b$.  
 \begin{figure}[htb]
 \begin{multline*}
 \begin{matrix}
\includegraphics[scale=0.8]{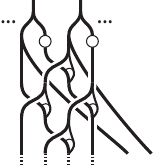}
\end{matrix}
\ = \ 
 \begin{matrix}
\includegraphics[scale=0.8]{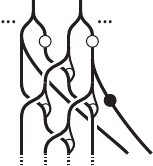}
\end{matrix}
\ \underset{\text{(bc)}}{=} \ 
\begin{matrix}
\includegraphics[scale=0.8]{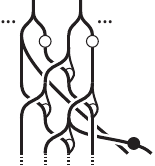}
\end{matrix}
\ \underset{\text{(bc)}}{=} \ 
\begin{matrix}
\includegraphics[scale=0.8]{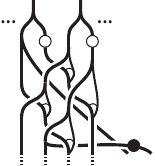}
\end{matrix}
\ = \ 
\begin{matrix}
\includegraphics[scale=0.8]{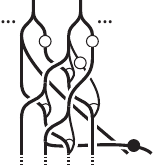}
\end{matrix}
\\
= \ 
\begin{matrix}
\includegraphics[scale=0.8]{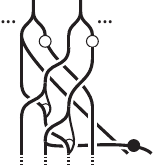}
\end{matrix}
\ = \ 
\begin{matrix}
\includegraphics[scale=0.8]{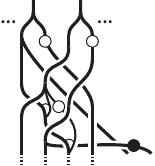}
\end{matrix}
\ = \ 
\begin{matrix}
\includegraphics[scale=0.8]{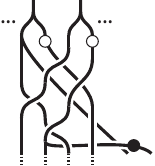}
\end{matrix}
\ = \ 
\begin{matrix}
\includegraphics[scale=0.8]{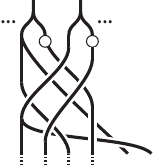}
\end{matrix}
\ = \ 
\begin{matrix}
\includegraphics[scale=0.8]{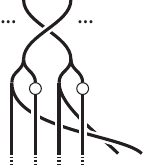}
\end{matrix}\ .
 \end{multline*}
 Therefore,\hfill\ { }
\[
\begin{matrix}
\includegraphics[scale=0.8]{plat41}
\end{matrix}
\ - 
\begin{matrix}
\includegraphics[scale=0.8]{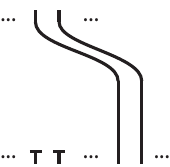}
\end{matrix}
\ = \ \ \ 
\begin{matrix}
\includegraphics[scale=0.8]{plat49}
\end{matrix}
\ - 
\begin{matrix}
\includegraphics[scale=0.8]{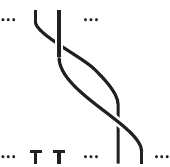}
\end{matrix}\ .
\]
 \caption{Invariance for $\sigma_{2i}\,\sigma_{2i+1}\, \sigma_{2i-1}\, \sigma_{2i}\, b \longleftrightarrow 
b$.}
 \label{fig:P21}
 \end{figure}
 \par
 {\bf P3.}
 First, we check for $b \, \sigma_2\, \sigma_1^2 \, \sigma_2
\longleftrightarrow
b$.  
Figure \ref{fig:P31} shows that 
$\mathcal{A}_{b \sigma_2\sigma_1^2  \sigma_2} = \mathcal{A}_b$.  
\begin{figure}[htb]
\[
\begin{matrix}
\includegraphics[scale=0.8]{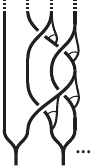}
\end{matrix}
\ = \ 
\begin{matrix}
\includegraphics[scale=0.8]{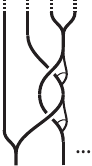}
\end{matrix}
\ \underset{\text{(bc)}}{=} \ 
\begin{matrix}
\includegraphics[scale=0.8]{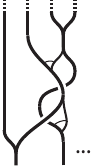}
\end{matrix}
\underset{\text{\bf Ha}}{\longleftrightarrow}\ 
\begin{matrix}
\includegraphics[scale=0.8]{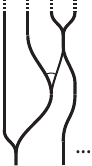}
\end{matrix}
\underset{\text{\bf HmL}}{\longleftrightarrow}\ 
\begin{matrix}
\includegraphics[scale=0.8]{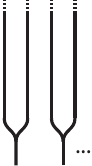}
\end{matrix} \ .
\]
\caption{Invariance for 
$b \, \sigma_2\, \sigma_1^2 \, \sigma_2
\longleftrightarrow b$.}
\label{fig:P31}
\end{figure}
\par
Next, we check for 
$\sigma_2\, \sigma_1^2 \, \sigma_2\, b \longleftrightarrow b$.  
\begin{figure}[htb]
\begin{multline*}
\begin{matrix}
\includegraphics[scale=0.7]{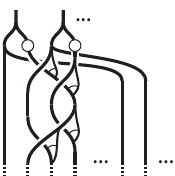}
\end{matrix}
\underset{\text{(bc)}}{=} \ 
\begin{matrix}
\includegraphics[scale=0.7]{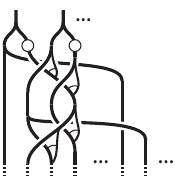}
\end{matrix}
= \ 
\begin{matrix}
\includegraphics[scale=0.7]{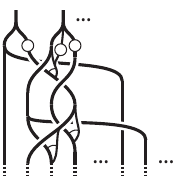}
\end{matrix}
= \ 
\begin{matrix}
\includegraphics[scale=0.7]{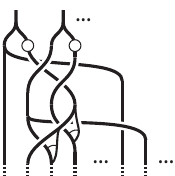}
\end{matrix}
= \ 
\begin{matrix}
\includegraphics[scale=0.7]{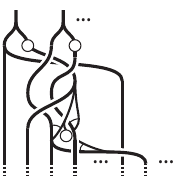}
\end{matrix} 
\underset{\text{(bc)}}{=} \ 
\\
\begin{matrix}
\includegraphics[scale=0.7]{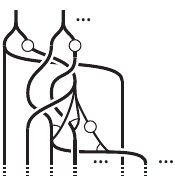}
\end{matrix}\!\!
\underset{\text{(bc)}}{=} 
\begin{matrix}
\includegraphics[scale=0.7]{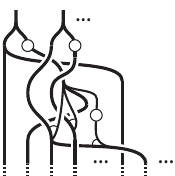}
\end{matrix}\!\!
= 
\begin{matrix}
\includegraphics[scale=0.7]{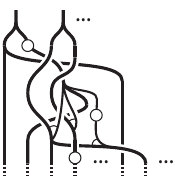}
\end{matrix}\!\!
= 
\begin{matrix}
\includegraphics[scale=0.7]{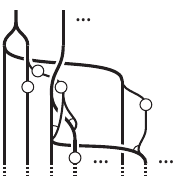}
\end{matrix}\!\!
= 
\begin{matrix}
\includegraphics[scale=0.7]{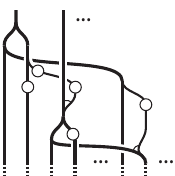}
\end{matrix}\!\!
= 
\begin{matrix}
\includegraphics[scale=0.7]{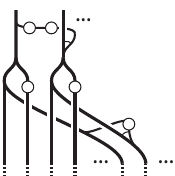}
\end{matrix}\ .
\end{multline*}
Therefore, \hfill{ }
\[
\begin{matrix}
\includegraphics[scale=0.8]{plat61}
\end{matrix}
- \ 
\begin{matrix}
\includegraphics[scale=0.8]{plat492}
\end{matrix}
 \ = \ 
 \begin{matrix}
\includegraphics[scale=0.8]{plat69}
\end{matrix}
- 
 \begin{matrix}
\includegraphics[scale=0.8]{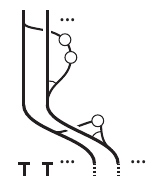}
\end{matrix}\ .
\]
\caption{Invariance for 
$b \, \sigma_2\, \sigma_1^2 \, \sigma_2
\longleftrightarrow b$.}
\label{fig:P32}
\end{figure}
Since $\mu_1 \circ S_2 \circ \Psi_2 \circ \ad_2$ is an isomorphism and $\mu_1 \circ S_2^2 \circ \Psi_2 \circ \ad_2$ is its inverse, 
Figure \ref{fig:P32} shows that $\mathcal{A}_{\sigma_2\sigma_1^2\sigma_2 b} = (\mu_1 \circ S_2^2 \circ \Psi_2 \circ \ad_2)(\mathcal{A}_b)$
and $\mathcal{A}_{\sigma_2\sigma_1^2\sigma_2 b}$ is isomorphic to $\mathcal{A}_b$.  
\par
{\bf P4.}
Let $f$ be the map from $\mathcal{T}_{k+1}$ to $\mathcal{T}_k$ given by $\mu_{k}\circ S_{k+1}$.  
\begin{figure}[htb]
\begin{multline*}
 \begin{matrix}
\includegraphics[scale=0.75]{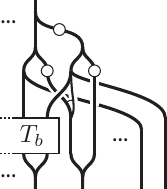}
\end{matrix}
\underset{\text{(bc)}}{=}
 \begin{matrix}
\includegraphics[scale=0.75]{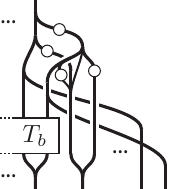}
\end{matrix}
=
 \begin{matrix}
\includegraphics[scale=0.75]{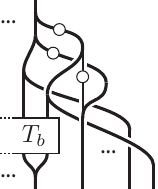}
\end{matrix}
= 
 \begin{matrix}
\includegraphics[scale=0.75]{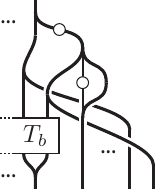}
\end{matrix}
=
 \begin{matrix}
\includegraphics[scale=0.75]{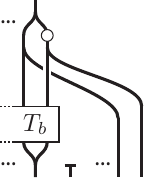}
\end{matrix}
=
 \begin{matrix}
\includegraphics[scale=0.75]{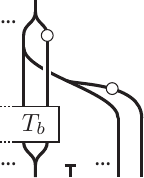}
\end{matrix}\ .
\end{multline*}
Therefore, \hfill{ }
\[
 \begin{matrix}
\includegraphics[scale=0.75]{plat71}
\end{matrix}
\ - \ 
 \begin{matrix}
\includegraphics[scale=0.75]{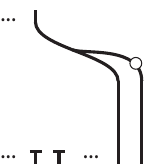}
\end{matrix}
\ = \ 
 \begin{matrix}
\includegraphics[scale=0.75]{plat76}
\end{matrix}
 \ - \ 
  \begin{matrix}
\includegraphics[scale=0.75]{plat77}
\end{matrix}\ .
\]
\caption{Invariance for $\sigma_{2k}\, b \longleftrightarrow b$.}
\label{fig:P4}
\end{figure}
Figure \ref{fig:P4} shows that $f\big(\Image (\widehat{T}_{\sigma_{2k}b}-\boldsymbol{\varepsilon}_{k+1})\big) = \Image (\widehat{T}_{b}-\boldsymbol{\varepsilon}_{k})$.  
We also show that $\ker f \subset \Image (\widehat{T}_{\sigma_{2k}b}-\boldsymbol{\varepsilon}_{k+1})$.  
Let $\bx \in \mathcal{T}_{k}$ and $\by \in \ker f$, then
Figure \ref{fig:P42} shows that $\widehat{T}_{\sigma_{2k}b}(\bx \otimes \by)=0$ since $\by \in \ker f$.  
\begin{figure}[htb]
\begin{multline*}
 \begin{matrix}
\includegraphics[scale=0.8]{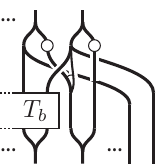}
\end{matrix}
=\ 
 \begin{matrix}
\includegraphics[scale=0.8]{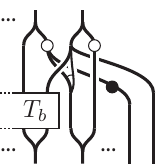}
\end{matrix}
=\ 
 \begin{matrix}
\includegraphics[scale=0.8]{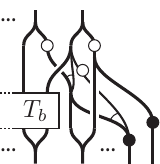}
\end{matrix}
=\ 
 \begin{matrix}
\includegraphics[scale=0.8]{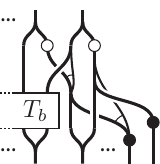}
\end{matrix}
\underset{\text{(bc)}}{=}
\\
\begin{matrix}
\includegraphics[scale=0.8]{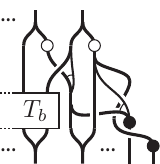}
\end{matrix}
\underset{\text{\bf HmR'}}{=}\ 
\begin{matrix}
\includegraphics[scale=0.8]{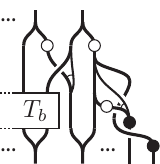}
\end{matrix}
=\ 
 \begin{matrix}
\includegraphics[scale=0.8]{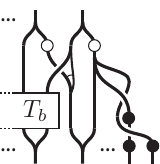}
\end{matrix}
=\ 
 \begin{matrix}
\includegraphics[scale=0.8]{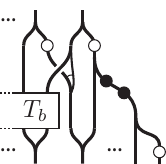}
\end{matrix}
\end{multline*}
\caption{Deformation of $\widehat{T}_{\sigma_{2k}b}$ to show that $\widehat{T}_{\sigma_{2k}b}(\bx\otimes \by)=0$ for $\by \in \ker f$.}
\label{fig:P42}
\end{figure}
Hence $(\widehat{T}_{\sigma_{2k}b} - \boldsymbol{\varepsilon_{k+1}})(\bx \otimes \by)
=
-\boldsymbol{\varepsilon_{k+1}}(\bx \otimes \by)
=
-\by$
and $\by$ is contained in the image of $\widehat{T}_{\sigma_{2k}b} - \boldsymbol{\varepsilon_{k+1}}$.
Since $f\big(\Image (\widehat{T}_{\sigma_{2k}b}-\boldsymbol{\varepsilon}_{k+1})\big) = \Image (\widehat{T}_{b}-\boldsymbol{\varepsilon}_{k})$ and $\ker f \subset \Image (\widehat{T}_{\sigma_{2k}b}-\boldsymbol{\varepsilon}_{k+1})$,  $\mathcal{T}_{k+1} / \Image(\widehat{T}_{b}-\boldsymbol{\varepsilon}_{k+1})$ is isomorphic to $\mathcal{T}_k/\Image(\widehat{T}_{b}-\boldsymbol{\varepsilon}_{k})$.  
\end{proof}
\begin{definition}
We call $\mathcal{A}_b$ {\it the universal space of representations} of $L$.   
\end{definition}
%
%
\subsection{Relation to closed braids}
Let  $L$ be a knot  isotopic to the closure of a braid $b$ contained in $B_k$.  
Let $A_b'$ be the module defined by $\mathcal{T}_k/\Image(T'_b-\boldsymbol{\varepsilon}_k)$ where $T'_b$ is the bottom tangle given by the braided Hopf diagram in Figure \ref{fig:closedbraid}.  
Then, the argument in \cite{MV} is valid for this bottom tangle case and $A'_b$ is an invariant of $L$.  
\begin{figure}[htb]
\[
T'_b \ :\ \  
 \begin{matrix}
\includegraphics[scale=1]{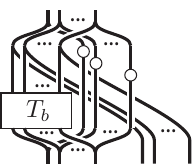}
\end{matrix}
\]
\caption{The bottom tangle $T'_b$ corresponding to the braid closure of $b$. }
\label{fig:closedbraid}
\end{figure}
\begin{thm}
If a knot $L$ is isotopic to the plat closure of a braid $b_1$ and is also isotopic to the braid  closure of $b_2$, then $\mathcal{A}_{b_1}$ and $\mathcal{A}'_{b_2}$ are isomorphic.  
\label{thm:braid}
\end{thm}
\begin{proof}
Let $b'_2\in B_{2k}$ be the braid in Figure \ref{fig:closedbraid2} whose plat closure is equal to the braid closure of $b_2$.  
\begin{figure}[htb]
\[
b_2 \ : \ \ 
 \begin{matrix}
\includegraphics[scale=1]{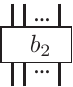}
\end{matrix}\in B_k
\ , 
\qquad
b'_2 \ : \ \ 
 \begin{matrix}
\includegraphics[scale=1]{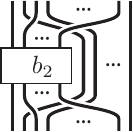}
\end{matrix}\in B_{2k}
\]
\caption{The braid $b'_2\in B_{2k}$ corresponding to $b_2 \in B_k$. }
\label{fig:closedbraid2}
\end{figure}
Then $\widehat{T}_{b'_2} \equiv T'_{b_2}$ modulo $\Image(T'_b-\boldsymbol{\varepsilon}_k)$ by Proposition 4.10 of \cite{MV} and so $\mathcal{A}_{b'_2}$ is isomorphic to $\mathcal{A}'_{b_2}$.   
Since $\mathcal{A}_{b_1}$ and $\mathcal{A}_{b'_2}$ are isomorphic, $\mathcal{A}_{b_1}$ and $\mathcal{A}'_{b_2}$ are also isomorphic. 
\end{proof}
%
%

For closed braids there is a conceptually simpler way of presenting the module. This will be important in matching up our approach to that of usual skein theory in the next section.
Instead of $J_{b}$ we could also consider the $\Ad$-comodule $\mathcal{J}_{b}$ of $\mathcal{T}_k$ given by the image of
\begin{equation}
\boldsymbol{\mu}^{(k)}\circ\big(\boldsymbol{\Psi}^{(k)}\big)^{-1}\circ\Big(id^{\otimes k}\otimes \big(T_b \circ\big(\boldsymbol{\theta}^{(k)}\big)^{-1}\circ \big(S^{-2}\big)^{\otimes k}\big)\Big) - \boldsymbol{\mu}^{(k)}
\label{eq:JJ}
\end{equation}
    where $\boldsymbol{\Psi}^{(k)}$ is the braiding of two bunches of $k$ strands and $\boldsymbol{\theta}^{(k)}$ is the positive full twist. The following theorem shows that $\mathcal{J}_{b}$ is isomorphic to $J_b$.
\begin{thm}
For any braid $b$ we have  $J_{b}\cong \mathcal{J}_{b}$ as $\Ad$-comodules of $\mathcal{T}_k$.
\end{thm}
\begin{proof}
This was proven in our previous paper, \cite{MV}.
\end{proof}

\subsection{$\Ad$ coinvariant space}
Let $b$ be a braid in $B_{2k}$.  
The image $\Ad(J_b)$ is contained in  
the image of 
$(\widehat{T}_b - \boldsymbol{\varepsilon_k})\otimes id$  
since $\Ad \circ(\widehat{T}_b - \boldsymbol{\varepsilon_k}) = \left((\widehat{T}_b - \boldsymbol{\varepsilon_k})\otimes id\right) \circ \Ad$ by Proposition \ref{prop:Ad}.  
Hence $\Ad$ induces a module map from $\mathcal{A}_b$ to $\mathcal{A}_b\otimes id$.
Let $\mathcal{A}_b^{\Ad}$ be the $\Ad$ coinvariantsubmodule of $\mathcal{A}_b$ defined by $\mathcal{A}_b^{\Ad} = \{x \in \mathcal{A}_b \mid \Ad(x) = x \otimes \eta\}$.  

\section{Skein algebras of  punctured disks}
In this section, we impose the Kauffman bracket skein relation to the algebra of bottom tangles $\mathcal{T}_k$.  
From now on, we assume that $\mathcal{R}$ is an integral domain with a distinguished invertible element $t$.  
%
\subsection{Kauffman bracket skein algebras}
We introduce a skein algebra based on the Kauffman bracket skein relation. 
\begin{definition}
Let 
\[
\mathcal S_{k,n} 
=
\mathcal{T}_{k,n}/{\sim}
\]
where $\sim$ is generated by the   Kauffman bracket skein relation and the circle relation, which are given by \eqref{eq:KBSR} and \eqref{eq:circle}.
Let $\mathcal{S}_k = \oplus_{n=1, 2, \cdots} \mathcal{S}_{k,n}$, then  $\mathcal{S}_k$ is an $\mathcal{R}$-algebra with the multiplication $\boldsymbol{\mu}$.  
We call $\mathcal{S}_{k}$ {\it the  skein algebra} of $D_k$.  
\end{definition}
%
\begin{align}
\text{Kauffman bracket skein relation :}
&\quad
\begin{matrix}
\epsfig{file=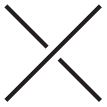, scale=0.6}
\end{matrix}
\ = \ 
t \, 
\begin{matrix}
\epsfig{file=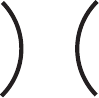, scale=0.6}
\end{matrix} + 
t^{-1} \,
\begin{matrix}
\epsfig{file=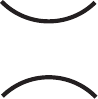, scale=0.6}
\end{matrix}
\label{eq:KBSR}
\\
\text{ Circle relation :}&\quad
\begin{matrix}
\epsfig{file=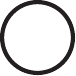, scale=0.8}
\end{matrix}
 = -(t^2 + t^{-2})
\label{eq:circle}
\end{align}
Note that the Kauffman bracket skein relation implies that
\[
\begin{matrix}
\epsfig{file=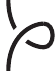, scale=0.8}
\end{matrix}
 = -t^3\ 
 \begin{matrix}
\epsfig{file=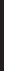, scale=0.8}
\end{matrix}\,
, 
\qquad
\begin{matrix}
\epsfig{file=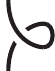, scale=0.8}
\end{matrix}
 = -t^{-3}\  
 \begin{matrix}
\epsfig{file=line, scale=0.8}
\end{matrix}\,
.
\]
%
The following are known for the skein algebra $\mathcal{S}_k$.   
\begin{thm}[\cite{PS}, Theorem 3, Corollary 4]
The skein algebra $\mathcal{S}_k$ is finitely generated, Noetherian, and an Ore domain.  
\end{thm}
%
%
%
\subsection{Quantum trace of $\mathcal{S}_{k,1}$}
\begin{definition}
For $x \in \mathcal{S}_{k,1}$, $\trace(x)$ is defined by
\begin{equation}
\trace(x) = -t^{2} \, x - t^{-2}\, (x \circ S).  
\label{eq:trace}
\end{equation}
\end{definition}
Let $x$ be an element of $\mathcal{S}_{k,1}$ corresponding to a bottom tangle $T$  in $\mathcal{T}_{k,1}$.  
Then $x\circ S$ is given by a bottom tangle in Figure \ref{fig:trace}.  
By applying the skein relation to $T \circ S$, 
it is decomposed as in  Figure \ref{fig:trace}, 
and $\trace(x)$ corresponds to the middle diagram of the figure. 
\begin{figure}[htb]
\[
\begin{matrix}
\begin{matrix}
\epsfig{file=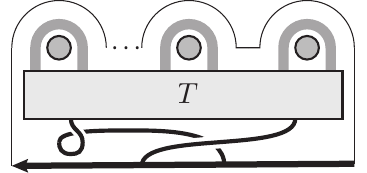, scale=0.6}
\end{matrix}
& = &
-t^{2} &
\begin{matrix}
\epsfig{file=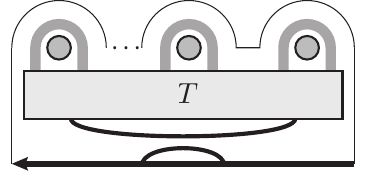, scale=0.6}
\end{matrix}
& 
- t^{4}&
\begin{matrix}
\epsfig{file=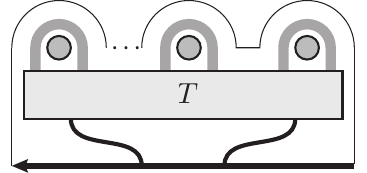, scale=0.6}
\end{matrix}
\\
x \circ S & & & \trace(x) & & x
\end{matrix}
\]
\caption{$x \circ S$ and $\trace(x)$ for $x \in \mathcal{S}_{k,1}$.}
\label{fig:trace}
\end{figure} 
\begin{prop}
For $x \in \mathcal{S}_{k,1}$, $\trace(x)$ is $\Ad$ coinvariant.  
\end{prop}
\begin{proof}
Let $x$ be an element of $\mathcal{S}_{k, 1}$ which corresponds to a bottom tangle $T$.  
By substituting the diagram of $\trace(x)$ in Figure \ref{fig:trace} to Figure \ref{fig:Adcommute}, we get  $\Ad(\trace(x)) = \trace(x) \otimes \eta$ as in Figure \ref{fig:Adtrace}.  
\end{proof}
\begin{figure}[htb]
\[
\begin{matrix}
\raisebox{1.5mm}{$\begin{matrix}
\epsfig{file=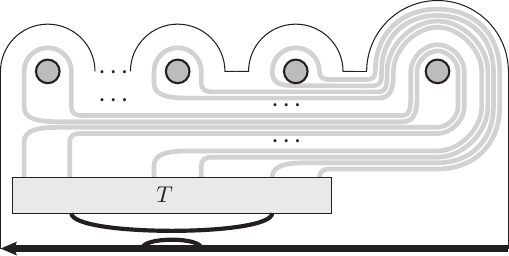, scale=0.7}
\end{matrix}$}
\ = \ 
\begin{matrix}
\epsfig{file=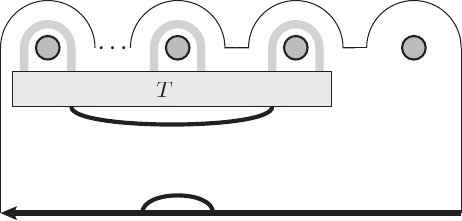, scale=0.7}
\end{matrix}
\\
\Ad(\trace(x))
\end{matrix}
\]
\caption{Graphical proof for $\Ad(\trace(x)) = \trace(x) \otimes \eta$.}
\label{fig:Adtrace}
\end{figure}
Let
$\mathcal{S}_{k, 1}^{\trace}$
be the image of $\trace$.
\subsection{Structure of $\mathcal{S}_{k,1}$}
The skein module $\mathcal{S}_{k,1}$ has an $\mathcal{R}$-algebra structure with the multiplication $m(x \otimes y) = \boldsymbol{\mu}\circ( x \otimes y) \circ \Delta$, and is a two-sided $\mathcal{S}_{k,1}^{\trace}$-module.  
Let $\alpha_i$ be the bottom tangle introduced in Figure \ref{fig:alpha}.  
\begin{lemma}
\label{lem:reductionformula}
The elements $\alpha_1$, $\alpha_2$, $\cdots$, $\alpha_k$ satisfy the following relations.  
\begin{align*}
\alpha_i^2 &= -t^{-4}-t^{-2}\, x_i\, \alpha_i ,
\\
\alpha_i^{-1} &= -t^{2}\, x_i-t^{4}\, \alpha_i,
\\
\alpha_j\, \alpha_i &= 
-t^{-4}\, x_i\, x_j- t^{-6} \, x_{ij}- 
t^{-2}\,  x_j\, \alpha_i  - t^{-2}\, x_i\, \alpha_j 
 -  t^{-4} \, \alpha_{ij} 
 \quad(i < j).
\end{align*}
\end{lemma}
These formulas are directly obtained from the skein relation.  
\begin{lemma}
$\mathcal{S}_{k,1}$ is spanned by words of $\alpha_1$, $\alpha_2$, $\cdots$, $\alpha_k$.  
\label{lem:generator}
\end{lemma}
\begin{proof}
By the Kauffman bracket skein relation,  $\mathcal{S}_{k,1}$ is spanned by skein diagrams without crossing.  
For a skein diagram $D\in \mathcal{S}_{k,1}$ having no crossing point, $D = D_0^{(1)}D_0^{(2)}\cdots D_0^{(l)}D_1$ where $D_0^{(j)}$ is an element of  $\mathcal{S}_{k,1}^{\trace}$ consisting of one closed ribbon and $D_1$ consists of one ribbon without crossing point.  
Then $D_1$ are expressed by a word of $\alpha_1$, $\alpha_2$, $\cdots$, $\alpha_k$ times $(-t^3)^\delta$ where $\delta$ comes from the twists in the diagram corresponding to the word.  
Note that 
$\alpha_i^{-1}=S(\alpha_i)$
since
$S(\alpha_i)$ is the inverse of $\alpha_i$ in $\mathcal{S}_{k,1}$.   
For example, let $D_1$ be a skein diagram in Figure \ref{fig:word}.  
\begin{figure}[htb]
\begin{align*}
D_1 &= 
\begin{matrix}
\epsfig{file=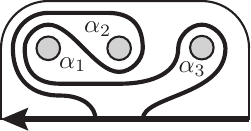, scale=0.8}
\end{matrix}\ ,
\\
\alpha_1\,\alpha_2\, \alpha_1^{-1}\,\alpha_3
&=
\begin{matrix}
\epsfig{file=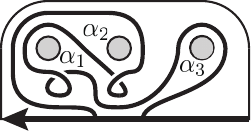, scale=0.8}
\end{matrix}
=
\begin{matrix}
\epsfig{file=freearcs40, scale=0.8}
\end{matrix}
=
 D_1.
\end{align*}
\caption{The word corresponding to a ribbon without crossing point.}
\label{fig:word}
\end{figure}
Starting from the left end point of the ribbon of $D_1$, and see how  the ribbon wind around the punctures.  
The ribbon first turn around the leftmost puncture clockwise, and it corresponds to $\alpha_1$.  
Next, the ribbon turn around the middle puncture clockwise, and it corresponds to $\alpha_2$.  
Then the ribbon turn around the leftmost puncture counterclockwise, and it corresponds to $\alpha_1^{-1}$. 
Lastly, the ribbon turn around the rightmost puncture clockwise, and it corresponds to $\alpha_3$. 
So the word corresponding to $D_1$ is $\alpha_1^{-1}\,\alpha_2\, \alpha_1^{-1}\,\alpha_3$.  
In this case, $\delta$ corresponding to the twists in the right diagram of Figure \ref{fig:word} is $0$ since the two twists in the diagram cancel each other.  
For any skein diagram in $\mathcal{S}_{k, 1}$ with one component and without crossing point, we can assign a word as the above example.  
\par
For $D_0^{(j)}$, there is a flat skein diagram $E_j \in \mathcal{S}_{k,1}$ such that $D_0^{(j)} = -t^2 \, E_j - t^{-2}\, (E_j \circ S)$ and $E_j$ is expressed by a word of $\alpha_1$, $\alpha_2$, $\cdots$, $\alpha_k$ times $(-t^3)^\delta$ as $D_1$.  
Hence $\mathcal{S}_{k,1}$ is spanned by words of $\alpha_1$, $\alpha_2$, $\cdots$, $\alpha_k$ as an $\mathcal{S}_{k,1}^{\trace}$-module.  
\end{proof}
Let  
\[
\alpha_{i_1i_2\cdots i_j} = \alpha_{i_1}\alpha_{i_2}\cdots\alpha_{i_j}
, \qquad
x_{i_1i_2\cdots i_j} = \trace(\alpha_{i_1i_2\cdots i_j})
\]
for $i_1 < i_2 < \cdots < i_j$.
%
Then Lemma \ref{lem:generator} and Lemma \ref{lem:reductionformula} imply the following. 
\begin{prop}
\label{prop:span}
$\mathcal{S}_{k,1}$ is spanned by $\alpha_{i_1i_2\cdots i_j}$  for $i_1 < i_2 < \cdots < i_j$ as a left 
$\mathcal{S}_{k,1}^{\trace}$-module.  
\end{prop}
\subsection{Generators of $\mathcal{S}_{k,1}^{\trace}$}
The $\mathcal{R}$-algebra $\mathcal{S}_{k,1}^{\trace}$ is generated by $x_{i_1\cdots i_j}$.  

\begin{prop}
\label{prop:spantrace}
Assuming we can invert $(t^2+t^{-2})$ in $\mathcal{R}$,
the algebra $\mathcal{S}_{k,1}^{\trace}$ is generated as an $\mathcal{R}$-algebra by $x_{i_1\cdots i_j}$ for $1 \leq i_1 < \cdots < i_j$, $1 \leq j \leq 3$.
\end{prop}
\begin{proof}
We follow Corollary 4.1.2 in \cite{GM} which is reformulated by using skein theory in \cite{B}.  
Let $\pi$ be the projection from $\mathcal{T}_k$ to $\mathcal{S}_k$. 
Any element of $\mathcal{S}_{k,1}^{\trace}$ is a linear combination of elements in the image of the map $\trace\circ\,\pi : \mathcal{T}_{k,1}^F \to \mathcal{S}_{k,1}^{\trace}$,  Lemma \ref{lem:generator}, Lemma \ref{lem:reductionformula} and Proposition \ref{prop:span} imply that
$x_{i_1i_2\cdots i_l}$ generate $\mathcal{S}_{k,1}^{\trace}$.  
Let $1 \leq i < j < l < m_1 < m_2 < \cdots < m_j$ and $\xi$ is the sequence $m_1$, $m_2$, $\cdots$, $m_j$. 
By using the skein relation, we have
\begin{multline*}
x_{il}\, x_{j\xi}
= (t^2+t^{-2})x_{ijl\xi}+\\
x_{jl\xi}\,x_{i}+x_{il\xi}\,x_{j}+x_{ij\xi}\,x_{l}+x_{ijl}\,x_{\xi}+
t^{-4} x_{i\xi}\,x_{jl}+
t^{4}x_{ij}\,x_{l\xi}+\\
t^{2}x_{i}\,x_j\,x_{l\xi}+t^{2}x_{l}\,x_{\xi}\,x_{ij}+
t^{-2}x_{i}\,x_{\xi}\,x_{jl}+t^{-2}x_{j}\,x_{l}\,x_{i\xi}+
x_i\,x_j\,x_l\,x_{\xi}.
\end{multline*}
%
By using this relation repeatedly, $x_{ijk\xi}$ is expressed by the quantum traces  $x_{\cdot\cdot\cdot}$ with three or less indices.  
\end{proof}
We also know the followings.
\begin{prop}
The elements $x_1$, $x_2$, $\cdots$, $x_k$  commute with any element of $\mathcal{S}_{k,1}$.  
\end{prop}
\begin{proof}
For any $y \in \mathcal{S}_{k,1}$, $x_j\, y=y\, x_j$ as explained in Figure \ref{fig:xjcommute}.  
\end{proof}
\begin{figure}[htb]
\begin{multline*}
m\left(
\begin{matrix}
\includegraphics[scale=0.7]{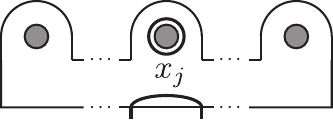}
\end{matrix}\ , 
\ \ 
\begin{matrix}
\includegraphics[scale=0.7]{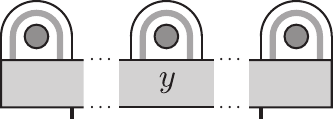}
\end{matrix}
\right)
\ = \ 
\begin{matrix}
\includegraphics[scale=0.7]{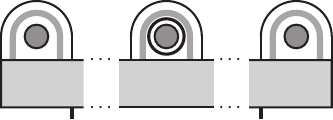}
\end{matrix}
\\
= \ m\left(
\begin{matrix}
\includegraphics[scale=0.7]{longitude20}
\end{matrix}\ , 
\ \ 
\begin{matrix}
\includegraphics[scale=0.7]{longitude22}
\end{matrix}
\right)\ .
\end{multline*}
\caption{The trace $x_j = \trace(\alpha_j)$ commutes with any element of $\mathcal{S}_{k,1}$ with respect to the multiplication $m$.}
\label{fig:xjcommute}
\end{figure}
\begin{prop}
The trace $x_j$ satisfies $S\circ x_j = x_j$.  
\end{prop}
\begin{proof}
The proof is given in Figure \ref{fig:Sxj}.  
\end{proof}
\begin{figure}[htb]
\[
S \circ x_j 
\ = \ 
\begin{matrix}
\includegraphics[scale=0.7]{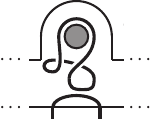}
\end{matrix}
\ =\ 
(-t^3)(-t^{-3})\,
\begin{matrix}
\includegraphics[scale=0.7]{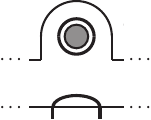}
\end{matrix}
\ =\ 
\begin{matrix}
\includegraphics[scale=0.7]{trace3}
\end{matrix}
\ = \ 
x_j.
\]
\caption{Proof of $S\circ x_j = x_j$.}
\label{fig:Sxj}
\end{figure}
\subsection{Structure of $\mathcal{S}_{k.n}$}
\begin{prop}
\label{prop:tensor}
The  inclusion map $\boldsymbol{\mu} : \underset{n}{\underbrace{\mathcal{T}_{k,1} \otimes \cdots \otimes \mathcal{T}_{k,1}}} \to \mathcal{T}_{k,n}$
induces a surjection 
\[
\overline{\boldsymbol{\mu} } : \underset{n}{\underbrace{\mathcal{S}_{k,1} \otimes \cdots \otimes \mathcal{S}_{k,1}}} \to \mathcal{S}_{k,n}.
\]  
\end{prop}
\begin{proof}
It is proved by an induction on $n$ and the number of crossings of a diagram in $\mathcal{S}_{k, n}$.  
From the Kauffman bracket skein relation, we have
\begin{equation}
t \, 
\begin{matrix}
\epsfig{file=skein1, scale=0.5}
\end{matrix}
\ -\ 
t^{-1} \ 
\begin{matrix}
\epsfig{file=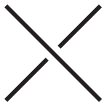, scale=0.5}
\end{matrix}
\ = \ 
(t^2 - t^{-2}) \ 
\begin{matrix}
\epsfig{file=skein2, scale=0.5}
\end{matrix}\ .
\label{eq:skein}
\end{equation}
Let $D$ be a skein diagram in $\mathcal{S}_{k, n}$ and $s$ be a string one of whose end points is the rightmost point at the bottom line.  
If $s$ passes  over at a crossing point $c$, 
let $D^u$ be the diagram which is obtained by applying the crossing change at $c$  to $s$.  
Then, by applying \eqref{eq:skein}, $D$ is expressed as a linear combination of $D^u$ and a skein  diagram whose number of crossing points is less than that of $s$.  
\par
Now we assume that $s$ passes under at all the crossing points.  
The end points of $s$ is as in Figure \ref{fig:end}.  
Let $D'$ is the diagram which is  obtained from $D$ by modifying as in Figure \ref{fig:end}.   
Then $D - t^{l}\, D'$ is a linear combination of diagrams of the form $x \otimes \eta$ where $x \in \mathcal{S}_{k, n-1}$ and $l$ is the number of end points between the two end points of $s$.  
The rightmost string $s$ of $D'$ passes under at all the crossings and there is a skein diagram $D_s \in \mathcal{S}_{k, 1}$ such that $D' = \boldsymbol{\mu}(({D\setminus s}) \otimes D_s))$ where $D\setminus s$ is the skein diagram obtained by removing $s$ from $D$.  
\begin{figure}[htb]
\[
D : 
\raisebox{-3mm}{$
\begin{matrix}
\epsfig{file=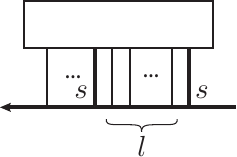, scale=0.8}
\end{matrix}$}
,
\qquad
D' : 
\begin{matrix}
\epsfig{file=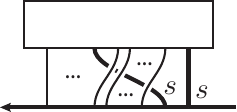, scale=0.8}
\end{matrix}
\]
\caption{Modifying $D$ to $D'$.}
\label{fig:end}
\end{figure}
\end{proof}
\subsection{Action of the bottom tangles on $\mathcal{S}_{k}$.}
Since the relations of the skein algebra are all local relations, we get the following.  
\begin{prop}
Let $T$ be a bottom tangle in $\mathcal{T}_{l,k}$.  
Then the $\mathcal{R}$-module map from $S_{k,n}$  to $S_{l,n}$  defined by composing  $T$ with elements of $S_{l,n}$  is well-defined. 
Moreover, if $T$ is a flat bottom tangle,  this map gives an $\mathcal{R}$-algebra homomorphism from $\mathcal{S}_k$ to $\mathcal{S}_l$.  
Especially, the bottom tangle corresponding to a braid with $k$ strings gives an $\mathcal{R}$-algebra automorphism of $\mathcal{S}_k$.  
\end{prop}
%
%
%
\section{$\BSL(2)$ character module of a knot}
In this section,  we introduce the space of $\BSL(2)$ representations and the $\BSL(2)$ character module.  
For simplicity, we assume that $\mathcal{R} = \mathbb{C}$ and $t$ be a generic element of $\mathbb{C}$, i.e. $t$ is not a root of any algebraic equation.  
\subsection{Space of representations coming from the skein algebra}
Let $\pi$ be the projection from $\mathcal{T}_k$ to $\mathcal{S}_k$ as before and
 $\widetilde J_b = \pi(J_b)$ where $J_b$ is the submodule introduced in \S \ref{sub:universal}.  
 \begin{prop}
 The submodule $\widetilde{J}_b$ is a two-sided  $\mathbb{C}[x_1, \cdots, x_k]$-module.  
 \end{prop}
 \begin{proof}
 For $y \in \widetilde{J}_b$, $x_j\, y$ is also contained in $\widetilde{J}_b$ 
 since $x_j$ commutes with any element. 
 See Figure \ref{fig:moduleJ}.  
 We also have $y \, x_j \in \widetilde{J}_b$ since $x_j\, y = y \, x_j$.  
 \end{proof}
 \begin{figure}[htb]
 \[
 \begin{matrix}
 \includegraphics[scale=0.8]{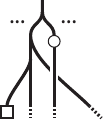}
 \end{matrix}
 \ - \ 
 \begin{matrix}
 \includegraphics[scale=0.8]{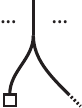}
 \end{matrix}
 \ = \ 
 \begin{matrix}
 \includegraphics[scale=0.8]{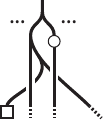}
 \end{matrix}
 \ - \ 
 \begin{matrix}
 \includegraphics[scale=0.8]{moduleJ0}
 \end{matrix}
 \ \in \ \widetilde{J}_b.
\]
 \caption{For $y \in \widetilde{J}_b$, $x_j\, y$ is also contained in $\widetilde{J}_b$. The box represents $x_j$.}
 \label{fig:moduleJ}
 \end{figure}
 \begin{definition}
 Let 
\begin{equation}
\widetilde{\mathcal{A}}_b = {\mathcal{S}}_k/\widetilde{J}_b.  
\label{eq:relation}
\end{equation}
We call $\widetilde{\mathcal{A}}_b$ {\it the space of $\BSL(2)$ representations} of the knot $\widehat{b}$. 
Here $\BSL(2)$ is the braided quantum group of $\SL(2)$ introduced by Majid in \cite{M91}.  
The relation to $\BSL(2)$ is given in Section \ref{sec:BSL}.  
 \end{definition}
The algebra $\widetilde{A}_b$ has a graded module structure
\[
\widetilde{A}_b = \bigoplus_{n=0, 1, 2, \cdots} \widetilde{A}_{b, n}
\]
where $\widetilde{A}_{b, n} = \mathcal{S}_{k,n}/\left(\mathcal{S}_{k,n}\cap \widetilde{J}_b\right)$.  
\begin{lemma}
\label{lem:skeingenerator}
The submodule $\widetilde{J}_b$ is spanned by the following set.

\[
\{(\widehat{T}_b-\boldsymbol{\varepsilon}_k)(y \otimes \boldsymbol{z}) \mid y \in \mathcal{S}_{k, 1}, \ 
\boldsymbol{z} \in \mathcal{T}_k\}.  
\]
\end{lemma}
\begin{proof}
$\widetilde{J}_b$ is by definition spanned by $(\widehat{T}_b-\boldsymbol{\varepsilon}_k)(\boldsymbol{y}\otimes \boldsymbol{z})$ 
$(\boldsymbol{y}, \boldsymbol{z} \in \mathcal{S}_k)$.  
Since $\overline{\boldsymbol{\mu}} : \otimes^n \mathcal{S}_{k,1}\to \mathcal{S}_{k,n}$ 
is surjective, it is enough to prove for the case that $\boldsymbol{y} = \overline{\boldsymbol{\mu}}(y_1\otimes y_2\otimes \cdots\otimes y_n)$ where $y_i
\in \mathcal{S}_{k,1}$ for $i =1, \cdots, n$.  
Let $\widetilde{J}'_b$ be the submodule spanned by $\{(\widehat{T}_b-\boldsymbol{\varepsilon}_k)(y \otimes \boldsymbol{z}) \mid y \in \mathcal{S}_{k, 1}, \ 
\boldsymbol{z} \in \mathcal{T}_k\}$.  
We show that 
$(\widehat{T}_b-\boldsymbol{\varepsilon}_k)(\mu(y_1 \otimes y_2) \otimes \boldsymbol{z})$ is contained in $\widetilde{J}'_b$ for $y_1$, $y_2 \in \mathcal{S}_{k, 1}$, $\boldsymbol{z} \in \mathcal{T}_k$.  
The following Figure \ref{fig:splitrelaion} shows that 
$
\widehat{T}_b(\mu(y_1 \otimes y_2) \otimes \boldsymbol{z} )
= 
\widehat{T}_b(y_1 \otimes \widehat{T}_b(y_2 \otimes \boldsymbol{z}) ).  
$
Hence 
\begin{multline*}
\widehat{T}_b(\mu(y_1 \otimes y_2) \otimes \boldsymbol{z})
- \boldsymbol{\varepsilon}_k(\mu(y_1 \otimes y_2) \otimes \boldsymbol{z})
=
\\
\widehat{T}_b(y_1 \otimes \widehat{T}_b(y_2 \otimes \boldsymbol{z}) )
- 
\boldsymbol{\varepsilon}_k(y_1 \otimes \widehat{T}_b(y_2 \otimes \boldsymbol{z}))
+
\\ {\ } \qquad\qquad\qquad\qquad\qquad\qquad
\boldsymbol{\varepsilon}_k(y_1 \otimes \widehat{T}_b(y_2 \otimes \boldsymbol{z}))
-
\boldsymbol{\varepsilon}_k(\mu(y_1 \otimes y_2) \otimes \boldsymbol{z})
=
\\
\widehat{T}_b(y_1 \otimes \widehat{T}_b(y_2 \otimes \boldsymbol{z}) )
- 
\boldsymbol{\varepsilon}_k(y_1 \otimes \widehat{T}_b(y_2 \otimes \boldsymbol{z}))
+
\\ 
\varepsilon^{\otimes k}(y_1)\, (\widehat{T}_b(y_2 \otimes \boldsymbol{z}))
-
\varepsilon^{\otimes k}(y_1)\, \boldsymbol{\varepsilon}_k(y_2 \otimes \boldsymbol{z})
\in \widetilde{J}'_b
\end{multline*}
since $(\widehat{T}_b-\boldsymbol{\varepsilon}_k)\big(y_2 \otimes \widehat{T}_b(y_2 \otimes \boldsymbol{z})\big)$, $(\widehat{T}_b-\boldsymbol{\varepsilon}_k)(y_1 \otimes \boldsymbol{z}) \in \widetilde{J}'_b$.  
Applying this relation repeatedly, we get
$
\widehat{T}_b(\overline{\boldsymbol{\mu}}(y_1 \otimes y_2\otimes \cdots\otimes y_n) \otimes \boldsymbol{z}) \in \widetilde{J}'_b
$.
\end{proof}
\begin{figure}[htb]
\begin{multline*}
\begin{matrix}
\includegraphics[scale=0.8]{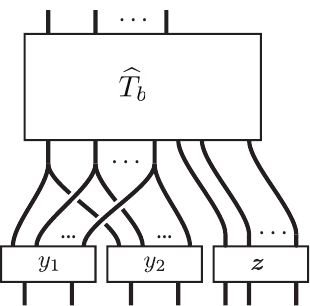}
\end{matrix}
=
\begin{matrix}
\includegraphics[scale=0.8]{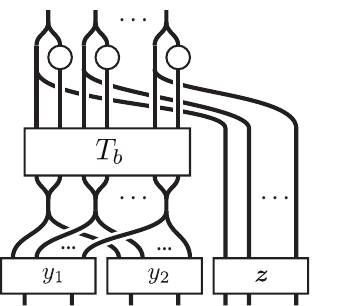}
\end{matrix}
=
\begin{matrix}
\includegraphics[scale=0.8]{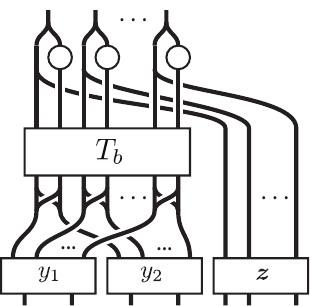}
\end{matrix}
\underset{Prop. \ref{prop:algaut}}{=}
\\
\begin{matrix}
\includegraphics[scale=0.8]{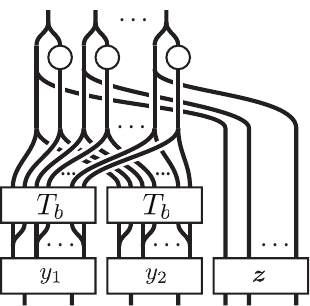}
\end{matrix}
=
\begin{matrix}
\includegraphics[scale=0.8]{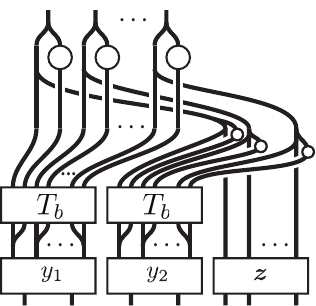}
\end{matrix}
=
\begin{matrix}
\includegraphics[scale=0.8]{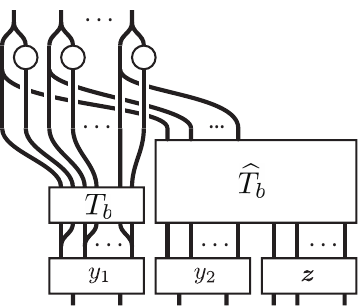}
\end{matrix}
\end{multline*}
\caption{Proof of $
\widehat{T}_b\big(\mu(x_1 \otimes x_2) \otimes \boldsymbol{y}\big)
= 
\widehat{T}_b\big(x_1 \otimes \widehat{T}_b(x_2 \otimes \boldsymbol{y} ) \big)
$.}
\label{fig:splitrelaion}
\end{figure}
\begin{prop}
\label{prop:generator}
The submodule $\widetilde{J}_b$ is spanned by the following set.
\[
\{(\widehat{T}_b-\boldsymbol{\varepsilon}_k)(\alpha_j^{\pm1} \otimes \boldsymbol{y}) \mid 1 \leq j \leq k, \ 
\boldsymbol{y} \in \mathcal{S}_k\}.  
\]
\end{prop}
\begin{proof}
Let $\widetilde{J}''_b$ be the subspace spanned by
$
\{(\widehat{T}_b-\boldsymbol{\varepsilon}_k)(\alpha_j^{\pm1} \otimes \boldsymbol{y}) \mid 1 \leq j \leq k, \ 
\boldsymbol{y} \in \mathcal{S}_k\}
$.  
We show that $\widetilde{J}'_b$ in Lemma \ref{lem:skeingenerator} is included in $\widetilde{J}''_b$.
For 
$
\alpha_{i_1}^{\varepsilon_1}\alpha_{i_2}^{\varepsilon_2}\cdots\alpha_{i_l}^{\varepsilon_{i_l}}
$ ($\varepsilon_1$, $\varepsilon_2$, $\cdots$, $\varepsilon_{i_l}=\pm1$),
\begin{multline*}
(\widehat{T}_b-\boldsymbol{\varepsilon}_k)
(\alpha_{i_1}^{\varepsilon_1}\alpha_{i_2}^{\varepsilon_2}\cdots\alpha_{i_l}^{\varepsilon_{i_l}} \otimes \boldsymbol{y})
=
(\widehat{T}_b-\boldsymbol{\varepsilon}_k)
\Big(\big(\overline{\boldsymbol{\mu}}(\alpha_{i_1}^{\varepsilon_1}\otimes\alpha_{i_2}^{\varepsilon_2}\cdots\alpha_{i_l}^{\varepsilon_{i_l}})\circ \Delta(id)\big)\otimes\boldsymbol{y}\Big)
\\
=
(\widehat{T}_b-\boldsymbol{\varepsilon}_k)
\big(\overline{\boldsymbol{\mu}}(\alpha_{i_1}^{\varepsilon_1}\otimes\alpha_{i_2}^{\varepsilon_2}\cdots\alpha_{i_l}^{\varepsilon_{i_l}})\otimes \boldsymbol{y}\big)\circ \big(\Delta(id)\otimes id^{\otimes k}\big)
\\
=
(\widehat{T}_b-\boldsymbol{\varepsilon}_k)
\Big(\alpha_{i_1}^{\varepsilon_1} \otimes(\widehat{T}_b-\boldsymbol{\varepsilon}_k)\big(\alpha_{i_2}^{\varepsilon_2}\cdots\alpha_{i_l}^{\varepsilon_{i_l}}\otimes \boldsymbol{y}\big)\Big)\circ \big(\Delta(id)\otimes id^{\otimes k}\big).
\end{multline*}
Hence $(\widehat{T}_b-\boldsymbol{\varepsilon}_k)
(\alpha_{i_1}^{\varepsilon_1}\alpha_{i_2}^{\varepsilon_2}\cdots\alpha_{i_l}^{\varepsilon_{i_l}}  \otimes \boldsymbol{y})
\in \widetilde{J}''_b$ and we get  $\widetilde{J}''_b = \widetilde{J}'_b$.
\end{proof}
\subsection{Quantum character module}
\begin{definition}
Let 
\[
\widetilde{\mathcal{A}}_{b, 1}^{\trace} = {\mathcal{S}}_{k,1}^{\trace}/(\widetilde{J}_{b}\cap{\mathcal{S}}_{k,1}^{\trace}).  
\]
We call $\widetilde{\mathcal{A}}_{b, 1}^{\trace}$ {\it the quantum character module} of $\widehat{b}$.  
\end{definition}

\begin{thm}
The quantum character module $\widetilde{\mathcal{A}}_{b, 1}^{\trace}$
is isomorphic to the skein module of the complement of $\widehat{b}$.
\end{thm}
\begin{proof}
We consider the closure of a braid $b$ and denote its complement in $S^3$ by $Y$. In this case $\mathcal{A}_{b, 1}$ is isomorphic to $\mathcal{S}_k/\widetilde{J}_b$, see Theorems \ref{thm:braid}. 
To make the connection to the skein module  $\Sigma$ of $Y$,  
define the map 
$f : \mathcal{S}_k^{\trace} \to \Sigma$ 
as follows. Choose a $k$ punctured disk $D\subset Y$ transversal to $b$ so that the strands of $b$ pass through the punctures of $D$. Using $D$ we can embed the bottom tangles into $Y$ and interpret them as skein elements. $f$ is surjective since we can use isotopies to bring all skeins of $Y$ into $D$. To finish the proof we need to make sure that the kernel of $f$ equals $\widetilde{J}_{b}\cap \mathcal{S}_{k,1}^{\trace}$.
\end{proof}

Note that  $\widetilde{A}_{b,1}$ is an $\widetilde{\mathcal{A}}_{b, 1}^{\trace} $-module.  
Let $\tau$ be the group homomorphism from $B_{2k}$ to $S_{2k}$, the symmetric group of degree  $2k$. 
Then $\overline{\boldsymbol{\mu}}_k(x_i\otimes T_b) = T_b \circ \overline{\boldsymbol{\mu}}_k(id^{\otimes k} \otimes x_{\tau(b)(i)})$ since 
\[
\overline{\boldsymbol{\mu}}_k(x_i\otimes T_{\sigma_j}) = \begin{cases}
T_{\sigma_j} \circ \overline{\boldsymbol{\mu}}_k(id^{\otimes k} \otimes x_{i+1}), & (i = j)
\\
T_{\sigma_j} \circ \overline{\boldsymbol{\mu}}_k(id^{\otimes k} \otimes x_{i-1}), & (i = j+1)
\\
T_{\sigma_j} \circ \overline{\boldsymbol{\mu}}_k(id^{\otimes k} \otimes x_{i}). & (\text{otherwise})
\end{cases}
\]
These relations come from 
\[
(\mathrm{ad} \circ \mu )\big(id \otimes \trace(id)\big) = \big((\mu\otimes id) \circ (id \otimes \Psi)\circ (\mathrm{ad}\otimes id)\big)\big(id \otimes \trace(id)\big),
\]
which is obtained by the second relation of the bottom row of Figure \ref{fig:propad} and the relation $\ad \circ \trace = \trace \otimes \eta$ explained in Figure
\ref{fig:permutation}.
\begin{figure}[htb]
\[
\begin{matrix}
\includegraphics[scale=0.6]{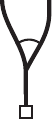}
\end{matrix}
\ \ =\ \ 
\begin{matrix}
\includegraphics[scale=0.8]{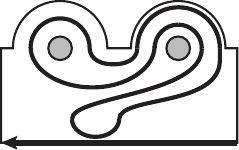}
\end{matrix}
\ \ =\ \ 
\begin{matrix}
\includegraphics[scale=0.8]{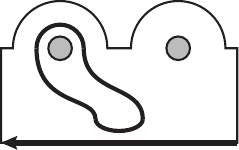}
\end{matrix}
\ \ =\ \ 
\begin{matrix}
\includegraphics[scale=0.8]{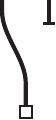}
\end{matrix}\ .
\qquad
\begin{matrix}
\\{}\\{}\\
\begin{matrix}
\includegraphics[scale=0.8]{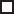}
\end{matrix}
=
\trace
\end{matrix}
\]
\caption{$\ad \circ \trace = \trace \otimes \eta$}
\label{fig:permutation}
\end{figure}
Therefore, if $\widehat b$ is a knot, then $\pi({x}_i) -\pi({x}_j) \in \widetilde{J}_b$ for any $1 \leq i, j \leq k$ and so $\pi(x_i)$ and $\pi(x_j)$ are the same element in $\widetilde{\mathcal{A}}_{b, 1}^{\trace}$.  
We put $x =\pi(x_1) = \pi(x_2) = \cdots = \pi(x_k)$.  
\par
Here we summarize some properties of $\trace$.  
\begin{figure}[htb]
\[
\begin{matrix}
\includegraphics[scale=0.8]{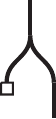}
\end{matrix}
\ = \ 
\begin{matrix}
\includegraphics[scale=0.8]{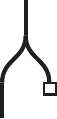}
\end{matrix} \ , 
\qquad
\begin{matrix}
\includegraphics[scale=0.8]{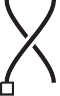}
\end{matrix}
\ = \ 
\begin{matrix}
\includegraphics[scale=0.8]{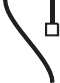}
\end{matrix} \ , 
\qquad
\begin{matrix}
\includegraphics[scale=0.8]{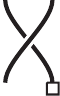}
\end{matrix}
\ = \ 
\begin{matrix}
\includegraphics[scale=0.8]{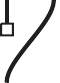}
\end{matrix} \ .
\]
\caption{Properties of $\trace$.}
\end{figure}
%
%
%
\subsection{$\widetilde{\mathcal{A}}_{b, 1}^{\trace}$-module structure of $\widetilde{\mathcal{A}}_{b, n}$}
Proposition \ref{prop:span} says  that the submodule 
 ${\mathcal{S}}_{k,1}$ of ${\mathcal{S}}_{k}$ is spanned by  $1=\varepsilon\otimes\eta^k$, $\alpha_{i_1}\alpha_{i_2}\cdots\alpha_{i_j}$
 $(i_1 < i_2 < \cdots < i_j)$
as an ${\mathcal{S}}_{k,1}^{\trace}$-module.  
Therefore, $\widetilde{\mathcal{A}}_{b, 1}$ is spanned by the images of $1=\varepsilon\otimes\eta^k$, $\alpha_{i_1}\alpha_{i_2}\cdots\alpha_{i_j}$
 $(i_1 < i_2 < \cdots < i_j)$ in $\widetilde{\mathcal{A}}_{b, 1}$
as an ${\mathcal{A}}_{k,1}^{\trace}$-module.  
\par
By Proposition \ref{prop:tensor}, there is a surjection 
$\overline{\boldsymbol{\mu} } : \otimes^n\mathcal{S}_{k,1} \to \mathcal{S}_{k,n}$,
which induces a surjection
$\widetilde{\boldsymbol{\mu} } : \otimes^n\widetilde{\mathcal{A}}_{b,1}\to \widetilde{\mathcal{A}}_{b,n}
$.  
Therefore, $\widetilde{\mathcal{A}}_{b,n}$ is spanned by the elements of the form $\widetilde{\mu}(a_1 \otimes \cdots \otimes a_n)$ as an ${\mathcal{A}}_{k,1}^{\trace}$-module where $a_j$ is one of $1=\varepsilon\otimes\eta^k$, $\alpha_{i_1}\alpha_{i_2}\cdots\alpha_{i_l}$
 $(i_1 < i_2 < \cdots < i_l)$.  
%
%
\subsection{Action of the skein algebra of thickened torus}
Let $b$ be an element of $B_{2k}$ whose plat closure $\widehat{b}$ is a knot.  
The skein module $\widetilde{\mathcal{A}}_b$ is isomorphic to the skein module of the complement of $\widehat{b}$, which is also isomorphic to the skein module of $S^3 \setminus N(\widehat{b})$ where $N(\widehat{b})$ is the regular neighborhood of $\widehat{b}$.  
The boundary $\partial (S^3 \setminus N(\widehat{b}))$ is isomorphic to a torus.  
Let $\mathcal{Q}$ be the  skein algebra of thickened torus.  
Let $u$, $v$ be skein elements of  $\widetilde{\mathcal{A}}_b$ and $\mathcal{Q}$ 
 respectively.  
By attaching $v$ to $u$ along $\partial (S^3 \setminus N(\widehat{b}))$, we get a skein element $u\cdot v \in \widetilde{\mathcal{A}}_b$.  
The product of two skein elements in $\mathcal{Q}$ is given by glueing two thickened tori at the outside of the one and the inside of the another one.  
\par
The structure of $\mathcal{Q}$ is given in \cite{FG} by using the quantum torus.  
Let $\widetilde{\mathcal{Q}}$ be the algebra of quantum torus which is given by
\[
\widetilde{\mathcal{Q}}
=
\mathbb{C}(\mu)\left< \lambda, \lambda^{-1}\right>
/ (\lambda\, \mu = t^2\, \mu\, \lambda).
 \]
 This comes from the ring of the quantum torus with localization by non-zero polynomials in $\mu$.  
 Let $\Theta$ be the involution of $\widetilde{\mathcal{Q}}$ sending $\lambda$ to $\lambda^{-1}$ and $\mu$ to $\mu^{-1}$.  
 Let $\widetilde{\mathcal{Q}}^\Theta$ be the subset of $\Theta$ invariant elements of $\widetilde{\mathcal{Q}}$ without denominator.  
 Then $\widetilde{\mathcal{Q}}^\Theta$ is a subalgebra of 
 $\widetilde{\mathcal{Q}}$ and is isomorphic to 
 $\mathcal{Q}$, which is generated by $x_l = -\lambda-\lambda^{-1}$, $x_m = -\mu-\mu^{-1}$ and $x_{lm} = t\,(\lambda\mu^{-1} + \mu\lambda^{-1})$. 
These generators  correspond to the skein elements of the thickened torus given in Figure \ref{fig:torusgenerator}.  
 Here we add  signs to $\lambda$ and $\mu$ to adjust them so that they mach the eigenvalues of the longitude and meridian in the classical case.  
The generators $x_l$ and $x_m$ represent  the longitude ant  the meridian respectively.  
\begin{figure}[htb]
\[
\begin{matrix}
\begin{matrix}
\includegraphics[scale=0.8]{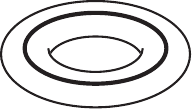}
\end{matrix}
&\qquad &
\begin{matrix}
\includegraphics[scale=0.8]{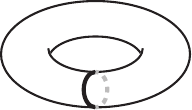}
\end{matrix}
&\qquad &
\begin{matrix}
\includegraphics[scale=0.8]{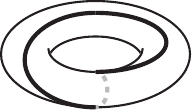}
\end{matrix}
\\
x_l & & x_m & & x_{lm}
\end{matrix}
\]
\caption{Generators of $\mathcal{Q}$.}
\label{fig:torusgenerator}
\end{figure}
\subsection{Action of the longitude}
Here we explain the action of the longitude $x_l$ of $\mathcal{Q}$ to $\widetilde{\mathcal{A}}_b$ using the trefoil as an example.  
The longitude of the trefoil knot is explained in Figure \ref{fig:longitude}.  
The picture (e) presents the standard longitude $x_l$ whose linking number with the core of the closed braid is zero, which is obtained by adding three twist around the second leftmost hole.  The picture (e) presents the longitude $x_{lm}$ which has linking number one with $\widehat{b}$.  
The skein element of $\widetilde{\mathcal{A}}_b$ sits on the horizontal plane.  
\begin{figure}[htb]
\[
\begin{matrix}
\begin{matrix}
\includegraphics[scale=0.6]{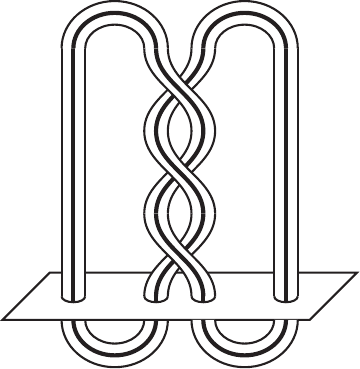}
\end{matrix}
&
\begin{matrix}
\includegraphics[scale=0.6]{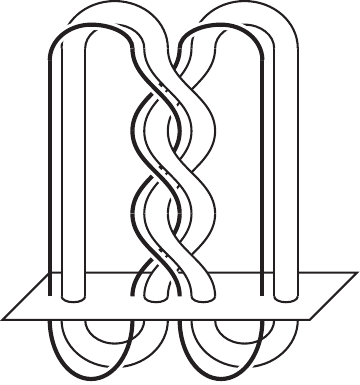}
\end{matrix}
&
\begin{matrix}
\includegraphics[scale=0.6]{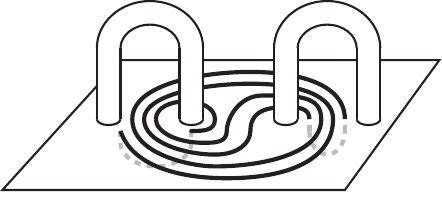}
\end{matrix}
\\
\text{(a) longitude}
&
\text{(b) pull forward a bit}
&
\begin{tabular}{l}
(c) push the middle part 
\\ \ \ downwards  to  the plane
\end{tabular}
\end{matrix}
\]
\[
\begin{matrix}
\begin{matrix}
\includegraphics[scale=0.6]{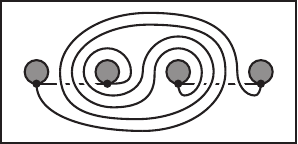}
\end{matrix}
& & \begin{matrix}
\includegraphics[scale=0.6]{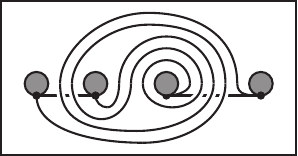}
\end{matrix}
& & 
\begin{matrix}
\includegraphics[scale=0.6]{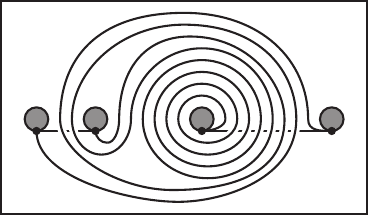}
\end{matrix}
& & 
\\
\text{(c) looking down (c)}
& & 
\text{(d) move a twist to right}
&  & 
\text{(e) $L_b$ :\ image of $x_l$}
&  &
\end{matrix}
\]
\[\begin{matrix}
\begin{matrix}
\includegraphics[scale=0.6]{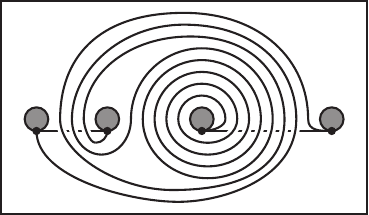}
\end{matrix}
& \qquad & 
\begin{matrix}
\includegraphics[scale=0.6]{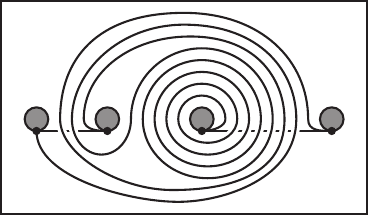}
\end{matrix}
\\
\text{(f) twisted longitude $x_{lm}$}
& &
\text{(g) $L_b'$ :\ image of $(-t^3)\, x_l$}
\end{matrix}\]
\caption{The action of the longitude $x_l$ of the trefoil.}
\label{fig:longitude}
\end{figure}
%
%
The action of the general case is explained in Figure \ref{fig:longitudegeneral}.  
The diagram $L_b$ representaing $x_l$ is obtained by applying the twists corresponding to the braid $b$ to $id^{\otimes 2k}$.
At the top of the picture, $2k$ holes are paired, and the antipode $S$ is applied to the right hole of each pair as  the right picture.    
The diagram $L_b'$ is obtained by one twist around a hole to $L_b$.   
\begin{figure}[htb]
\[
\begin{matrix}
\includegraphics[scale=0.8]{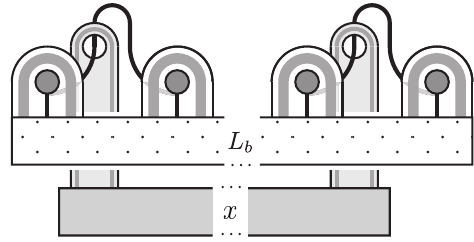}
\end{matrix}
\ = \ 
\begin{matrix}
\includegraphics[scale=0.8]{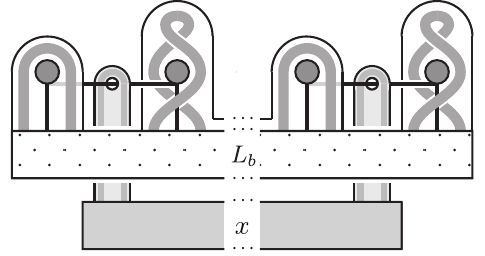}
\quad
x \in \widetilde{\mathcal{A}}_b
\end{matrix}
\]
\caption{The action of the longitude for general case.}
\label{fig:longitudegeneral}
\end{figure}
%
%
%
 \subsection{Colored Jones polynomial}
By removing the punctures from the image $L_b$ of  $x_l\cdot 1$, we get the  diagram of $\widehat{b}$.  
Similarly, $L_b'$ also corresponds to $\widehat{b}$.
Hence $\boldsymbol{\varepsilon}_k(L_b)$ and $\boldsymbol{\varepsilon}_k((-t^3)L_b')$ are both equal to the Jones polynomial of $\widehat{b}$.  
In general, we get the following.
\begin{lemma}
By removing the holes, we have
\[
\boldsymbol{\varepsilon}_k(x_l\cdot u) 
=
\boldsymbol{\varepsilon}_k((-t^3)x_{lm} \cdot u).  
\]
\label{lem:hole}
\end{lemma}
Let $K_n$ be $\boldsymbol{\varepsilon}_k(x_l^n\cdot 1)$ and $K'_n$ be $\boldsymbol{\varepsilon}_k((-t^3)x_{lm}x_l^{n-1}\cdot 1)$, then they are equal to the Jones polynomial of $n$ parallel version of $\widehat{b}$.  
The $n$ parallel version of a knot $K$ is obtained from $K$ by replacing the string by $n$ parallel strings so that the linking number of any two strings is zero.  
\par
Let $T_n(x)$ be a version of the 
Chebyshev polynomial defined by 
\[
T_n(x) = x\, T_{n-1}(x) - T_{n-2}(x), \qquad
T_0(x) = 1, \qquad T_1(x) = x.  
\]
Then, for $x_l = -\lambda - \lambda^{-1}$,  we have 
\begin{equation}
T_n(x_l) =(-1)^n ( \lambda^n + \lambda^{n-2} + \cdots + \lambda^{-n}).
\label{eq:chebyshev}
\end{equation}  
The following is well-known.  
\begin{prop}
The colored Jones polynomial $J_{\widehat{b}}^{(n)}(t)$ is obtained from the Jones polynomials of  $j$ parallels of $\widehat{b}$ as follows.  
\[
J_{\widehat{b}}^{(n)} =
\sum_{j=n, n-2, \dots, 0\, \text{or}\,1} \, c_j^{(n)}\, J_{\widehat{b}^{(j)}}(t)
\]
where $\widehat{b}^{(j)}$ is the $j$ parallel version of $\widehat{b}$, and
$c_j^{(n)}$ is the $j$-th coefficient of $T_n(x)$.  
In other words, $J_{\widehat{b}}^{(n)}$ is given by
\begin{equation}
J_{\widehat{b}}^{(n)}(t) = 
\boldsymbol{\varepsilon}_k\left(T_n(x_l)\right) = 
(-1)^n\sum_{j=n, n-2, \dots, 0\, \text{or}\, 1}c_j^{(n)}\, \boldsymbol{\varepsilon}_k (x_l^j \cdot 1).  
\label{eq:coloredJones}
\end{equation}
\end{prop}
\subsection{Values of the counit}
We extend the algebra $\widetilde{\mathcal{A}}_b$ by $\widetilde{\mathcal{Q}} \otimes_{\mathcal{Q}}\widetilde{\mathcal{A}}_b$ and extend the counit $\boldsymbol{\varepsilon}_k$ to this extended algebra.
\par
For $x_m$, $x_m\cdot u = x\, u$ and 
$\boldsymbol{\varepsilon}_k(x_m \cdot u) = 
\boldsymbol{\varepsilon}_k(x)\boldsymbol{\varepsilon}_k(u)$.  
So 
we assign $\boldsymbol{\varepsilon}_k(\mu \cdot u) =  \boldsymbol{\varepsilon}_k(\mu) \boldsymbol{\varepsilon}_k( u)$ and $\boldsymbol{\varepsilon}_k(\mu^{-1} \cdot u) =  \boldsymbol{\varepsilon}_k(\mu^{-1}) \boldsymbol{\varepsilon}_k( u)$. 
Since $\mu \, \mu^{-1} = 1$ and $-\mu- \mu^{-1} = x$,  
 $\boldsymbol{\varepsilon}_k(\mu)\boldsymbol{\varepsilon}_k(\mu^{-1}) = \boldsymbol{\varepsilon}_k(1) = 1$ and $\boldsymbol{\varepsilon}_k(\mu) + \boldsymbol{\varepsilon}_k(\mu) ^{-1}=\boldsymbol{\varepsilon}_k(-x) = t^2 + t^{-2}$.  
Therefore $\boldsymbol{\varepsilon}_k(\mu)$ must be $t^2$ or $t^{-2}$.
Here we choose  $\boldsymbol{\varepsilon}_k(\mu) = t^2$.  
\par
Next, we investigate $\boldsymbol{\varepsilon}_k(\lambda^n\cdot 1)$.  
\begin{lemma}
For $\lambda$, we have
$\boldsymbol{\varepsilon}_k(\lambda^n\cdot 1) =
-\boldsymbol{\varepsilon}_k(\lambda^{-n-2}\cdot 1)$. 
\label{lem:inversion}
\end{lemma}
\begin{proof}
Since 
$\boldsymbol{\varepsilon}_k\big(T_n(x_l)\cdot 1 - T_{n-2}(x_l)\cdot 1\big) 
= 
\boldsymbol{\varepsilon}_k\Big(x_l \cdot\big( T_{n-1}(x_l)\cdot 1\big)\Big)
=
\boldsymbol{\varepsilon}_k\Big((-t^3)x_{lm} \cdot\big( T_{n-1}(x_l)\cdot 1\big)\Big)$ 
by Lemma \ref{lem:hole}, 
$x_l = -\lambda-\lambda^{-1}$, $x_{lm}= t(\lambda\mu^{-1} +\lambda^{-1}\mu)$,  and \eqref{eq:chebyshev},   
we have
\begin{align*}
&\boldsymbol{\varepsilon}_k\Big(\big(x_l \, T_{n-1}(x_l) -
 (-t^3)x_{lm} \, T_{n-1}(x_l)\big)\cdot 1\Big) 
 = 
 \\
&\qquad
(-1)^n\boldsymbol{\varepsilon}_k\Big(\big(
\lambda^n + 2\, \lambda^{n-2} + 2\, \lambda^{n-4} + \cdots +
2 \, \lambda^{-n+2} + \lambda^{-n}\big)- 
\\
&\qquad\qquad
t^4\big(
(\lambda \mu^{-1}+ \lambda^{-1}\mu)\lambda^{n-1} + ( \lambda\mu^{-1} + \lambda^{-1}\mu) \lambda^{n-3} +  \cdots 
+ (\lambda\mu^{-1} + \lambda^{-1}\mu)  \lambda^{-n+1}\big)\cdot 1\Big)
=
\\
&\qquad
(-1)^n\boldsymbol{\varepsilon}_k\bigg(\Big(\big(
\lambda^n + 2\, \lambda^{n-2} + 2\, \lambda^{n-4} + \cdots +
2 \, \lambda^{-n+2} + \lambda^{-n}\big) - 
\\
&\qquad\qquad\qquad\qquad
t^2\big(
\mu^{-1}\lambda^{n} + (\mu +\mu^{-1})\lambda^{n-2}  + \cdots
+
(\mu + \mu^{-1})  \lambda^{-n+2} + \mu \lambda^{-n}\big)\Big)\cdot1 \bigg)
=  
\\
&\qquad
(-1)^n\boldsymbol{\varepsilon}_k\bigg(\Big(\big(
\lambda^n + 2\, \lambda^{n-2} + 2\, \lambda^{n-4} + \cdots +
2 \, \lambda^{-n+2} + \lambda^{-n}\big) - 
\\
&\qquad\qquad\qquad
\big(
\lambda^{n} + (1+t^4)\lambda^{n-2} + (1+t^4)  \lambda^{n-4} + \cdots 
+
(1+t^4)  \lambda^{-n+2} +t^4 \lambda^{-n}\big)\Big)\cdot 1\bigg)
=
\\
&\qquad\qquad\qquad\qquad\qquad\qquad\qquad
\quad
(-1)^n(1-t^4)\boldsymbol{\varepsilon}_k\big((\lambda^{n-2} + \lambda^{n-4} + \cdots + \lambda^{-n})\cdot 1)\big)
= 0.  
\end{align*}
Therefore, $\boldsymbol{\varepsilon}_k\big((\lambda^{n-2} + \lambda^{n-4} + \cdots + \lambda^{-n})\cdot 1\big) = 0$
for any integer $n$.  
This implies that 
$\boldsymbol{\varepsilon}_k\left(\lambda^{n}\cdot 1\right) =- \boldsymbol{\varepsilon}_k\left(\lambda^{-n-2}\cdot 1\right)$.  
\end{proof}
\begin{prop}
The colored Jones polynomial is given by
\[
J_{\widehat{b}}^{(n)}(t) =
(-1)^n\boldsymbol{\varepsilon}_k(\lambda^n\cdot 1).
\]
\end{prop}
\begin{proof}
By \eqref{eq:coloredJones},  \eqref{eq:chebyshev} and Lemma \ref{lem:inversion}, we have
$J_{\widehat{b}}^{(n)}(t) = \boldsymbol{\varepsilon}_k\big(T_n(x_l) \cdot 1\big)
=
\boldsymbol{\varepsilon}_k\big((-1)^n(\lambda^n + \lambda^{n-2} + \cdots \lambda^{-n})\cdot 1\big)
=
(-1)^n\boldsymbol{\varepsilon}_k\big(\lambda^n\cdot 1\big)
$.  
\end{proof}
 \subsection{Recurrence relation of the colored Jones polynomial}
 \label{ss:recurrence}
Let $\rho$ be the map sending $x \in \mathcal{Q}$ to $x \cdot 1\in \widetilde{\mathcal{A}}_b$ and $\mathcal{J} = \ker \rho$.   
In some literatures, $\mathcal{J}$ is called the peripheral ideal.  
Let $\widehat{\mathcal{A}}_b = \widetilde{\mathcal{Q}}\otimes_{{\mathcal{Q}}}\mathrm{Im} \rho$, 
then $\widehat{\mathcal{A}}_b$ is isomorphic to $\widetilde{\mathcal{Q}}/\widetilde{\mathcal{J}}$, where  
 $\widetilde{\mathcal{J}}$ is the left ideal in $\widetilde{\mathcal{Q}}$ generated by $\mathcal{J}$. 
 Since $\widetilde{\mathcal{Q}}$ is a PID, $\widetilde{\mathcal{J}}$ is generated by a single element, say $G(\lambda, \mu)$, which is contained in $\mathcal{Q}$.  
 The image of $G(\mu, \lambda)$ is zero in $\widetilde{\mathcal{A}}_b$, 
 the image of $\lambda^n G(\mu, \lambda)$ is also zero.  
 Then $\boldsymbol{\varepsilon}_k\big(\lambda^n G(\mu, \lambda)\big)$ is a linear combination of the colored Jones polynomial giving the recurrence relation of the colored Jones polynomial. 
 The AJ conjecture predicts that $G(\lambda, \mu)$ coincide with the generator of the annihilating ideal of the colored Jones polynomial for most cases.  
See \cite{Le} for the detail of the AJ conjecture.   
\section{Examples}
\label{sec:examples}
In this section, we consider  two bridge knots and links, which are plat closure of four braids and their space of representations are quotients of the skein module of a disk with two punctures.  
In this section, we assume that $t$ is generic, in other words, $t$ is not a root of any algebraic equation. 
\subsection{Skein module of a disk with two punctures}
It is known by \cite{P} that ${\mathcal{S}}_{2, 1}^{\trace}$ is a polynomial ring $\mathcal{R}[x_1, x_2, x_{12}]$.
Propositions \ref{prop:span} and \ref{prop:spantrace} imply that the skein algebra
${\mathcal{S}}_{2, 1} $ is a free left
${\mathcal{S}}_{2, 1}^{\trace}$-algebra with basis $1$, $\alpha_1$, $\alpha_2$ and $\alpha_{12}$.  
Now let us investigate the structure of 
the module
$\widetilde{\mathcal{A}}_{b, 1} = {\mathcal{S}}_{k,1}/\widetilde{J}_{b, 1}$.  
Proposition \ref{prop:generator} says that $\widetilde{J}_{b, 1}$ is generated by $(T_b-\boldsymbol{\varepsilon}_k)(\alpha_j^{\pm1}\otimes \boldsymbol{y})$ for $j = 1$, $2$ and $\boldsymbol{y}\in \mathcal{S}_{i,1}$.  
\subsection{Action of $\alpha_1$ and $\alpha_2$}
Here we  investigate the right action of $\alpha_1$ and $\alpha_2$ on the left ${\mathcal{S}}_{2,1}^{\trace}$-module ${\mathcal{S}}_{2,1}$.  
%
By using Lemma \ref{lem:reductionformula}, the right action of $\alpha_1$ is given by
\begin{align*}
1\, \alpha_1&= \alpha_1, \\
\alpha_1 \, \alpha_1 &= - t^{-4} -t^{-2} \, x_1 \, \alpha_1 , \\
\alpha_2\, \alpha_1 &= -t^{-4}\, x_1\, x_2- t^{-6} \, x_{12}- 
t^{-2}\,  x_2\, \alpha_1  - t^{-2}\, x_1\, \alpha_2 
 -  t^{-4} \, \alpha_{12} 
,
  \\
\alpha_{12}\, \alpha_1 &= 
t^{-2} \, x_2-  t^{-2}\, x_{12} \, \alpha_1 + \alpha_2 
.
\end{align*}
The matrix corresponding to the action of $\alpha_1$ is
\[
A_1 = 
\begin{pmatrix}
0 & 1 & 0 & 0 \\
-t^{-4} & -t^{-2} \, x_1 & 0 & 0 \\
-t^{-4} \,x_1\,x_2 - t^{-6}\, x_{12} & -t^{-2}\, x_2\ & -t^{-2}\,x_1 & -t^{-4}  \\
t^{-2}\, x_2 & -t^{-2}\, x_{12} & 1 & 0
\end{pmatrix}
\]
For the classical case ($t = -1$), such action is given by Proposition 3.1 in \cite{BH}.  
Similarly, the right action of $\alpha_2$ is given by
\begin{align*}
1\, \alpha_2  &= \alpha_2, \\
\alpha_1 \, \alpha_2 &= 
\alpha_{12}, \\
 \alpha_2 \, \alpha_2 &= -t^{-4}  -t^{-2}\, x_2 \, \alpha_2, \\
 \alpha_{12}\, \alpha_2 &=
 -t^{-4}\, \alpha_1  - t^{-2} \, x_2\, \alpha_{12},
\end{align*}
and the corresponding matrix is
\[
A_2=
\begin{pmatrix}
0 &0 & 1 & 0
\\
0 & 0 & 0 & 1 
\\
-t^{-4} & 0 & -t^{-2}\,x_2 & 0 
\\
0 & -t^{-4} & 0 & -t^{-2}x_2
\end{pmatrix}.
\]
\subsection{Two-bridge knots and links}
Let $b \in B_4$, then the plat closure $\widehat{b}$ is a two-bridge knot or link.  
By the relation in Figure \ref{fig:trefoil1}, we have
\[
\alpha_1\circ S^2 = \alpha_1.
\]
Since $\alpha_1 \circ S^2 = -t^4 \, \alpha_1\circ S - t^{2} \, x_1$, we have 
$
\alpha_1 =  -t^4 \, \alpha_1\circ S - t^{2} \, x_1
=
t^8 \, \alpha_1 + t^6 \, x_1  - t^2 \, x_1
$
and we get
\[
\alpha_1 = -\dfrac{1}{t^2 + t^{-2}}\, x_1 
\]
if $t^8-1 \neq 0$.  
Similarly, we have 
$\alpha_2 = -\frac{1}{t^2 + t^{-2}}\, x_2$.  
\begin{figure}[htb]
\begin{multline*}
-(t^2 + t^{-2})\, 
\begin{matrix}
\includegraphics[scale=0.8]{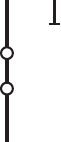}
\end{matrix}
\ = \ 
\begin{matrix}
\includegraphics[scale=0.8]{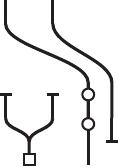}
\end{matrix}
\ \longleftrightarrow \ 
\begin{matrix}
\includegraphics[scale=0.8]{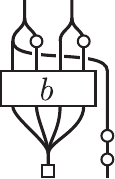}
\end{matrix}
\ = \
\begin{matrix}
\includegraphics[scale=0.8]{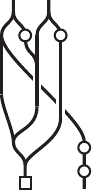}
\end{matrix}
\ = \
\\
\begin{matrix}
\includegraphics[scale=0.8]{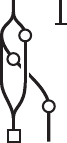}
\end{matrix}
\ = \
\begin{matrix}
\includegraphics[scale=0.8]{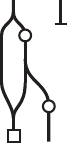}
\end{matrix}
\ \underset{HmR'}{=} \
\begin{matrix}
\includegraphics[scale=0.8]{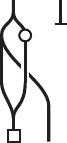}
\end{matrix}
\ \longleftrightarrow \
\begin{matrix}
\includegraphics[scale=0.8]{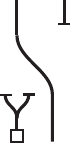}
\end{matrix}
\ = \
-(t^2 + t^{-2}) \, 
\begin{matrix}
\includegraphics[scale=0.8]{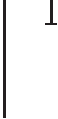}
\end{matrix}.
\end{multline*}
\caption{Relation for $\alpha_1$.}
\label{fig:trefoil1}
\end{figure}
\par
Next, we consider about $\alpha_1 \, x_{12}$.  
The relations given in Figure \ref{fig:alpha1x12} imply
\[
 x_{12}\, \alpha_1 = \alpha_1\, x_{12}. 
\]
\begin{figure}[htb]
\begin{multline*}
-(t^2 + t^{-2})\, 
\begin{matrix}
\includegraphics[scale=0.8]{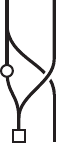}
\end{matrix}
\ = \ 
-(t^2 + t^{-2})\, 
\begin{matrix}
\includegraphics[scale=0.8]{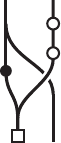}
\end{matrix}
\ = \ 
\begin{matrix}
\includegraphics[scale=0.8]{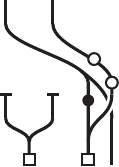}
\end{matrix}
\ = \ 
\begin{matrix}
\includegraphics[scale=0.8]{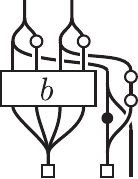}
\end{matrix}
\ = \ 
\begin{matrix}
\includegraphics[scale=0.8]{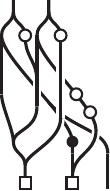}
\end{matrix}
\ = \ 
\\
\begin{matrix}
\includegraphics[scale=0.8]{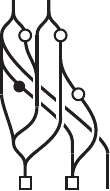}
\end{matrix}
\ = \ 
\begin{matrix}
\includegraphics[scale=0.8]{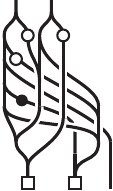}
\end{matrix}
\ = \ 
\begin{matrix}
\includegraphics[scale=0.8]{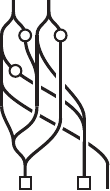}
\end{matrix}
\ = \ 
\begin{matrix}
\includegraphics[scale=0.8]{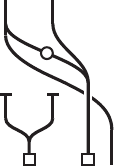}
\end{matrix}
\ = \ 
-(t^2 + t^{-2})\ 
\begin{matrix}
\includegraphics[scale=0.8]{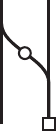}
\end{matrix} \ .
\end{multline*}
%
%
\caption{Explain the relation $x_{12}\, \alpha_1 = \alpha_1 \, x_{12}$.}
\label{fig:alpha1x12}
\end{figure}
Similarly, we get $x_{12}\, \alpha_2 =\alpha_2\, x_{12}$.  
These relations imply
\[
\alpha_{12}\, x_1 = \alpha_1 \, x_{12}, 
\qquad
\alpha_{12}\, x_2 = \alpha_2 \, x_{12}.  
\]
\subsection{Trivial knot}
Let $ b = \sigma_2\in B_4$.  
Then the plat closure of $b$ is the trivial knot, and $x_1 = x_2$. 
So we denote $x = x_1 = x_2$.  
As in Figure \ref{fig:trivialy}, we have
$
\alpha_{12} = 1$.  
\begin{figure}[htb]
\[
\begin{matrix}
\includegraphics[scale=0.8]{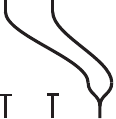}
\end{matrix}
\longleftrightarrow
\begin{matrix}
\includegraphics[scale=0.8]{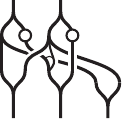}
\end{matrix}
=
\begin{matrix}
\includegraphics[scale=0.8]{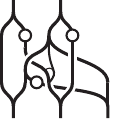}
\end{matrix}
\underset{\text{(bc)}}{=}
\begin{matrix}
\includegraphics[scale=0.8]{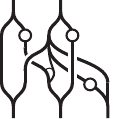}
\end{matrix}
=
\begin{matrix}
\includegraphics[scale=0.8]{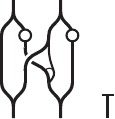}
\end{matrix}
\longleftrightarrow
\begin{matrix}
\includegraphics[scale=0.8]{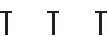}
\end{matrix}
\]
\caption{Relation of $\alpha_{12}$.}
\label{fig:trivialy}
\end{figure}
So $\widetilde{A}_{b, 1}$ is spanned by $1$, $x$, $x^2$, $\cdots$, and $\widetilde{A}_{b, 1}$ is isomorphic to $\widetilde{A}_{b}^{\trace}$.
\par
The longitude is given by (a) of Figure \ref{fig:trivial}, and the standard longitude is given by (b) of the same figure.  
Here we have a surjection from $\mathcal{Q}$ to $\widetilde A_{b}^{\trace}$
by sending $x \in \mathcal{Q}$ to $x \cdot 1 \in \widetilde A_{b}^{\trace}$.
Then the standard longitude is mapped to the trivial loop.  
\begin{figure}[htb]
\[
\begin{matrix}
\begin{matrix}
\includegraphics[scale=0.6]{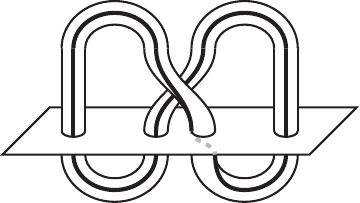}
\end{matrix}
&&
\begin{matrix}
\includegraphics[scale=0.6]{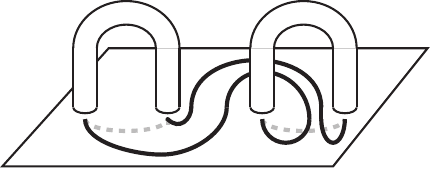}
\end{matrix}
&&
\begin{matrix}
\includegraphics[scale=0.8]{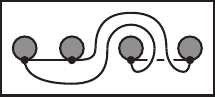}
\end{matrix}
\\
\text{standard}&& &&  \text{trivial}
\\
\text{longitude $x_l$}&& &&  \text{loop}
\end{matrix}
\]
\caption{Longitude of the trivial knot $\widehat b$.}
\label{fig:trivial}
\end{figure}
The image of $x_{lm}$ is given in Figure \ref{fig:lm} , which is equal to $-t^3 x$.  
The kernel of this map is the left ideal generated by $x_l - (-t^2 -t^{-2})$ and $x_{lm} +t^3\, x_m$.    
\begin{figure}[htb]
\[
\begin{matrix}
\begin{matrix}
\includegraphics[scale=0.6]{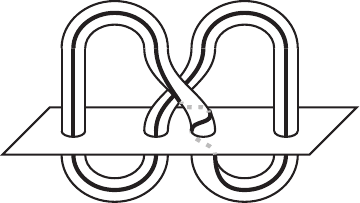}
\end{matrix}
&&
\begin{matrix}
\includegraphics[scale=0.6]{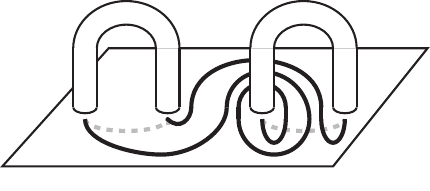}
\end{matrix}
&&
\begin{matrix}
\includegraphics[scale=0.8]{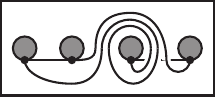}
\end{matrix}
\\
x_{lm} && && -t^3\,x
\end{matrix}
\]
\caption{The image of $x_{lm}$.}
\label{fig:lm}
\end{figure}
\subsection{Hopf link}
Let $b = \sigma_2^2 \in B_4$, then its plat closure $\widehat{b}$ is the Hopf link.  
The Hopf link is a two component link, and $\mathcal{Q}$ acts separately to each component.  
Let $\mathcal{Q}^{(1)}$, $\mathcal{Q}^{(2)}$ be algebras isomorphic to $\mathcal{Q}$.
Then $\mathcal{Q}^{(1)} \otimes \mathcal{Q}^{(2)}$ acts to $\widetilde{\mathcal{A}}_b^{\trace}$.  
\par
Figure \ref{fig:hopfalpha} implies that 
$S^{-1}(\alpha_1)S(\alpha_2) = S(\alpha_2)S(\alpha_1)$, so
\begin{align*}
S^{-1}(\alpha_1)S(\alpha_2) - S(\alpha_2)S(\alpha_1)
&=
-t^{-2} x_1 S(\alpha_2) - t^{-4} \alpha_1 S(\alpha_2) - S(\alpha_2)S(\alpha_1)
\\
&=
\alpha_{12} + t^4\alpha_{12}  +t^2  \alpha_2 x_1 + (\alpha_1 x_2)/t^2 + x_1 x_2 + t^2 y
\\&
=
\alpha_{12} + t^4\alpha_{12}  +  t^2 y = 0\quad(y = x_{12})
\end{align*}
since $\alpha_1 x_2 = x_1 x_2/(-t^2-t^{-2})$ and $\alpha_2 x_1 = x_1 x_2/(-t^2-t^{-2})$.  
\begin{figure}[htb]
\[
\begin{matrix}
\includegraphics[scale=0.7]{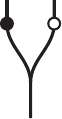}
\end{matrix}
=
\begin{matrix}
\includegraphics[scale=0.7]{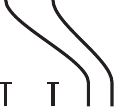}
\\
\circ
\\
\includegraphics[scale=0.8]{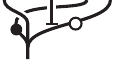}
\end{matrix}
\longleftrightarrow
\begin{matrix}
\includegraphics[scale=0.8]{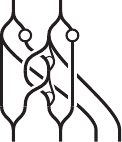}
\\[-3pt]
\circ
\\
\includegraphics[scale=0.8]{h11}
\end{matrix}
=
\begin{matrix}
\includegraphics[scale=0.7]{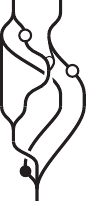}
\end{matrix}
=
\begin{matrix}
\includegraphics[scale=0.7]{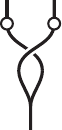}
\end{matrix}
\]
\caption{Relation of of $\widetilde{\mathcal{A}}_{b,1}^{\trace}$ to reduce $\alpha_{12}$ to $y$.}
\label{fig:hopfalpha}
\end{figure}
Hence we get 
\[
\alpha_{12} = -(t^2 + t^{-2})^{-1} y.
\]  
Therefore, $\widetilde{\mathcal{A}}_{b, 1}$ is isomorphic to $\widetilde{\mathcal{A}}_{b, 1}^{\trace}$.   
Moreover, Figure \ref{fig:hy2} implies that 
\[
y^2 = 
t^{-8} + 2 t^{-4}+1 - t^{-4} x_1^2 - t^{-4}x_2^2 - t^{-2}x_1 x_2 y.  
\]
\begin{figure}[htb]
\[
-t^2-t^{-2} = 
\begin{matrix}
\includegraphics[scale=0.7]{h0}
\\
\circ
\\
\includegraphics[scale=0.8]{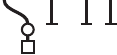}
\end{matrix}
\longleftrightarrow
\begin{matrix}
\includegraphics[scale=0.8]{h1}
\\[-3pt]
\circ
\\
\includegraphics[scale=0.8]{trefoil51}
\end{matrix}
\ =\ 
\begin{matrix}
\includegraphics[scale=0.7]{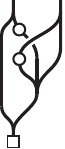}
\end{matrix} 
\ 
\begin{tabular}{l}
\\[5pt]
$=\trace(\alpha_1\alpha_2^{-1}\alpha_1^{-1}\alpha_2)$
\\[5pt]
$=
t^2 + t^6 - t^2 x_1^2 - t^2 x_2^2 - t^4 x_1 x_2 y - t^6 y^2$.
\end{tabular}
\]
\caption{Relation of $\widetilde{\mathcal{A}}_{b,1}^{\trace}$ to reduce $y^2$.}
\label{fig:hy2}
\end{figure}
Similar computations imply that $y^k$ $(k > 1)$ is expressed as $p_0(x_1, x_2) + p_1(x_1, x_2)\, y$, and $\widetilde{\mathcal{A}}_{b, 1}$ is generated by $1$, $y$ as $\mathcal{R}[x_1, x_2]$ module.  
Let $x_l^{(i)}$, $x_m{(i)}$ be elements of $\mathcal{Q}^{(i)}$ correponding to $x_l$, $x_m$ in $\mathcal{Q}$.
Then  $x_l^{(1)} \cdot 1 = x_2$, $x_m^{(1)} \cdot 1 = x_1$, $x_l^{(2)} \cdot 1 = x_1$ and $x_m^{(2)} \cdot 1 = x_2$ as in Figure \ref{fig:haction}.  
\begin{figure}[htb]
\[
\begin{matrix}
\includegraphics[scale=0.6]{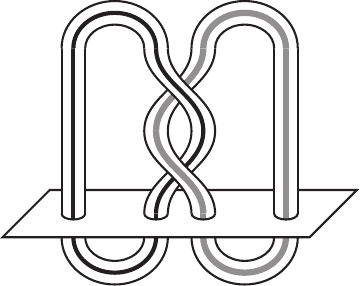}
\end{matrix}
\qquad
\begin{matrix}
\includegraphics[scale=1]{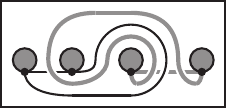}
\end{matrix}
\]
\caption{The action of the longitude $x_l^{(1)}$, $x_l^{(2)}$.}
\label{fig:haction}
\end{figure}
Then $\widetilde{\mathcal{A}}_{b,1}^{\trace}$ is generated by $1$ and $y$ as  $\mathcal{Q}^{(1)} \otimes \mathcal{Q}^{(2)}$ module.  
\subsection{Trefoil}
The trefoil knot is isotopic to $\widehat{b}$ where $b = \sigma_2^3 \in B_4$.  
Here we consider the module $\widetilde{\mathcal{A}}_b^{\trace'}=(S^{-1} \otimes id)(\widetilde{\mathcal{A}}_b^{\trace})$.  
By the relation in Figure \ref{fig:trefoil}, we have
\[ 
\trace\big(\alpha_2^{-1}\alpha_1^{-1}\alpha_2 \alpha_1
-
\alpha_1 \alpha_2\big)
=
-t^{-6}(y^2+(t^2 x^2 + t^6) y+2 t^4 x^2-t^4-1)
\in \widetilde{J}_b 
\]
\begin{figure}[htb]
\begin{multline*}
\begin{matrix}
\includegraphics[scale=0.8]{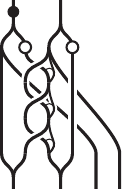}
\\[-5pt]
\circ
\\
\includegraphics[scale=0.8]{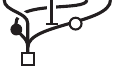}
\end{matrix}
\ = \ 
\begin{matrix}
\includegraphics[scale=0.8]{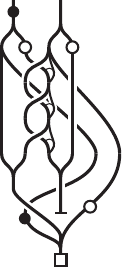}
\end{matrix}
\ = \
\begin{matrix}
\includegraphics[scale=0.8]{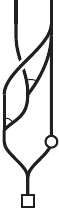}
\end{matrix}
\ = \
\begin{matrix}
\includegraphics[scale=0.8]{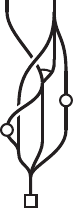}
\end{matrix}
\ = \
\begin{matrix}
\includegraphics[scale=0.8]{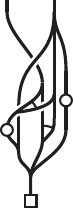}
\end{matrix}
\ \underset{\text({bc)}}{=} \
\begin{matrix}
\includegraphics[scale=0.8]{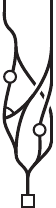}
\end{matrix}
\ = \
\begin{matrix}
\includegraphics[scale=0.8]{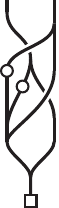}
\end{matrix}
\\
\ = \ 
\trace(\alpha_2^{-1}\alpha_1^{-1}\alpha_2\alpha_1)
\ = -t^{-6}y^2-t^{-4} x^2 y-2 t^{-2} x^2 + t^{-2}+t^{-6},
\end{multline*}
\[
\begin{matrix}
\includegraphics[scale=0.8]{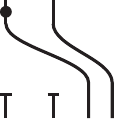}
\\[-5pt]
\circ
\\
\includegraphics[scale=0.8]{trefoil11}
\end{matrix}
\ = \ 
\begin{matrix}
\includegraphics[scale=0.8]{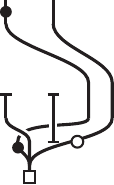}
\end{matrix}
\ = \ 
\begin{matrix}
\includegraphics[scale=0.8]{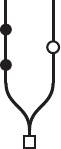}
\end{matrix}
\ = \ 
\begin{matrix}
\includegraphics[scale=0.8]{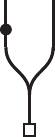}
\end{matrix}
= 
\begin{matrix}
\includegraphics[scale=0.8]{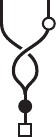}
\end{matrix}
=
\begin{matrix}
\includegraphics[scale=0.8]{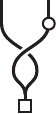}
\end{matrix}
 = 
\trace(\alpha_2^{-1}\alpha_1) = 
-t^{-4}y - t^{-2} x^2.
\]
\caption{Relation of $\widetilde{\mathcal{A}}_b^{\trace'}$ to reduce $y^2$. }
\label{fig:trefoil}
\end{figure}
where $x$, $y$ are images of $x_1$, 
$x_{12}$ in $\widetilde{\mathcal{S}}_{2,1}^{\trace}$,  $\alpha_j^{-1}$ is $\alpha_j\circ S$ since $\alpha_j (\alpha_j\circ S) = (\alpha_j\circ S)\alpha_j = id$.  
This relation implies that
\begin{equation}
y^2= 
t^2(1  -  x^2) y  - t^4 x^2+ t^4+1.
\label{eq:y2}
\end{equation}
in $\widetilde{\mathcal{A}}_b^{\trace}$.  
\par
By the relation in Figure \ref{fig:trefoil3}, we get
\begin{figure}[htb]
\begin{multline*}
\begin{matrix}
\includegraphics[scale=0.7]{trefoil1}
\\[-5pt]
\circ
\\
\ \ \includegraphics[scale=0.7]{trefoil51}
\end{matrix}
\!\!= 
\begin{matrix}
\includegraphics[scale=0.75]{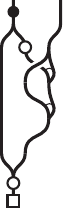}
\end{matrix}
\ =\  
\begin{matrix}
\includegraphics[scale=0.8]{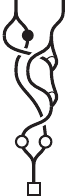}
\end{matrix}
\ = \ 
\begin{matrix}
\includegraphics[scale=0.8]{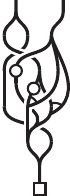}
\end{matrix}
\ \underset{\text{(bc)}}{=} \ 
\begin{matrix}
\includegraphics[scale=0.8]{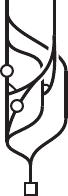}
\end{matrix}
\ = \ 
\begin{matrix}
\includegraphics[scale=0.8]{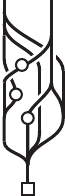}
\end{matrix}
\\
\ = \ 
\trace(\alpha_2^{-1}\alpha_1^{-1}\alpha_2^{-1}\alpha_1\alpha_2\alpha_1)
=
-t^{-8}
\big(y^3 + t^2 x^2 y^2 + (2 t^4 x^2 - t^4 -2)y + (2 t^6-t^2) x^2\big)
, 
\end{multline*}
\[
\begin{matrix}
\includegraphics[scale=0.8]{trefoil21}
\\[-5pt]
\circ
\\
\ \ \includegraphics[scale=0.8]{trefoil51}
\end{matrix}
= -t^2-t^{-2}. 
\]
\caption{Relation of $\widetilde{\mathcal{A}}_b^{\trace'}$ to reduce $y^{3}$.  
 }
\label{fig:trefoil3}
\end{figure}
\[
\big(y^3 + t^2 x^2 y^2 + (2 t^4 x^2 - t^4 -2)y + (2 t^6-t^2) x^2\big)
- t^{10} - t^6 
   \in (S^{-1} \otimes id)(\widetilde{J}_b). 
\]
Combining this and \eqref{eq:y2}, we get
\begin{equation}
y^3 = (2 + t^4 - 3 t^4 x^2  + t^4 x^4) y
+t^6 + t^{10} - 3 t^6 x^2 + 
 t^6 x^4
 \label{eq:y3}
\end{equation}
in $\widetilde{A}_b^{\trace'}$. 
Similarly, $y^k$ $(k > 3)$ can also reduced and 
$\widetilde{A}_b^{\trace'}$ is generated by $1$ and $y$ as a $\mathcal{R}[x]$ module.  
\subsection{Action of the longitude of the trefoil}
Let $L$ be the longitude of the trefoil.  
Then $L$ is a skein element of the thickened torus acting the boundary of the complement of the trefoil, and $L$ corresponds to $x_l$ in $\mathcal{Q}$.  
To get the $A_q$ polynomial of the trefoil, we first compute the action of $L$ to $x^n$.  
Let $L_n$ and $L'_n$ be the diagrams as in Figure \ref{fig:L1}.  
\begin{figure}[htb]
\[
L_n : 
\begin{matrix}
\includegraphics[scale=0.8]{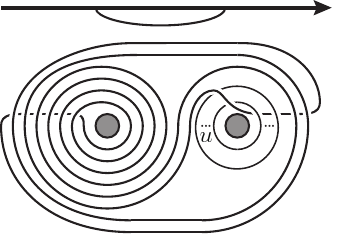}
\end{matrix}\ ,
\qquad
L'_n :
\begin{matrix}
\includegraphics[scale=0.8]{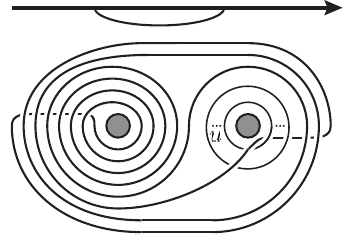}
\end{matrix}\ .  
\]
\caption{The diagrams $L_n = L \cdot x^n= x_l\cdot x^n$ and $L'_n= (-t^3)x_{lm}\cdot x^n$.}
\label{fig:L1}
\end{figure}
Here $L_n = L \cdot x^n$ is the result of the action of $L$ to $x^n$. 
The skein diagrams $L_n$ and $L'_n$ satisfy the following recursive relations.  
\begin{equation}
\begin{cases}
L_n = t^{2}\, x \, L_{n-1} + t^{-2}\, (t^2 - t^{-2}) \, L'_{n-1},
\\
L'_n = -t^{2}\, (t^2-t^{-2}) \, L_{n-1} + t^{-2}\, x \, L'_{n-1}.
\end{cases}
\label{eq:recursion}
\end{equation}
Then we can get $L_n$ and $L'_n$ from $L_0 = L$ and $L'_0$.  
Note that $L_0'$ is equal to $L_b'$ which is given in Figure \ref{fig:longitude} (g).  
The skein diagram $L$ is reduced as in Figures \ref{fig:L2}.  
The skein diagram $L'$ is deformed as in Figure \ref{fig:L3} by moving a twist around the left hole to right hole.  
These diagrams are computed as in these figures and we get the followoing.
\[
\begin{cases}
L = -t^2 - t^6 + 6 t^6 x^2 - 5 t^6 x^4 + 
 t^6 x^6 + (-1 + t^4 - 3 t^4 x^2 + t^4 x^4) y,
\\
L'_0 =  -t^6 x (3 t^2 - 4 t^2 x^2 + t^2 x^4) - t^6 x (-2 + x^2) y.
\end{cases}
\]
\begin{figure}[htb]
\[
\begin{matrix}
\includegraphics[scale=0.7]{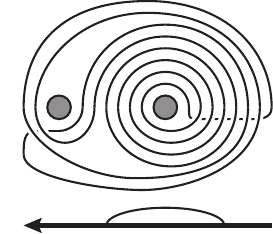}
\end{matrix}
\ = \ 
t^{-2}\, 
\begin{matrix}
\includegraphics[scale=0.7]{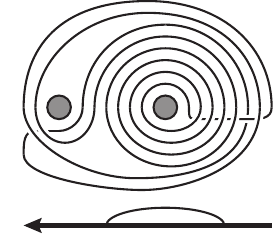}
\end{matrix}
+
t^{-1} \, (t^{2} - t^{-2})\, 
\begin{matrix}
\includegraphics[scale=0.7]{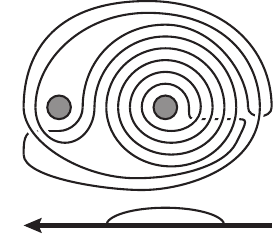}
\end{matrix}
\]
\begin{align*}
&= 
t^{-2} \, (-t^{3})^{2} \,\trace(\alpha_2^{-1}\alpha_1^{-1}\alpha_2^4\alpha_1^{-1}\alpha_2^{-1})
+
t^{-1} \, (-t^{3})^5\, (t^{2} - t^{-2})\, 
\trace(\alpha_2^{-1}\alpha_1^{-1}\alpha_2^5\alpha_1)
\\
&= 
-t^2 - t^6 + 6 t^6 x^2 - 5 t^6 x^4 + 
 t^6 x^6 + (-1 + t^4 - 3 t^4 x^2 + t^4 x^4) y.
 \end{align*}
\caption{Longitude $L$ of the trefoil in $\widetilde{A}^{\trace'}_b$.}
\label{fig:L2}
\end{figure}
\begin{figure}[htb]
\begin{align*}
&\begin{matrix}
\includegraphics[scale=0.7]{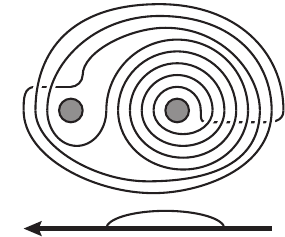}
\end{matrix}
\ = \ 
t^{-2}\begin{matrix}
\includegraphics[scale=0.7]{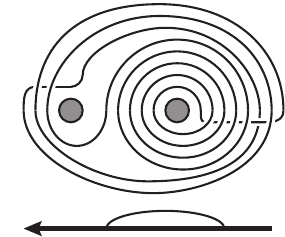}
\end{matrix}
\ + \ 
t^{-1}(t^2-t^{-2})
\begin{matrix}
\includegraphics[scale=0.7]{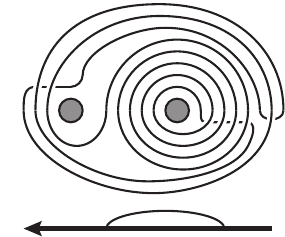}
\end{matrix}
\\
&= 
t^{-2} \, (-t^{3})^{2} \,\trace(\alpha_2^{-1}\alpha_1^{-1}\alpha_2^4\alpha_1^{-1}\alpha_2^{-2}\alpha_1^{-1})
+
t^{-1} \, (-t^{3})^5\, (t^{2} - t^{-2})\, 
\trace(\alpha_2^{-1}\alpha_1^{-1}\alpha_2^5)
\\
    &=
 -t^6 x (3 t^2 - 4 t^2 x^2 + t^2 x^4) - t^6 x (-2 + x^2) y .
\end{align*}
\caption{The skein diagram $L'_0$ in $\widetilde{A}^{\trace'}_b$.}
\label{fig:L3}
\end{figure}
\par
Next, we compute $L \cdot x^n y$.  
Let $\Lambda_n = L \cdot x^n y$ and $\Lambda'_n$ be the skein diagrams given in Figure \ref{fig:lambda}.  
\begin{figure}[htb]
\[
\Lambda_n : 
\begin{matrix}
\includegraphics[scale=0.8]{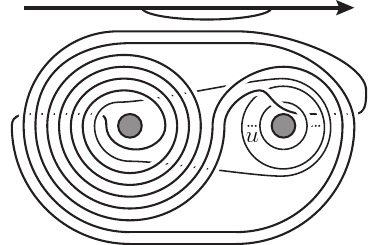}
\end{matrix}\ ,
\qquad
\Lambda'_n :
\begin{matrix}
\includegraphics[scale=0.8]{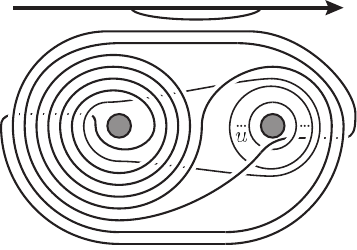}
\end{matrix}\ .  
\]
\caption{The skein diagrams $\Lambda_n$ and $\Lambda'_n$.}
\label{fig:lambda}
\end{figure}
Then  $\Lambda_n$ and $\Lambda'_n$ satisfy the following recursive relation.  
\begin{equation}
\begin{cases}
\Lambda_n = t^{2}\, x \, \Lambda_{n-1} + t^{-2}\, (t^2 - t^{-2}) \, \Lambda'_n,
\\
\Lambda'_n = -t^{-2}\, (t^2-t^{-2}) \, \Lambda_{n-1} + t^2\, x \, \Lambda'_{n-1}.
\end{cases}
\label{eq:Lambdarecursion}
\end{equation}
Hence $\Lambda_n$ and $\Lambda'_n$ can be obtained from $\Lambda_0$ and $\Lambda'_0$.   The skein diagram  $\Lambda_0$  is reduced as in Figures \ref{fig:lambda1}.  
The skein diagram $\Lambda_0'$ is obtained by adding $\alpha_1^{-1}$ at the end of the each term of the second last formula in Figure \ref{fig:lambda1}, and this operation is equivalent to remove $\alpha_2$ which presents the twist around the knot. 
Therefore
\begin{multline*}
\Lambda_0' = 
-t^{10} \trace(\alpha_2^{-1}\alpha_1^{-1}\alpha_2^3)
+t^{16}(t^2-t^{-2})
\trace(\alpha_2^{-1}\alpha_1^{-1}\alpha_2^4\alpha_1\alpha_2\alpha_1)
+ 
\\
-t^{-2} 
\trace(\alpha_2^{-1}\alpha_1^{-1}\alpha_2^3\alpha_1^{-1}\alpha_2^{-1}\alpha_1^{-1}\alpha_2^{-1})
\ + 
t^{8}(t^2-t^{-2})
\trace(\alpha_2^{-1}\alpha_1^{-1}\alpha_2^3)
\end{multline*}
and
\[
\begin{cases}
\Lambda_0 =
-\frac{1}{t^2}(t^2 - t^{10} + 6 t^{10} x^2 - 5 t^{10} x^4 + t^{10} x^6)
  - \frac{1}{t^2}(1 + t^8 - 3 t^8 x^2 + t^8 x^4) y,
\\[5pt]
\Lambda'_0 = x (t^2 + 3 t^{10} - 4 t^{10} x^2 + t^{10} x^4) + x (1 - 2 t^8 + t^8 x^2) y.
\end{cases}
\]
\begin{figure}[htb]
\begin{multline*}
\begin{matrix}
\includegraphics[scale=0.6]{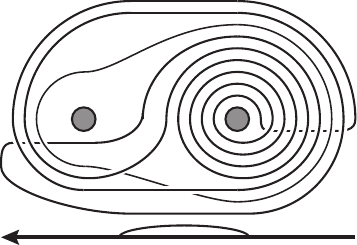}
\end{matrix}
\ = \ 
t\, 
\begin{matrix}
\includegraphics[scale=0.6]{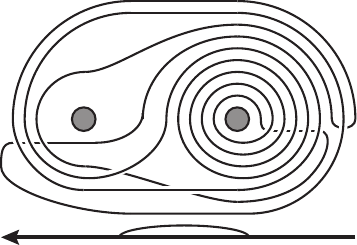}
\end{matrix}
\ + \ 
t^{-1}
\begin{matrix}
\includegraphics[scale=0.6]{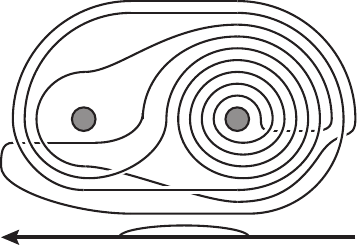}
\end{matrix}
\\
\ = \ 
t (-t^3)^3\trace(\alpha_2^{-1}\alpha_1^{-1}\alpha_2^4)
\ +
t^{-2}(t^2-t^{-2})
\begin{matrix}
\includegraphics[scale=0.6]{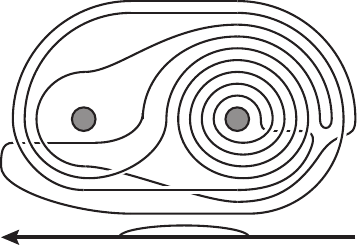}
\end{matrix}
\ + \ 
t^{-3}
\begin{matrix}
\includegraphics[scale=0.6]{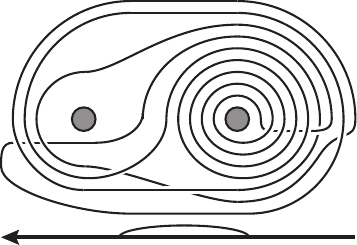}
\end{matrix}
\\
 = 
-t^{10}\trace(\alpha_2^{-1}\alpha_1^{-1}\alpha_2^4)
+t^{-2}(t^2-t^{-2})(-t^3)^6
\trace(\alpha_2^{-1}\alpha_1^{-1}\alpha_2^5\alpha_1\alpha_2\alpha_1)
+ \qquad\qquad\qquad\qquad
\\
t^{-5} 
\begin{matrix}
\includegraphics[scale=0.6]{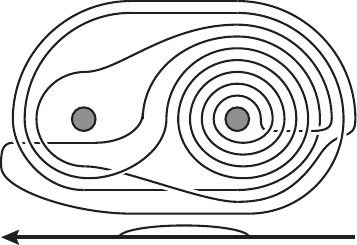}
\end{matrix}
\ + 
t^{-4}(t^2-t^{-2})
\begin{matrix}
\includegraphics[scale=0.6]{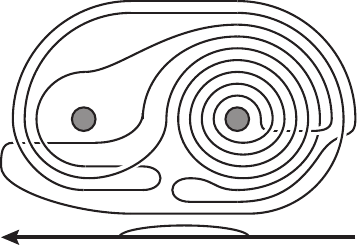}
\end{matrix}
\\
 = 
-t^{10} \trace(\alpha_2^{-1}\alpha_1^{-1}\alpha_2^4)
+t^{16}(t^2-t^{-2})
\trace(\alpha_2^{-1}\alpha_1^{-1}\alpha_2^5\alpha_1\alpha_2\alpha_1)
+ \qquad\qquad\qquad\qquad\qquad
\\
t^{-5} (-t^3)
\trace(\alpha_2^{-1}\alpha_1^{-1}\alpha_2^4\alpha_1^{-1}\alpha_2^{-1}\alpha_1^{-1}\alpha_2^{-1})
\ + 
t^{-4}(t^2-t^{-2})(-t^3)^4
\trace(\alpha_2^{-1}\alpha_1^{-1}\alpha_2^4)
\\
=
-\frac{1}{t^2}(t^2 - t^{10} + 6 t^{10} x^2 - 5 t^{10} x^4 + t^{10} x^6)
  - \frac{1}{t^2}(1 + t^8 - 3 t^8 x^2 + t^8 x^4) y.
\end{multline*}
\caption{Reduction of $\Lambda_0$.}
\label{fig:lambda1}
\end{figure}
\begin{figure}[htb]
\begin{align*}
&\begin{matrix}
\includegraphics[scale=0.7]{lambdadn}
\end{matrix}
= 
\begin{matrix}
\includegraphics[scale=0.7]{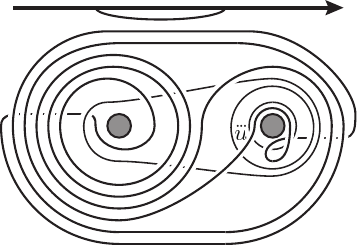}
\end{matrix}
= 
-t^3 
\begin{matrix}
\includegraphics[scale=0.7]{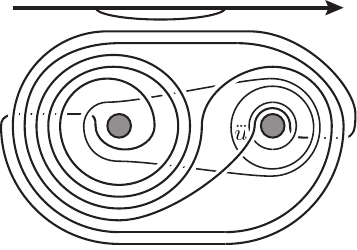}
\end{matrix}
\\
&=
t^{-5} \, (-t^3) \, 
\trace(\alpha_2^{-1}\alpha_1^4\alpha_2^{-1}\alpha_1^{-2}\alpha_2^{-1}\alpha_1^{-1})
+
t^{-4}\, (t^2-t^{-2})\,(-t^3)^4 \,
\trace(\alpha_2^{-1}\alpha_1^5\alpha_2\alpha_1^{-1}\alpha_2^{-1}\alpha_1^{-1})
\\
&\qquad\qquad+
t^{-2}\,(t^2-t^{-2})\,(-t^3)^6\,
\trace(\alpha_2^{-1}\alpha_1^4)
+ t \, (-t^3)^3 \, 
\trace(\alpha_2^{-1}\alpha_1^3)
\\
&=
x (-t^2 + t^6 + 4 t^{10} + 2 t^2 x^2 - t^6 x^2 - 5 t^{10} x^2 - t^2 x^4 + 
    t^6 x^4 + t^{10} x^4) 
\\& \qquad\qquad
+ x (1 - 2 t^8 - x^2 + t^4 x^2 + t^8 x^2) y.
\end{align*}
\caption{Reduction of $\Lambda_0'$.}
\label{fig:lambdad0}
\end{figure}
\par\noindent
%
%
%
\subsection{$A_q$ polynomial}
Here we first obtain a relation between $L$ and $L_0'$ whose coefficients are polynomials of $x$ by eliminating $y$ from $L\cdot 1$ and $L'_0 \cdot 1$.  
The resulting relation is 
\[
L \cdot\left(t^{16} x^3-2 t^{16}
   x\right)+
  L'_0\cdot \left(t^{12} x^4-3 t^{12}
   x^2+t^{12}-1\right)+
   t^{12} x-t^8 x^5+5 t^8 x^3-5 t^8 x=0.  
   \]
 Next, substitute 
 $L = \lambda+\lambda^{-1}$, 
 $L'_0 = (-t^3) t (\lambda\mu^{-1} + \mu\lambda^{-1})$ 
 and 
 $x = -\mu-\mu^{-1}$, we get a polynomial of $\lambda$ with coefficients of polynomials of $\mu$ as follows.  
 \begin{equation}
  \lambda\cdot (\mu^{-1} -t^{12} \mu^{-5})+
  (  t^4\, \mu^{5}- t^8\, \mu  -  t^8\, \mu^{-1}+t^4\, \mu^{-5} ) 
   +\lambda^{-1}\cdot( \mu - t^{12}\, \mu^{5} )=0.  
   \label{eq:Aq1}
 \end{equation}
 As explained in \S\S\ref{ss:recurrence}, this polynomial gives the recurrence relation of the colored Jones polynomial.  
Actually, this relation coincides with the $A_q$ polynomial of the trefoil knot given in \cite{Le}, which is
 \[
 (t^4M^{10} - M^6)L^2 - (t^2M^{10} + t^{-18}- t^{-10}M^6 - t^{-14}M^4)L + (t^{-16} - t^{-4}M^4).
 \]
Our formula is obtained by substituting 
$M = \mu$, 
 $L = -t^6 \, \lambda \, \mu^{-6}$ and multiply by $t^{16}\, \mu$.  
Since  $L$ in \cite{Le} correspond to the standard longitude with three twist around the knot, $L$ and $\lambda$ are different.  
\section{Final remarks}
We  constructed the quantization of the knot group $G$, and also constructed a $q$-deformation of the $SL(2, \mathbb{C})$ representation of $G$ by using Kauffman bracket skein algebra.  
Such skein theory is also constructed for other linear algebraic groups, and $q$-deformations of the representations to such groups can be constructed by imposing the corresponding skein theory to the module $\mathcal{A}_b$.  
\par
In \cite{F}, Faitg investigate the action of mapping class group of a surface by using algebras including the stated skein algebra.  
His construction is a kind of dual of our method here.  
We only consider puncture disks and investigate the action of the braid group by using bottom tangles.  
This construction may be generalized to  the mapping class groups of higher genus surfaces.  
\par
It is proved that the skein module $\mathcal{S}(M)$ of a closed 3-manifold $M$ is finite dimensional $\mathbb{C}(t)$ module in \cite{GJS}.  
This fact may be also shown along with our setting by replacing the
Heegaard splitting in \cite{GJS} by the torus decomposition corresponding to the Dehn surgery.  
\par
Here we obtained the skein module by using the knot diagram. 
Such construction through the diagram may be useful to analyze the geometric meaning of the skein algebra as the case of the volume conjecture.  
The study of the volume conjecture clarifies the geometric meaning of the quantum $R$-matrix, more precisely,  the quantum $R$-matrix has the local geometric information around the crossing.  
\end{document}